\documentclass[pdflatex,sn-mathphys-num]{sn-jnl}
\usepackage{graphicx}%
\usepackage{amsmath,amssymb,amsfonts}%
\usepackage{amsthm}%
\usepackage{mathrsfs}%
\usepackage{stmaryrd}
\usepackage{prftree}
\usepackage{comment}
\usepackage{url}
\usepackage{xcolor}
\definecolor{coqgrayblue}{RGB}{102,123,142}
\ExplSyntaxOn
\prop_new:N \g_hyperlogic_coq_urls_prop
\prop_gput:Nnn \g_hyperlogic_coq_urls_prop {AdjointEpsilonFullFibration} {https://github.com/mozow01/HyperLogic/blob/8d86e202886016097db383f932bc00a7f22c1cc0/theories/Semantics/EpsilonFull/EpsilonFull_Fibration.v\c_hash_str L112}
\prop_gput:Nnn \g_hyperlogic_coq_urls_prop {AllPredicateEpsilonFullModel} {https://github.com/mozow01/HyperLogic/blob/8d86e202886016097db383f932bc00a7f22c1cc0/theories/Semantics/EpsilonFull/EpsilonFull_AllPredicates.v\c_hash_str L152}
\prop_gput:Nnn \g_hyperlogic_coq_urls_prop {BaseEpsilonFullModel} {https://github.com/mozow01/HyperLogic/blob/8d86e202886016097db383f932bc00a7f22c1cc0/theories/Semantics/EpsilonFull/EpsilonFull_Fibration.v\c_hash_str L27}
\prop_gput:Nnn \g_hyperlogic_coq_urls_prop {Base_EpsilonFull_Completeness.Base_CartesianCategory} {https://github.com/mozow01/HyperLogic/blob/8d86e202886016097db383f932bc00a7f22c1cc0/theories/Semantics/EpsilonFull/Base_EpsilonFull_Completeness.v\c_hash_str L46}
\prop_gput:Nnn \g_hyperlogic_coq_urls_prop {Base_EpsilonFull_Completeness.base_completeness_equational} {https://github.com/mozow01/HyperLogic/blob/8d86e202886016097db383f932bc00a7f22c1cc0/theories/Semantics/EpsilonFull/Base_EpsilonFull_Completeness.v\c_hash_str L217}
\prop_gput:Nnn \g_hyperlogic_coq_urls_prop {Base_EpsilonFull_Completeness.hom_base} {https://github.com/mozow01/HyperLogic/blob/8d86e202886016097db383f932bc00a7f22c1cc0/theories/Semantics/EpsilonFull/Base_EpsilonFull_Completeness.v\c_hash_str L14}
\prop_gput:Nnn \g_hyperlogic_coq_urls_prop {Base_EpsilonFull_Completeness.interp_tm_term} {https://github.com/mozow01/HyperLogic/blob/8d86e202886016097db383f932bc00a7f22c1cc0/theories/Semantics/EpsilonFull/Base_EpsilonFull_Completeness.v\c_hash_str L110}
\prop_gput:Nnn \g_hyperlogic_coq_urls_prop {Base_EpsilonFull_Soundness.base_soundness_equational} {https://github.com/mozow01/HyperLogic/blob/8d86e202886016097db383f932bc00a7f22c1cc0/theories/Semantics/EpsilonFull/Base_EpsilonFull_Soundness.v\c_hash_str L25}
\prop_gput:Nnn \g_hyperlogic_coq_urls_prop {Base_SyntacticEpsilonFullModel} {https://github.com/mozow01/HyperLogic/blob/8d86e202886016097db383f932bc00a7f22c1cc0/theories/Semantics/EpsilonFull/Base_EpsilonFull_Completeness.v\c_hash_str L83}
\prop_gput:Nnn \g_hyperlogic_coq_urls_prop {BiCCCFunctor_equiv} {https://github.com/mozow01/HyperLogic/blob/8d86e202886016097db383f932bc00a7f22c1cc0/theories/Cat/Fibration.v\c_hash_str L159}
\prop_gput:Nnn \g_hyperlogic_coq_urls_prop {CanonicalEpsilon.EpsilonLabel} {https://github.com/mozow01/HyperLogic/blob/8d86e202886016097db383f932bc00a7f22c1cc0/theories/Logic/Hyper/Epsilon/Canonical_Skolem.v\c_hash_str L16}
\prop_gput:Nnn \g_hyperlogic_coq_urls_prop {CanonicalEpsilonComparison.EpsilonFullSign} {https://github.com/mozow01/HyperLogic/blob/8d86e202886016097db383f932bc00a7f22c1cc0/theories/Logic/Metatheory/Canonical_Epsilon_Comparison.v\c_hash_str L23}
\prop_gput:Nnn \g_hyperlogic_coq_urls_prop {CanonicalEpsilonComparison.EpsilonSign} {https://github.com/mozow01/HyperLogic/blob/8d86e202886016097db383f932bc00a7f22c1cc0/theories/Logic/Metatheory/Canonical_Epsilon_Comparison.v\c_hash_str L22}
\prop_gput:Nnn \g_hyperlogic_coq_urls_prop {CanonicalEpsilonFull.EpsilonLabel} {https://github.com/mozow01/HyperLogic/blob/8d86e202886016097db383f932bc00a7f22c1cc0/theories/Logic/Hyper/EpsilonFull/Canonical_Skolem.v\c_hash_str L16}
\prop_gput:Nnn \g_hyperlogic_coq_urls_prop {CanonicalEpsilonFull.dependency_well_founded} {https://github.com/mozow01/HyperLogic/blob/8d86e202886016097db383f932bc00a7f22c1cc0/theories/Logic/Hyper/EpsilonFull/Canonical_Skolem.v\c_hash_str L238}
\prop_gput:Nnn \g_hyperlogic_coq_urls_prop {CanonicalEpsilonFull.epsilon_rank} {https://github.com/mozow01/HyperLogic/blob/8d86e202886016097db383f932bc00a7f22c1cc0/theories/Logic/Hyper/EpsilonFull/Canonical_Skolem.v\c_hash_str L84}
\prop_gput:Nnn \g_hyperlogic_coq_urls_prop {CanonicalEpsilonFull.immediate_dependency} {https://github.com/mozow01/HyperLogic/blob/8d86e202886016097db383f932bc00a7f22c1cc0/theories/Logic/Hyper/EpsilonFull/Canonical_Skolem.v\c_hash_str L176}
\prop_gput:Nnn \g_hyperlogic_coq_urls_prop {CanonicalEpsilonFull.immediate_dependency_rank} {https://github.com/mozow01/HyperLogic/blob/8d86e202886016097db383f932bc00a7f22c1cc0/theories/Logic/Hyper/EpsilonFull/Canonical_Skolem.v\c_hash_str L192}
\prop_gput:Nnn \g_hyperlogic_coq_urls_prop {CanonicalEpsilonFullSemantics.completeness} {https://github.com/mozow01/HyperLogic/blob/8d86e202886016097db383f932bc00a7f22c1cc0/theories/Semantics/EpsilonFull/Canonical_EpsilonFull.v\c_hash_str L53}
\prop_gput:Nnn \g_hyperlogic_coq_urls_prop {CanonicalEpsilonFullSemantics.soundness} {https://github.com/mozow01/HyperLogic/blob/8d86e202886016097db383f932bc00a7f22c1cc0/theories/Semantics/EpsilonFull/Canonical_EpsilonFull.v\c_hash_str L49}
\prop_gput:Nnn \g_hyperlogic_coq_urls_prop {CanonicalEpsilonSemantics.completeness} {https://github.com/mozow01/HyperLogic/blob/8d86e202886016097db383f932bc00a7f22c1cc0/theories/Semantics/Epsilon/Canonical_Epsilon.v\c_hash_str L33}
\prop_gput:Nnn \g_hyperlogic_coq_urls_prop {CanonicalEpsilonSemantics.soundness} {https://github.com/mozow01/HyperLogic/blob/8d86e202886016097db383f932bc00a7f22c1cc0/theories/Semantics/Epsilon/Canonical_Epsilon.v\c_hash_str L29}
\prop_gput:Nnn \g_hyperlogic_coq_urls_prop {DerivationExtraction.extract_concrete} {https://github.com/mozow01/HyperLogic/blob/8d86e202886016097db383f932bc00a7f22c1cc0/theories/Logic/Metatheory/DerivationExtraction.v\c_hash_str L19}
\prop_gput:Nnn \g_hyperlogic_coq_urls_prop {DualDrinker_Syntax.Dual_Drinker} {https://github.com/mozow01/HyperLogic/blob/8d86e202886016097db383f932bc00a7f22c1cc0/theories/Logic/Hyper/EpsilonFull/DualDrinker_Syntax.v\c_hash_str L17}
\prop_gput:Nnn \g_hyperlogic_coq_urls_prop {Epsilon.Hyper_Syntax.FE_Eps_Functorial} {https://github.com/mozow01/HyperLogic/blob/8d86e202886016097db383f932bc00a7f22c1cc0/theories/Logic/Hyper/Epsilon/Hyper_Syntax.v\c_hash_str L139}
\prop_gput:Nnn \g_hyperlogic_coq_urls_prop {Epsilon.Hyper_Syntax.h_derivation} {https://github.com/mozow01/HyperLogic/blob/8d86e202886016097db383f932bc00a7f22c1cc0/theories/Logic/Hyper/Epsilon/Hyper_Syntax.v\c_hash_str L251}
\prop_gput:Nnn \g_hyperlogic_coq_urls_prop {Epsilon.Hyper_Syntax.h_ty} {https://github.com/mozow01/HyperLogic/blob/8d86e202886016097db383f932bc00a7f22c1cc0/theories/Logic/Hyper/Epsilon/Hyper_Syntax.v\c_hash_str L14}
\prop_gput:Nnn \g_hyperlogic_coq_urls_prop {EpsilonAsSigmaFibration} {https://github.com/mozow01/HyperLogic/blob/8d86e202886016097db383f932bc00a7f22c1cc0/theories/Semantics/EpsilonFull/EpsilonFull_Fibration.v\c_hash_str L166}
\prop_gput:Nnn \g_hyperlogic_coq_urls_prop {EpsilonFull.Base_Syntax.BT_Const} {https://github.com/mozow01/HyperLogic/blob/8d86e202886016097db383f932bc00a7f22c1cc0/theories/Logic/Base/EpsilonFull/Base_Syntax.v\c_hash_str L46}
\prop_gput:Nnn \g_hyperlogic_coq_urls_prop {EpsilonFull.Base_Syntax.FunctionSignature} {https://github.com/mozow01/HyperLogic/blob/8d86e202886016097db383f932bc00a7f22c1cc0/theories/Logic/Base/EpsilonFull/Base_Syntax.v\c_hash_str L15}
\prop_gput:Nnn \g_hyperlogic_coq_urls_prop {EpsilonFull.Base_Syntax.b_tm} {https://github.com/mozow01/HyperLogic/blob/8d86e202886016097db383f932bc00a7f22c1cc0/theories/Logic/Base/EpsilonFull/Base_Syntax.v\c_hash_str L20}
\prop_gput:Nnn \g_hyperlogic_coq_urls_prop {EpsilonFull.Base_Syntax.b_ty} {https://github.com/mozow01/HyperLogic/blob/8d86e202886016097db383f932bc00a7f22c1cc0/theories/Logic/Base/EpsilonFull/Base_Syntax.v\c_hash_str L8}
\prop_gput:Nnn \g_hyperlogic_coq_urls_prop {EpsilonFull.Base_Syntax.base_equiv} {https://github.com/mozow01/HyperLogic/blob/8d86e202886016097db383f932bc00a7f22c1cc0/theories/Logic/Base/EpsilonFull/Base_Syntax.v\c_hash_str L58}
\prop_gput:Nnn \g_hyperlogic_coq_urls_prop {EpsilonFull.Base_Syntax.base_typing} {https://github.com/mozow01/HyperLogic/blob/8d86e202886016097db383f932bc00a7f22c1cc0/theories/Logic/Base/EpsilonFull/Base_Syntax.v\c_hash_str L32}
\prop_gput:Nnn \g_hyperlogic_coq_urls_prop {EpsilonFull.DualDrinker_Syntax.FE_Pullback_Ext_Fst_Snd} {https://github.com/mozow01/HyperLogic/blob/8d86e202886016097db383f932bc00a7f22c1cc0/theories/Logic/Hyper/EpsilonFull/DualDrinker_Syntax.v\c_hash_str L138}
\prop_gput:Nnn \g_hyperlogic_coq_urls_prop {EpsilonFull.DualDrinker_Syntax.diagonal_base_equation} {https://github.com/mozow01/HyperLogic/blob/8d86e202886016097db383f932bc00a7f22c1cc0/theories/Logic/Hyper/EpsilonFull/DualDrinker_Syntax.v\c_hash_str L158}
\prop_gput:Nnn \g_hyperlogic_coq_urls_prop {EpsilonFull.Examples.conversion_along_equal_witnesses} {https://github.com/mozow01/HyperLogic/blob/8d86e202886016097db383f932bc00a7f22c1cc0/theories/Logic/Hyper/EpsilonFull/Examples.v\c_hash_str L162}
\prop_gput:Nnn \g_hyperlogic_coq_urls_prop {EpsilonFull.Examples.existence_to_selected_typing} {https://github.com/mozow01/HyperLogic/blob/8d86e202886016097db383f932bc00a7f22c1cc0/theories/Logic/Hyper/EpsilonFull/Examples.v\c_hash_str L233}
\prop_gput:Nnn \g_hyperlogic_coq_urls_prop {EpsilonFull.Examples.existential_prenex_typing} {https://github.com/mozow01/HyperLogic/blob/8d86e202886016097db383f932bc00a7f22c1cc0/theories/Logic/Hyper/EpsilonFull/Examples.v\c_hash_str L283}
\prop_gput:Nnn \g_hyperlogic_coq_urls_prop {EpsilonFull.Examples.free_bound_projection_equation} {https://github.com/mozow01/HyperLogic/blob/8d86e202886016097db383f932bc00a7f22c1cc0/theories/Logic/Hyper/EpsilonFull/Examples.v\c_hash_str L151}
\prop_gput:Nnn \g_hyperlogic_coq_urls_prop {EpsilonFull.Examples.nested_epsilon_term_typing} {https://github.com/mozow01/HyperLogic/blob/8d86e202886016097db383f932bc00a7f22c1cc0/theories/Logic/Hyper/EpsilonFull/Examples.v\c_hash_str L78}
\prop_gput:Nnn \g_hyperlogic_coq_urls_prop {EpsilonFull.Examples.persons_under_binder_wellformed} {https://github.com/mozow01/HyperLogic/blob/8d86e202886016097db383f932bc00a7f22c1cc0/theories/Logic/Hyper/EpsilonFull/Examples.v\c_hash_str L115}
\prop_gput:Nnn \g_hyperlogic_coq_urls_prop {EpsilonFull.Examples.projection_epsilon_beck_chevalley} {https://github.com/mozow01/HyperLogic/blob/8d86e202886016097db383f932bc00a7f22c1cc0/theories/Logic/Hyper/EpsilonFull/Examples.v\c_hash_str L180}
\prop_gput:Nnn \g_hyperlogic_coq_urls_prop {EpsilonFull.Examples.represented_selection_from_existence} {https://github.com/mozow01/HyperLogic/blob/8d86e202886016097db383f932bc00a7f22c1cc0/theories/Logic/Hyper/EpsilonFull/Examples.v\c_hash_str L266}
\prop_gput:Nnn \g_hyperlogic_coq_urls_prop {EpsilonFull.Examples.represented_selection_witness_typing} {https://github.com/mozow01/HyperLogic/blob/8d86e202886016097db383f932bc00a7f22c1cc0/theories/Logic/Hyper/EpsilonFull/Examples.v\c_hash_str L263}
\prop_gput:Nnn \g_hyperlogic_coq_urls_prop {EpsilonFull.Examples.second_label_rank} {https://github.com/mozow01/HyperLogic/blob/8d86e202886016097db383f932bc00a7f22c1cc0/theories/Logic/Hyper/EpsilonFull/Examples.v\c_hash_str L25}
\prop_gput:Nnn \g_hyperlogic_coq_urls_prop {EpsilonFull.Examples.selected_to_existence_typing} {https://github.com/mozow01/HyperLogic/blob/8d86e202886016097db383f932bc00a7f22c1cc0/theories/Logic/Hyper/EpsilonFull/Examples.v\c_hash_str L215}
\prop_gput:Nnn \g_hyperlogic_coq_urls_prop {EpsilonFull.Hyper_Syntax.All_WellFormed} {https://github.com/mozow01/HyperLogic/blob/8d86e202886016097db383f932bc00a7f22c1cc0/theories/Logic/Hyper/EpsilonFull/Hyper_Syntax.v\c_hash_str L131}
\prop_gput:Nnn \g_hyperlogic_coq_urls_prop {EpsilonFull.Hyper_Syntax.FE_Eps_Adjoint_Iso} {https://github.com/mozow01/HyperLogic/blob/8d86e202886016097db383f932bc00a7f22c1cc0/theories/Logic/Hyper/EpsilonFull/Hyper_Syntax.v\c_hash_str L108}
\prop_gput:Nnn \g_hyperlogic_coq_urls_prop {EpsilonFull.Hyper_Syntax.FE_Eps_BC} {https://github.com/mozow01/HyperLogic/blob/8d86e202886016097db383f932bc00a7f22c1cc0/theories/Logic/Hyper/EpsilonFull/Hyper_Syntax.v\c_hash_str L124}
\prop_gput:Nnn \g_hyperlogic_coq_urls_prop {EpsilonFull.Hyper_Syntax.PredicateSignature} {https://github.com/mozow01/HyperLogic/blob/8d86e202886016097db383f932bc00a7f22c1cc0/theories/Logic/Hyper/EpsilonFull/Hyper_Syntax.v\c_hash_str L9}
\prop_gput:Nnn \g_hyperlogic_coq_urls_prop {EpsilonFull.Hyper_Syntax.WellFormed} {https://github.com/mozow01/HyperLogic/blob/8d86e202886016097db383f932bc00a7f22c1cc0/theories/Logic/Hyper/EpsilonFull/Hyper_Syntax.v\c_hash_str L81}
\prop_gput:Nnn \g_hyperlogic_coq_urls_prop {EpsilonFull.Hyper_Syntax.b_subst_ext} {https://github.com/mozow01/HyperLogic/blob/8d86e202886016097db383f932bc00a7f22c1cc0/theories/Logic/Hyper/EpsilonFull/Hyper_Syntax.v\c_hash_str L55}
\prop_gput:Nnn \g_hyperlogic_coq_urls_prop {EpsilonFull.Hyper_Syntax.ctx_equiv} {https://github.com/mozow01/HyperLogic/blob/8d86e202886016097db383f932bc00a7f22c1cc0/theories/Logic/Hyper/EpsilonFull/Hyper_Syntax.v\c_hash_str L137}
\prop_gput:Nnn \g_hyperlogic_coq_urls_prop {EpsilonFull.Hyper_Syntax.ctx_subst} {https://github.com/mozow01/HyperLogic/blob/8d86e202886016097db383f932bc00a7f22c1cc0/theories/Logic/Hyper/EpsilonFull/Hyper_Syntax.v\c_hash_str L71}
\prop_gput:Nnn \g_hyperlogic_coq_urls_prop {EpsilonFull.Hyper_Syntax.form_equiv} {https://github.com/mozow01/HyperLogic/blob/8d86e202886016097db383f932bc00a7f22c1cc0/theories/Logic/Hyper/EpsilonFull/Hyper_Syntax.v\c_hash_str L93}
\prop_gput:Nnn \g_hyperlogic_coq_urls_prop {EpsilonFull.Hyper_Syntax.form_subst} {https://github.com/mozow01/HyperLogic/blob/8d86e202886016097db383f932bc00a7f22c1cc0/theories/Logic/Hyper/EpsilonFull/Hyper_Syntax.v\c_hash_str L58}
\prop_gput:Nnn \g_hyperlogic_coq_urls_prop {EpsilonFull.Hyper_Syntax.h_context} {https://github.com/mozow01/HyperLogic/blob/8d86e202886016097db383f932bc00a7f22c1cc0/theories/Logic/Hyper/EpsilonFull/Hyper_Syntax.v\c_hash_str L24}
\prop_gput:Nnn \g_hyperlogic_coq_urls_prop {EpsilonFull.Hyper_Syntax.h_derivation} {https://github.com/mozow01/HyperLogic/blob/8d86e202886016097db383f932bc00a7f22c1cc0/theories/Logic/Hyper/EpsilonFull/Hyper_Syntax.v\c_hash_str L222}
\prop_gput:Nnn \g_hyperlogic_coq_urls_prop {EpsilonFull.Hyper_Syntax.h_tm} {https://github.com/mozow01/HyperLogic/blob/8d86e202886016097db383f932bc00a7f22c1cc0/theories/Logic/Hyper/EpsilonFull/Hyper_Syntax.v\c_hash_str L35}
\prop_gput:Nnn \g_hyperlogic_coq_urls_prop {EpsilonFull.Hyper_Syntax.h_ty} {https://github.com/mozow01/HyperLogic/blob/8d86e202886016097db383f932bc00a7f22c1cc0/theories/Logic/Hyper/EpsilonFull/Hyper_Syntax.v\c_hash_str L14}
\prop_gput:Nnn \g_hyperlogic_coq_urls_prop {EpsilonFull.Hyper_Syntax.tm_subst} {https://github.com/mozow01/HyperLogic/blob/8d86e202886016097db383f932bc00a7f22c1cc0/theories/Logic/Hyper/EpsilonFull/Hyper_Syntax.v\c_hash_str L302}
\prop_gput:Nnn \g_hyperlogic_coq_urls_prop {EpsilonFull.Hyper_Syntax.tm_var} {https://github.com/mozow01/HyperLogic/blob/8d86e202886016097db383f932bc00a7f22c1cc0/theories/Logic/Hyper/EpsilonFull/Hyper_Syntax.v\c_hash_str L37}
\prop_gput:Nnn \g_hyperlogic_coq_urls_prop {EpsilonFull.Hyper_Syntax_Lemmas.form_equiv_subst} {https://github.com/mozow01/HyperLogic/blob/8d86e202886016097db383f932bc00a7f22c1cc0/theories/Logic/Hyper/EpsilonFull/Hyper_Syntax_Lemmas.v\c_hash_str L137}
\prop_gput:Nnn \g_hyperlogic_coq_urls_prop {EpsilonFull.Hyper_Syntax_Lemmas.form_equiv_wf} {https://github.com/mozow01/HyperLogic/blob/8d86e202886016097db383f932bc00a7f22c1cc0/theories/Logic/Hyper/EpsilonFull/Hyper_Syntax_Lemmas.v\c_hash_str L61}
\prop_gput:Nnn \g_hyperlogic_coq_urls_prop {EpsilonFull.Hyper_Syntax_Lemmas.meta_ctx_wf} {https://github.com/mozow01/HyperLogic/blob/8d86e202886016097db383f932bc00a7f22c1cc0/theories/Logic/Hyper/EpsilonFull/Hyper_Syntax_Lemmas.v\c_hash_str L399}
\prop_gput:Nnn \g_hyperlogic_coq_urls_prop {EpsilonFullModel} {https://github.com/mozow01/HyperLogic/blob/8d86e202886016097db383f932bc00a7f22c1cc0/theories/Semantics/EpsilonFull/EpsilonFull_ModelCategory.v\c_hash_str L33}
\prop_gput:Nnn \g_hyperlogic_coq_urls_prop {EpsilonFullModelHom} {https://github.com/mozow01/HyperLogic/blob/8d86e202886016097db383f932bc00a7f22c1cc0/theories/Semantics/EpsilonFull/EpsilonFull_ModelCategory.v\c_hash_str L99}
\prop_gput:Nnn \g_hyperlogic_coq_urls_prop {EpsilonFullSubstitutivity.DefinitionalSubstitutivity} {https://github.com/mozow01/HyperLogic/blob/8d86e202886016097db383f932bc00a7f22c1cc0/theories/Semantics/Metatheory/EpsilonFull_Substitutivity.v\c_hash_str L29}
\prop_gput:Nnn \g_hyperlogic_coq_urls_prop {EpsilonFullSubstitutivity.convertible_witnesses_not_definitionally_equal} {https://github.com/mozow01/HyperLogic/blob/8d86e202886016097db383f932bc00a7f22c1cc0/theories/Semantics/Metatheory/EpsilonFull_Substitutivity.v\c_hash_str L635}
\prop_gput:Nnn \g_hyperlogic_coq_urls_prop {EpsilonFullSubstitutivity.counterexample_equation_valid_in} {https://github.com/mozow01/HyperLogic/blob/8d86e202886016097db383f932bc00a7f22c1cc0/theories/Semantics/Metatheory/EpsilonFull_Substitutivity.v\c_hash_str L444}
\prop_gput:Nnn \g_hyperlogic_coq_urls_prop {EpsilonFullSubstitutivity.fin_H3_interp_h_ty_subst_const} {https://github.com/mozow01/HyperLogic/blob/8d86e202886016097db383f932bc00a7f22c1cc0/theories/Semantics/Metatheory/EpsilonFull_Substitutivity.v\c_hash_str L357}
\prop_gput:Nnn \g_hyperlogic_coq_urls_prop {EpsilonFullSubstitutivity.separating_model_refutes_substitutivity} {https://github.com/mozow01/HyperLogic/blob/8d86e202886016097db383f932bc00a7f22c1cc0/theories/Semantics/Metatheory/EpsilonFull_Substitutivity.v\c_hash_str L562}
\prop_gput:Nnn \g_hyperlogic_coq_urls_prop {EpsilonFullSubstitutivity.substituted_witness} {https://github.com/mozow01/HyperLogic/blob/8d86e202886016097db383f932bc00a7f22c1cc0/theories/Semantics/Metatheory/EpsilonFull_Substitutivity.v\c_hash_str L118}
\prop_gput:Nnn \g_hyperlogic_coq_urls_prop {EpsilonFullSubstitutivity.substitution_counterexample_formula} {https://github.com/mozow01/HyperLogic/blob/8d86e202886016097db383f932bc00a7f22c1cc0/theories/Semantics/Metatheory/EpsilonFull_Substitutivity.v\c_hash_str L80}
\prop_gput:Nnn \g_hyperlogic_coq_urls_prop {EpsilonFullSubstitutivity.uniform_model_validates_substitutivity} {https://github.com/mozow01/HyperLogic/blob/8d86e202886016097db383f932bc00a7f22c1cc0/theories/Semantics/Metatheory/EpsilonFull_Substitutivity.v\c_hash_str L566}
\prop_gput:Nnn \g_hyperlogic_coq_urls_prop {EpsilonFullSubstitutivity.witness_of_substituted_formula} {https://github.com/mozow01/HyperLogic/blob/8d86e202886016097db383f932bc00a7f22c1cc0/theories/Semantics/Metatheory/EpsilonFull_Substitutivity.v\c_hash_str L115}
\prop_gput:Nnn \g_hyperlogic_coq_urls_prop {EpsilonFull_Completeness.epsilon_adjoint_completeness_inhabitational_gen} {https://github.com/mozow01/HyperLogic/blob/8d86e202886016097db383f932bc00a7f22c1cc0/theories/Semantics/EpsilonFull/EpsilonFull_Completeness.v\c_hash_str L2440}
\prop_gput:Nnn \g_hyperlogic_coq_urls_prop {EpsilonFull_Completeness.fiber_equiv} {https://github.com/mozow01/HyperLogic/blob/8d86e202886016097db383f932bc00a7f22c1cc0/theories/Semantics/EpsilonFull/EpsilonFull_Completeness.v\c_hash_str L43}
\prop_gput:Nnn \g_hyperlogic_coq_urls_prop {EpsilonFull_Completeness.h_derivation_subst_general} {https://github.com/mozow01/HyperLogic/blob/8d86e202886016097db383f932bc00a7f22c1cc0/theories/Semantics/EpsilonFull/EpsilonFull_Completeness.v\c_hash_str L911}
\prop_gput:Nnn \g_hyperlogic_coq_urls_prop {EpsilonFull_Completeness.h_make_ctx_ty} {https://github.com/mozow01/HyperLogic/blob/8d86e202886016097db383f932bc00a7f22c1cc0/theories/Semantics/EpsilonFull/EpsilonFull_Completeness.v\c_hash_str L2291}
\prop_gput:Nnn \g_hyperlogic_coq_urls_prop {EpsilonFull_Completeness.hom_fiber} {https://github.com/mozow01/HyperLogic/blob/8d86e202886016097db383f932bc00a7f22c1cc0/theories/Semantics/EpsilonFull/EpsilonFull_Completeness.v\c_hash_str L32}
\prop_gput:Nnn \g_hyperlogic_coq_urls_prop {EpsilonFull_Completeness.term_model_interp_ctx_id} {https://github.com/mozow01/HyperLogic/blob/8d86e202886016097db383f932bc00a7f22c1cc0/theories/Semantics/EpsilonFull/EpsilonFull_Completeness.v\c_hash_str L2320}
\prop_gput:Nnn \g_hyperlogic_coq_urls_prop {EpsilonFull_Completeness.term_model_interp_id} {https://github.com/mozow01/HyperLogic/blob/8d86e202886016097db383f932bc00a7f22c1cc0/theories/Semantics/EpsilonFull/EpsilonFull_Completeness.v\c_hash_str L1606}
\prop_gput:Nnn \g_hyperlogic_coq_urls_prop {EpsilonFull_Fibration.CoherentBaseModel} {https://github.com/mozow01/HyperLogic/blob/8d86e202886016097db383f932bc00a7f22c1cc0/theories/Semantics/EpsilonFull/EpsilonFull_Fibration.v\c_hash_str L53}
\prop_gput:Nnn \g_hyperlogic_coq_urls_prop {EpsilonFull_Fibration.EpsilonHyperModel} {https://github.com/mozow01/HyperLogic/blob/8d86e202886016097db383f932bc00a7f22c1cc0/theories/Semantics/EpsilonFull/EpsilonFull_Fibration.v\c_hash_str L63}
\prop_gput:Nnn \g_hyperlogic_coq_urls_prop {EpsilonFull_Fibration.i_atom} {https://github.com/mozow01/HyperLogic/blob/8d86e202886016097db383f932bc00a7f22c1cc0/theories/Semantics/EpsilonFull/EpsilonFull_Fibration.v\c_hash_str L69}
\prop_gput:Nnn \g_hyperlogic_coq_urls_prop {EpsilonFull_Fibration.interp_h_ctx} {https://github.com/mozow01/HyperLogic/blob/8d86e202886016097db383f932bc00a7f22c1cc0/theories/Semantics/EpsilonFull/EpsilonFull_Fibration.v\c_hash_str L100}
\prop_gput:Nnn \g_hyperlogic_coq_urls_prop {EpsilonFull_Fibration.interp_h_ty} {https://github.com/mozow01/HyperLogic/blob/8d86e202886016097db383f932bc00a7f22c1cc0/theories/Semantics/EpsilonFull/EpsilonFull_Fibration.v\c_hash_str L81}
\prop_gput:Nnn \g_hyperlogic_coq_urls_prop {EpsilonFull_ModelCategory.TermModel_initial_unique} {https://github.com/mozow01/HyperLogic/blob/8d86e202886016097db383f932bc00a7f22c1cc0/theories/Semantics/EpsilonFull/EpsilonFull_ModelCategory.v\c_hash_str L821}
\prop_gput:Nnn \g_hyperlogic_coq_urls_prop {EpsilonFull_ModelCategory.TermModel_interpretation_identity} {https://github.com/mozow01/HyperLogic/blob/8d86e202886016097db383f932bc00a7f22c1cc0/theories/Semantics/EpsilonFull/EpsilonFull_ModelCategory.v\c_hash_str L845}
\prop_gput:Nnn \g_hyperlogic_coq_urls_prop {EpsilonFull_ModelCategory.TermModel_interpretation_natural} {https://github.com/mozow01/HyperLogic/blob/8d86e202886016097db383f932bc00a7f22c1cc0/theories/Semantics/EpsilonFull/EpsilonFull_ModelCategory.v\c_hash_str L839}
\prop_gput:Nnn \g_hyperlogic_coq_urls_prop {EpsilonFull_ModelCategory.TermModel_weak_initial} {https://github.com/mozow01/HyperLogic/blob/8d86e202886016097db383f932bc00a7f22c1cc0/theories/Semantics/EpsilonFull/EpsilonFull_ModelCategory.v\c_hash_str L775}
\prop_gput:Nnn \g_hyperlogic_coq_urls_prop {EpsilonFull_SelectionModels.EpsilonFullModel_from_chosen_units} {https://github.com/mozow01/HyperLogic/blob/8d86e202886016097db383f932bc00a7f22c1cc0/theories/Semantics/EpsilonFull/EpsilonFull_SelectionModels.v\c_hash_str L112}
\prop_gput:Nnn \g_hyperlogic_coq_urls_prop {EpsilonFull_SelectionModels.selected_instance_interpretation} {https://github.com/mozow01/HyperLogic/blob/8d86e202886016097db383f932bc00a7f22c1cc0/theories/Semantics/EpsilonFull/EpsilonFull_SelectionModels.v\c_hash_str L38}
\prop_gput:Nnn \g_hyperlogic_coq_urls_prop {EpsilonFull_SelectionModels.selected_section_projection} {https://github.com/mozow01/HyperLogic/blob/8d86e202886016097db383f932bc00a7f22c1cc0/theories/Semantics/EpsilonFull/EpsilonFull_SelectionModels.v\c_hash_str L32}
\prop_gput:Nnn \g_hyperlogic_coq_urls_prop {EpsilonFull_Soundness.form_subst_iso} {https://github.com/mozow01/HyperLogic/blob/8d86e202886016097db383f932bc00a7f22c1cc0/theories/Semantics/EpsilonFull/EpsilonFull_Soundness.v\c_hash_str L20}
\prop_gput:Nnn \g_hyperlogic_coq_urls_prop {EpsilonFull_Soundness.interp_ctx_equiv} {https://github.com/mozow01/HyperLogic/blob/8d86e202886016097db383f932bc00a7f22c1cc0/theories/Semantics/EpsilonFull/EpsilonFull_Soundness.v\c_hash_str L289}
\prop_gput:Nnn \g_hyperlogic_coq_urls_prop {EpsilonFull_Soundness.interp_form_equiv} {https://github.com/mozow01/HyperLogic/blob/8d86e202886016097db383f932bc00a7f22c1cc0/theories/Semantics/EpsilonFull/EpsilonFull_Soundness.v\c_hash_str L239}
\prop_gput:Nnn \g_hyperlogic_coq_urls_prop {EpsilonFull_Soundness.interp_subst_comp} {https://github.com/mozow01/HyperLogic/blob/8d86e202886016097db383f932bc00a7f22c1cc0/theories/Semantics/EpsilonFull/EpsilonFull_Soundness.v\c_hash_str L155}
\prop_gput:Nnn \g_hyperlogic_coq_urls_prop {EpsilonFull_Soundness.soundness} {https://github.com/mozow01/HyperLogic/blob/8d86e202886016097db383f932bc00a7f22c1cc0/theories/Semantics/EpsilonFull/EpsilonFull_Soundness.v\c_hash_str L309}
\prop_gput:Nnn \g_hyperlogic_coq_urls_prop {EpsilonGoodFibration} {https://github.com/mozow01/HyperLogic/blob/8d86e202886016097db383f932bc00a7f22c1cc0/theories/Semantics/Epsilon/Epsilon_Fibration.v\c_hash_str L116}
\prop_gput:Nnn \g_hyperlogic_coq_urls_prop {Epsilon_Completeness.TermModel_EpsilonGoodFibration} {https://github.com/mozow01/HyperLogic/blob/8d86e202886016097db383f932bc00a7f22c1cc0/theories/Semantics/Epsilon/Epsilon_Completeness.v\c_hash_str L2207}
\prop_gput:Nnn \g_hyperlogic_coq_urls_prop {Epsilon_SelectionModels.EpsilonGoodFibration_from_chosen_units} {https://github.com/mozow01/HyperLogic/blob/8d86e202886016097db383f932bc00a7f22c1cc0/theories/Semantics/Epsilon/Epsilon_SelectionModels.v\c_hash_str L169}
\prop_gput:Nnn \g_hyperlogic_coq_urls_prop {Epsilon_Soundness.form_subst_iso} {https://github.com/mozow01/HyperLogic/blob/8d86e202886016097db383f932bc00a7f22c1cc0/theories/Semantics/Epsilon/Epsilon_Soundness.v\c_hash_str L51}
\prop_gput:Nnn \g_hyperlogic_coq_urls_prop {FHCanonicalChoice_unfold} {https://github.com/mozow01/HyperLogic/blob/8d86e202886016097db383f932bc00a7f22c1cc0/theories/Semantics/EpsilonFull/Canonical_FiniteChain_Model.v\c_hash_str L472}
\prop_gput:Nnn \g_hyperlogic_coq_urls_prop {GoodFibration} {https://github.com/mozow01/HyperLogic/blob/8d86e202886016097db383f932bc00a7f22c1cc0/theories/Cat/Fibration.v\c_hash_str L13}
\prop_gput:Nnn \g_hyperlogic_coq_urls_prop {H3CanonicalFullModel} {https://github.com/mozow01/HyperLogic/blob/8d86e202886016097db383f932bc00a7f22c1cc0/theories/Semantics/EpsilonFull/FiniteHeytingExamples.v\c_hash_str L159}
\prop_gput:Nnn \g_hyperlogic_coq_urls_prop {H3FiniteHeytingAlgebra} {https://github.com/mozow01/HyperLogic/blob/8d86e202886016097db383f932bc00a7f22c1cc0/theories/Semantics/EpsilonFull/FiniteHeytingExamples.v\c_hash_str L36}
\prop_gput:Nnn \g_hyperlogic_coq_urls_prop {H3_canonical_LEM_underivable} {https://github.com/mozow01/HyperLogic/blob/8d86e202886016097db383f932bc00a7f22c1cc0/theories/Semantics/EpsilonFull/FiniteHeytingExamples.v\c_hash_str L259}
\prop_gput:Nnn \g_hyperlogic_coq_urls_prop {H5FiniteHeytingAlgebra} {https://github.com/mozow01/HyperLogic/blob/8d86e202886016097db383f932bc00a7f22c1cc0/theories/Semantics/EpsilonFull/FiniteHeytingExamples.v\c_hash_str L282}
\prop_gput:Nnn \g_hyperlogic_coq_urls_prop {IFOL.Base_Syntax.b_tm} {https://github.com/mozow01/HyperLogic/blob/8d86e202886016097db383f932bc00a7f22c1cc0/theories/Logic/Base/IFOL/Base_Syntax.v\c_hash_str L23}
\prop_gput:Nnn \g_hyperlogic_coq_urls_prop {IFOL.Hyper_Syntax.h_derivation} {https://github.com/mozow01/HyperLogic/blob/8d86e202886016097db383f932bc00a7f22c1cc0/theories/Logic/Hyper/IFOL/Hyper_Syntax.v\c_hash_str L191}
\prop_gput:Nnn \g_hyperlogic_coq_urls_prop {IFOL.Hyper_Syntax.h_ty} {https://github.com/mozow01/HyperLogic/blob/8d86e202886016097db383f932bc00a7f22c1cc0/theories/Logic/Hyper/IFOL/Hyper_Syntax.v\c_hash_str L14}
\prop_gput:Nnn \g_hyperlogic_coq_urls_prop {IFOLHyperModel} {https://github.com/mozow01/HyperLogic/blob/8d86e202886016097db383f932bc00a7f22c1cc0/theories/Semantics/IFOL/IFOL_Fibration.v\c_hash_str L57}
\prop_gput:Nnn \g_hyperlogic_coq_urls_prop {PasqualiEnoughEpsilonTerms} {https://github.com/mozow01/HyperLogic/blob/8d86e202886016097db383f932bc00a7f22c1cc0/theories/Semantics/Metatheory/Pasquali_Epsilon.v\c_hash_str L131}
\prop_gput:Nnn \g_hyperlogic_coq_urls_prop {Pullback} {https://github.com/mozow01/HyperLogic/blob/8d86e202886016097db383f932bc00a7f22c1cc0/theories/Cat/Fibration.v\c_hash_str L23}
\prop_gput:Nnn \g_hyperlogic_coq_urls_prop {TermModel_AdjointEpsilonFullFibration} {https://github.com/mozow01/HyperLogic/blob/8d86e202886016097db383f932bc00a7f22c1cc0/theories/Semantics/EpsilonFull/EpsilonFull_Completeness.v\c_hash_str L1693}
\prop_gput:Nnn \g_hyperlogic_coq_urls_prop {TermModel_EpsilonFullModel} {https://github.com/mozow01/HyperLogic/blob/8d86e202886016097db383f932bc00a7f22c1cc0/theories/Semantics/EpsilonFull/EpsilonFull_ModelCategory.v\c_hash_str L496}
\prop_gput:Nnn \g_hyperlogic_coq_urls_prop {TermModel_IFOLHyperModel} {https://github.com/mozow01/HyperLogic/blob/8d86e202886016097db383f932bc00a7f22c1cc0/theories/Semantics/IFOL/IFOL_Completeness.v\c_hash_str L1370}
\prop_gput:Nnn \g_hyperlogic_coq_urls_prop {all_predicate_epsilonfull_soundness_completeness} {https://github.com/mozow01/HyperLogic/blob/8d86e202886016097db383f932bc00a7f22c1cc0/theories/Semantics/EpsilonFull/EpsilonFull_AllPredicates.v\c_hash_str L324}
\prop_gput:Nnn \g_hyperlogic_coq_urls_prop {epsilon_completeness_inhabitational_gen} {https://github.com/mozow01/HyperLogic/blob/8d86e202886016097db383f932bc00a7f22c1cc0/theories/Semantics/Epsilon/Epsilon_Completeness.v\c_hash_str L2380}
\prop_gput:Nnn \g_hyperlogic_coq_urls_prop {epsilon_section} {https://github.com/mozow01/HyperLogic/blob/8d86e202886016097db383f932bc00a7f22c1cc0/theories/Cat/EpsilonSelection.v\c_hash_str L16}
\prop_gput:Nnn \g_hyperlogic_coq_urls_prop {epsilon_selection_unit_equivalence} {https://github.com/mozow01/HyperLogic/blob/8d86e202886016097db383f932bc00a7f22c1cc0/theories/Cat/EpsilonSelection.v\c_hash_str L110}
\prop_gput:Nnn \g_hyperlogic_coq_urls_prop {epsilon_selection_unit_iff_inhabited_adjunction} {https://github.com/mozow01/HyperLogic/blob/8d86e202886016097db383f932bc00a7f22c1cc0/theories/Cat/EpsilonSelection.v\c_hash_str L356}
\prop_gput:Nnn \g_hyperlogic_coq_urls_prop {epsilonfull_dual_drinker_provable} {https://github.com/mozow01/HyperLogic/blob/8d86e202886016097db383f932bc00a7f22c1cc0/theories/Semantics/Metatheory/DualDrinker.v\c_hash_str L41}
\prop_gput:Nnn \g_hyperlogic_coq_urls_prop {epsilonfull_not_conservative_over_ifol} {https://github.com/mozow01/HyperLogic/blob/8d86e202886016097db383f932bc00a7f22c1cc0/theories/Semantics/Metatheory/IFOLFull_NonConservativity.v\c_hash_str L87}
\prop_gput:Nnn \g_hyperlogic_coq_urls_prop {ifol_completeness_inhabitational_gen} {https://github.com/mozow01/HyperLogic/blob/8d86e202886016097db383f932bc00a7f22c1cc0/theories/Semantics/IFOL/IFOL_Completeness.v\c_hash_str L1943}
\prop_gput:Nnn \g_hyperlogic_coq_urls_prop {ifol_dual_drinker_refuted} {https://github.com/mozow01/HyperLogic/blob/8d86e202886016097db383f932bc00a7f22c1cc0/theories/Semantics/IFOL/DualDrinker.v\c_hash_str L674}
\prop_gput:Nnn \g_hyperlogic_coq_urls_prop {ifol_soundness} {https://github.com/mozow01/HyperLogic/blob/8d86e202886016097db383f932bc00a7f22c1cc0/theories/Semantics/IFOL/IFOL_Soundness.v\c_hash_str L283}
\prop_gput:Nnn \g_hyperlogic_coq_urls_prop {ifol_to_epsilon_derivation} {https://github.com/mozow01/HyperLogic/blob/8d86e202886016097db383f932bc00a7f22c1cc0/theories/Semantics/Metatheory/IFOL_EpsilonFull_Translation.v\c_hash_str L258}
\prop_gput:Nnn \g_hyperlogic_coq_urls_prop {ifol_to_epsilon_h_ty} {https://github.com/mozow01/HyperLogic/blob/8d86e202886016097db383f932bc00a7f22c1cc0/theories/Semantics/Metatheory/IFOL_EpsilonFull_Translation.v\c_hash_str L57}
\prop_gput:Nnn \g_hyperlogic_coq_urls_prop {interpreted_formula_epsilon_selection_adjunction} {https://github.com/mozow01/HyperLogic/blob/8d86e202886016097db383f932bc00a7f22c1cc0/theories/Semantics/EpsilonFull/EpsilonFull_AllPredicates.v\c_hash_str L67}
\prop_gput:Nnn \g_hyperlogic_coq_urls_prop {pasquali_has_enough_epsilon_terms_semantic} {https://github.com/mozow01/HyperLogic/blob/8d86e202886016097db383f932bc00a7f22c1cc0/theories/Semantics/Metatheory/Pasquali_Epsilon.v\c_hash_str L193}
\prop_gput:Nnn \g_hyperlogic_coq_urls_prop {pullback_has_left_adjoint} {https://github.com/mozow01/HyperLogic/blob/8d86e202886016097db383f932bc00a7f22c1cc0/theories/Cat/Fibration.v\c_hash_str L49}
\prop_gput:Nnn \g_hyperlogic_coq_urls_prop {pullback_has_right_adjoint} {https://github.com/mozow01/HyperLogic/blob/8d86e202886016097db383f932bc00a7f22c1cc0/theories/Cat/Fibration.v\c_hash_str L50}
\ExplSyntaxOff

\ExplSyntaxOn
\DeclareRobustCommand{\coqterm}[1]{
  \href{\prop_item:Nn \g_hyperlogic_coq_urls_prop {#1}}{
    \textcolor{coqgrayblue}{\textnormal{\ttfamily
      \str_map_inline:nn {#1}
        {##1\str_case:nn {##1} {{.}{\allowbreak}{_}{\allowbreak}}}
    }}
  }
}
\ExplSyntaxOff

\usepackage{tikz}%
\usetikzlibrary{fit}
\usetikzlibrary{positioning, arrows.meta, shapes.geometric, 3d, babel, calc}%
\usepackage{tikz-cd}%

\newcommand{\sem}[1]{[\![#1]\!]}

\theoremstyle{thmstyleone}%
\newtheorem{theorem}{Theorem}[section]%
\newtheorem{proposition}{Proposition}[section]%
\newtheorem{corollary}{Corollary}[section]%

\theoremstyle{thmstyletwo}%
\newtheorem{example}{Example}[section]%
\newtheorem{remark}{Remark}[section]%

\theoremstyle{thmstylethree}%
\newtheorem{definition}{Definition}[section]%
\newtheorem{notation}{Notation}[section]%

\begin{document}

\title[Internal Languages for Thin Epsilon-Hyperdoctrines]{Internal Languages for Thin Epsilon-Hyperdoctrines: Soundness, Completeness, and Rocq Formalization}

\author*[1]{\fnm{Zolt\'an G\'abor} \sur{Moln\'ar}}\email{mozow@math.bme.hu}
\author[1]{\fnm{Csaba} \sur{T\'oth}}\email{toth.csaba@ttk.bme.hu}

\affil*[1]{\orgdiv{Department of Algebra and Geometry, Institute of Mathematics, Faculty of Natural Sciences}, \orgname{Budapest University of Technology and Economics}, \orgaddress{\street{Egry J\'ozsef u. 1.}, \city{Budapest}, \postcode{H-1111}, \country{Hungary}}}

\abstract{We present a typed proof-term calculus for intuitionistic
intensional first-order logic with Hilbert's epsilon-operator and prove
its soundness and completeness with respect to hyperdoctrinal models
with thin bicartesian closed setoid categories (BiCCCs) as fibres and
formula-indexed witness selections.
The calculus provides an explicit internal language for these
epsilon-hyperdoctrines: for each fixed typed signature, its syntactic
term model is initial for syntax-indexed, interpretation-preserving
model morphisms, with uniqueness up to setoid equality.
The nested epsilon syntax is organized by a well-founded dependency
relation on formula labels. Concrete countermodels exhibit failure of
excluded middle and show that formula conversion need not identify the
selected witnesses. We also compare the full calculus with its
epsilon-free and epsilon-only fragments.
The principal constructions and metatheorems are formalized in Rocq.
With setoid BiCCCs, proof irrelevance is a choice of equality, not the
erasure of proofs: the proofs and their categorical operations remain
available as the basis for a future proof-relevant categorical semantics.}



\maketitle

\section{Introduction}
\label{sec:introduction}

Hilbert's $\varepsilon$-operator gives a name to an individual object satisfying a predicate, if such an object exists. In the usual first-order notation, a formula $\Phi(\vec x,y)$ has an epsilon term $\varepsilon y.\Phi(\vec x,y)$, where $\vec x$ lists its
parameters. A first-order logical calculus with epsilon terms is called an intensional Hilbertian epsilon calculus if only the critical epsilon axiom is required:
\[
 \exists y\,\Phi(\vec x,y)
 \;\to \;
 \Phi\bigl(\vec x,\varepsilon y.\Phi(\vec x,y)\bigr)
\]
and neither extensionality nor substitution invariance of the selected witnesses is assumed. (Cf.~\cite{HilbertBernays1939}). We define a proof term calculus in natural-deduction style for the intuitionistic, intensional, first-order $\varepsilon$-logic~\cite{Prawitz1965} and we use the notation $\mathrm{iFOL}_{\exists\varepsilon}$ for it. Building on the works on categorical
$\varepsilon$-doctrines~\cite{Bell1993Intuitionistic,Pasquali2014}, we give a sound and complete thin, first-order, hyperdoctrinal categorical semantics for $\mathrm{iFOL}_{\exists\varepsilon}$, whose models we call \textit{thin $\varepsilon$-hyperdoctrines}.

\subsection{Thin \texorpdfstring{$\varepsilon$}{epsilon}-hyperdoctrines}
We work with bicartesian closed setoid category-valued hyperdoctrines with thin fibres, i.e. a setoid equivalence is defined over hom-sets and parallel fibre morphisms are regarded as equal under this equivalence.
With setoid BiCCCs, proof irrelevance is a choice of equality, not the
erasure of proofs: the proofs and their categorical operations remain
available as the basis for a future proof-relevant categorical semantics.
We suppose that the base category is a cartesian category ($\mathcal B$) and the hyperdoctrine's fibre system is given by a $\mathcal B$-indexed category
\[
 P:\mathcal B^{\mathrm{op}}\longrightarrow\mathsf{BiCCC},
 \qquad
 f^*:=P(f):P(Y)\longrightarrow P(X)
 \quad(f:X\to Y).
\]
For every projection morphism $\pi:X\times A\to X$ of the base, quantifications are given by the usual adjunctions
$\Sigma_\pi\dashv\pi^*\dashv\Pi_\pi$, with the Beck--Chevalley
isomorphisms~\cite{Lawvere1970,Jacobs1999}. These data define a total hyperdoctrinal category $\int _P$.

Let us fix a signature $\Sigma$ for $\mathrm{iFOL}_{\exists\varepsilon}$. We equip $\int _P$ with the structure of a thin $\varepsilon$-hyperdoctrine by the following construction. Predicates are interpreted by objects of the fibres of $\int _P$, and the morphisms represent proofs from given premisses. Suppose $\Phi(\vec x,y)$ is interpreted by the object
$\varphi\in P(X\times A)$, where $X$ interprets its parameter context and
$A$ is the type of $y$. For an epsilon term $\varepsilon y.\Phi(\vec x,y)$, it is supposed that there is a base morphism $e_\Phi:X\to A$ satisfying
\[
 \Sigma_\pi\varphi
 \cong
 \langle\mathrm{id}_X,e_\Phi\rangle^*\varphi
\]
(Cf.~\cite{Pasquali2016}). This condition corresponds to the Hilbertian equivalence
\[
 \exists y\,\Phi(\vec x,y)\leftrightarrow \Phi(\vec x,\varepsilon y.\Phi(\vec x,y)).
\]
Thus a thin $\varepsilon$-hyperdoctrine depends on the fixed signature $\Sigma$, and the base morphisms $e_\Phi$ are indexed by formulas $\Phi$ of the syntax.

\subsection{Results}

Our main results are soundness and completeness for
$\mathrm{iFOL}_{\exists\varepsilon}$, with iFOL as the epsilon-free
comparison calculus. Appendix~\ref{app:epsilon-only} gives the
corresponding syntax, models, and metatheorems for the epsilon-only
calculus $\mathrm{iFOL}_\varepsilon$. Each calculus has its own class
of thin hyperdoctrinal models. Soundness interprets
derivations as fibre morphisms. For well-formed formulae and contexts,
completeness recovers a derivation whenever the corresponding hom-set is
inhabited in every model. In each completeness proof, we obtain a
derivation from a morphism in the term model. For a fixed signature
$\Sigma$, the term model of $\mathrm{iFOL}_{\exists\varepsilon}$ is
initial for the model morphisms defined below, with uniqueness
up to their setoid equivalence. In this sense, the calculus gives an
internal language for these models.

The categorical interpretation of the critical axiom follows
Pasquali~\cite[Prop.~6.2]{Pasquali2016}. Completeness for
$\mathrm{iFOL}_{\exists\varepsilon}$ still holds when models are equipped
with $\varepsilon$-selections for every semantic predicate over an
interpreted product type, not only for predicates represented by formulae.
The term model itself satisfies this extra condition. It is Pasquali's
condition of having enough $\varepsilon$-terms, restricted to the
interpreted base types~\cite[Def.~2.1]{Pasquali2019}.
We do not require the equality predicates of his
$\varepsilon$-doctrines~\cite[Def.~2.2]{Pasquali2019}.

$\mathrm{iFOL}_{\exists\varepsilon}$ proves the Dual Drinker formula
\[
 \exists x\,\bigl((\exists y\,R(y))\to R(x)\bigr)
\]
for a unary predicate $R$. This is the schematic Hilbert choice principle
studied in~\cite[Sec.~4]{WarrenDienerMcKubreJordens2018}. An $H_5$-valued
countermodel refutes it in intuitionistic first-order logic, proving
non-conservativity over that calculus. A model over the three-element
Heyting chain $H_3$ refutes excluded middle in
$\mathrm{iFOL}_{\exists\varepsilon}$. Thus the calculus has the expected
Hilbertian choice principle without becoming classical, even though
$\varepsilon$-substitution has the full left-adjoint role of existential
quantification. This agrees with Pasquali's nonclassical
$\varepsilon$-semantics~\cite[Sec.~4]{Pasquali2019}.

We also prove that a concrete substitution equation between
$\varepsilon$-terms is independent of our semantic axioms. It holds in
one model and fails in another, although their base, fibres, and atomic
interpretation are the same. The selected formula
instances remain stable under substitution up to isomorphism.
Thus the substitution law for the chosen formula instances does not
force the substitution law for their witnesses. This is a definitional
version of the intensional--substitutive distinction discussed
in~\cite[Sec.~2.3]{Molnar2013}. Formula conversion need not give equal
$\varepsilon$-terms either. These are base-term equalities, so no equality
predicate is needed in the object language.

The principal constructions and metatheorems are formalized in Rocq.
The proofs were developed with AI assistance under the authors'
mathematical supervision. Section~\ref{sec:formalization-coq} records the
formal scope and metatheoretic assumptions.

\section{Syntax and calculus}
\label{sec:syntax-and-calculus}

\subsection{Base and formula syntax}

We first give $\mathrm{iFOL}_{\exists\varepsilon}$ over a typed
function--predicate signature $\Sigma$. Its proof judgement is
$I\mid\Gamma\vdash_{\exists\varepsilon}p:\Phi$; inside its displayed
rule trees we abbreviate this as $I\mid\Gamma\vdash p:\Phi$.
The epsilon-free calculus is specified below; the epsilon-only calculus
is given in Appendix~\ref{app:epsilon-only}.
Names in environment headings identify the corresponding Coq declarations
or the definitions and lemmas used in the displayed instance.

\begin{remark}[The dependence of epsilon terms on formulae, \coqterm{CanonicalEpsilonFull.EpsilonLabel}, \coqterm{CanonicalEpsilonFull.dependency_well_founded}]
An $\varepsilon$-term is indexed by a formula, which may itself contain
$\varepsilon$-terms. We record the formula in the label of a base constant,
so terms, formulae, and labels must be defined together. If these
dependencies were circular, evaluating a labelled constant could require
its own value. We define finite syntax trees by simultaneous induction
and prove that label dependencies decrease a natural-number rank.
This gives the recursion needed later to construct epsilon interpretations.
\end{remark}

\begin{definition}[Base types, \coqterm{EpsilonFull.Base_Syntax.b_ty}]
\label{def:base-types}
The type $\mathsf{BType}$ of base types is inductively defined by
\[
 A,B,I,J ::= \mathbf 1\mid S\mid A\times B.
\]
(Cf.~\cite[p.~138]{Jacobs1999}).
\end{definition}

\begin{remark}[Cartesian reading, \coqterm{EpsilonFull.Base_Syntax.b_ty}]
We use base types for base contexts, too. The base contexts are intended to be the variable declarations of a predicate formula. Base types mimic a cartesian base category:
$\mathbf 1$ represents the terminal object, $S$ is the unique primitive sort
of the one-sorted language, and $A\times B$ represents the binary product.
Thus base contexts are product types, where
$\mathbf 1$ represents the empty context and $I\times A$ the extension of
$I$ by type $A$. This is an simplified version of Jacobs' base types, without function types and contexts represented by lists.
(Cf.~\cite[p.~139]{Jacobs1999}).
\end{remark}

\begin{definition}[Typed function--predicate signature, \coqterm{EpsilonFull.Base_Syntax.FunctionSignature}, \coqterm{EpsilonFull.Hyper_Syntax.PredicateSignature}]
\label{def:typed-signature}
A typed function--predicate signature $\Sigma$ over $\mathsf{BType}$ consists of
sets $\mathcal F$ of function symbols and $\mathcal R$ of predicate
symbols, together with maps
\[
 \operatorname{dom},\operatorname{cod}:\mathcal F\to\mathsf{BType},
 \qquad
 \operatorname{ar}:\mathcal R\to\mathsf{BType}.
\]
Thus $u\in\mathcal F$ has input type $\operatorname{dom}(u)$ and output type
$\operatorname{cod}(u)$, while $R\in\mathcal R$ has argument type
$\operatorname{ar}(R)$. Multiple arguments are represented by product types.
(Cf.~\cite[p.~222]{Jacobs1999}).
\end{definition}

\begin{example}[Nested epsilon terms, \coqterm{EpsilonFull.Examples.nested_epsilon_term_typing}]
\label{ex:nested-epsilon-terms}
In the intended $\mathcal L_{\exists\varepsilon}$-language, let $R$ be a predicate symbol with $\operatorname{ar}(R)=S\times S$.
In ordinary variable notation, with $x,y$ of sort $S$, consider
\[
 \varepsilon x\,R\bigl(x,\varepsilon y\,R(x,y)\bigr).
\]
The inner epsilon term depends on $x$. A recursive interpretation first
defines its value for each $x$, then uses these values to interpret the
outer epsilon term.
\end{example}

\begin{definition}[The $\mathrm{iFOL}_{\exists\varepsilon}$-language, \coqterm{EpsilonFull.Base_Syntax.b_tm}, \coqterm{EpsilonFull.Hyper_Syntax.h_ty}, \coqterm{CanonicalEpsilonFull.EpsilonLabel}]
\label{def:epsilon-language}
Fix a typed function--predicate signature $\Sigma$. The language
$\mathcal L_{\exists\varepsilon}$ is given by the sets
$\mathsf{BTerm}_\Sigma^{\exists\varepsilon}$ of raw base terms,
$\mathsf{Form}_\Sigma^{\exists\varepsilon}$ of raw formulae, and $\mathsf{Lab}_\Sigma^{\exists\varepsilon}$ of
$\varepsilon$-labels, defined simultaneously by induction as sets of
finite syntax trees:
\[
\begin{aligned}
 \mathsf{BTerm}_\Sigma^{\exists\varepsilon}\ni t
   &::=\mathrm{id}\mid g\circ f\mid {!}\mid\langle f,g\rangle
       \mid\pi_1\mid\pi_2\mid \widehat u\mid c_s^A,\\
 \mathsf{Form}_\Sigma^{\exists\varepsilon}\ni\Phi
   &::=\top\mid\bot\mid R(t)\mid
       \Phi\land\Psi\mid\Phi\lor\Psi\mid\Phi\to\Psi
       \mid\forall_A\Phi\mid\exists_A\Phi,\\
 \mathsf{Lab}_\Sigma^{\exists\varepsilon}\ni s&::=\langle A,\Phi\rangle,
\end{aligned}
\]
where $f,g,t\in\mathsf{BTerm}_\Sigma^{\exists\varepsilon}$,
$\Phi,\Psi\in\mathsf{Form}_\Sigma^{\exists\varepsilon}$, $s\in\mathsf{Lab}_\Sigma^{\exists\varepsilon}$,
$A\in\mathsf{BType}$, $u\in\mathcal F$, and $R\in\mathcal R$.
A logical context is a list
$\Gamma\in\operatorname{List}(\mathsf{Form}_\Sigma^{\exists\varepsilon})$.
For $A\in\mathsf{BType}$ and $\Phi\in\mathsf{Form}_\Sigma^{\exists\varepsilon}$, define
\[
 \varepsilon_\Phi^A:=c_{\langle A,\Phi\rangle}^A
 \in\mathsf{BTerm}_\Sigma^{\exists\varepsilon}.
\]
A label stores the raw formula itself, including its terms and their
labels. Typing conditions are imposed separately below.
\end{definition}

\begin{definition}[Label dependency and rank, \coqterm{CanonicalEpsilonFull.immediate_dependency}, \coqterm{CanonicalEpsilonFull.epsilon_rank}]
\label{def:label-rank}
In the raw $\mathcal L_{\exists\varepsilon}$-syntax over a fixed typed signature $\Sigma$, a label $r$ occurs directly in a raw formula $\Phi$ if a constant $c_r^B$
occurs in one of its base terms, for some $B\in\mathsf{BType}$, without
entering the formulae stored in other labels. For $r,s\in\mathsf{Lab}_\Sigma^{\exists\varepsilon}$,
write $r\prec s$ if $s=\langle A,\Phi\rangle$ and $r$ occurs directly in
$\Phi$.

The rank $\operatorname{rk}$ is defined simultaneously by structural
recursion on labels, terms, and formulae. It is $0$ on
$\mathrm{id},{!},\pi_1,\pi_2,\widehat u,\top,\bot$, and
\[
\begin{aligned}
 \operatorname{rk}\langle A,\Phi\rangle
   &:=1+\operatorname{rk}(\Phi),
 \qquad \operatorname{rk}(c_s^B):=\operatorname{rk}(s),\\
 \operatorname{rk}(g\circ f)
   &:=\operatorname{rk}\langle f,g\rangle
     :=\max\{\operatorname{rk}(f),\operatorname{rk}(g)\},\\
 \operatorname{rk}(R(t))&:=\operatorname{rk}(t),\\
 \operatorname{rk}(\Phi\star\Psi)
   &:=\max\{\operatorname{rk}(\Phi),\operatorname{rk}(\Psi)\},\\
 \operatorname{rk}(\forall_A\Phi)
   &:=\operatorname{rk}(\exists_A\Phi):=\operatorname{rk}(\Phi),
\end{aligned}
\]
where $\star\in\{\land,\lor,\to\}$. Thus rank measures label nesting,
not the number of logical connectives.
\end{definition}

\begin{example}[Label nesting, \coqterm{EpsilonFull.Examples.second_label_rank}]
\label{ex:label-nesting}
In $\mathcal L_{\exists\varepsilon}$ over a signature with $\operatorname{ar}(R)=S$, the raw expressions
\[
 s_0:=\langle S,\top\rangle,\qquad
 \Phi_1:=R(c_{s_0}^{S}),\qquad
 s_1:=\langle S,\Phi_1\rangle
\]
satisfy $s_0\prec s_1$,
$\operatorname{rk}(s_0)=\operatorname{rk}(\Phi_1)=1$, and
$\operatorname{rk}(s_1)=2$. Here $s_0$ is a subtree of $s_1$, not a
reference to a separately assigned name.
\end{example}

\begin{proposition}[Decrease of label rank, \coqterm{CanonicalEpsilonFull.immediate_dependency_rank}]
\label{prop:label-rank}
Fix a typed signature $\Sigma$ and its finite-tree language $\mathcal L_{\exists\varepsilon}$. For all $r,s\in\mathsf{Lab}_\Sigma^{\exists\varepsilon}$,
\[
 r\prec s\quad\Longrightarrow\quad
 \operatorname{rk}(r)<\operatorname{rk}(s).
\]
\end{proposition}

\begin{proof}
Induction on terms and formulae gives
$\operatorname{rk}(r)\leq\operatorname{rk}(\Phi)$ whenever $r$ occurs
directly in $\Phi$. For $s=\langle A,\Phi\rangle$ this gives
$\operatorname{rk}(r)\leq\operatorname{rk}(\Phi)
<1+\operatorname{rk}(\Phi)=\operatorname{rk}(s)$.
\end{proof}

\begin{corollary}[Well-founded recursion on labels, \coqterm{CanonicalEpsilonFull.dependency_well_founded}]
\label{cor:label-recursion}
For the raw $\mathcal L_{\exists\varepsilon}$-labels over any typed signature $\Sigma$, the direct-occurrence relation $\prec$ is well founded and has no dependency cycles.
A value at $\langle A,\Phi\rangle$ can therefore be defined recursively
from values at the labels occurring directly in $\Phi$.
\end{corollary}

\begin{proof}
Proposition~\ref{prop:label-rank} reduces $\prec$ to the well-founded
order $<$ on $\mathbb N$. Hence $\prec$ is well founded and supports
well-founded recursion.
A dependency cycle would imply
$\operatorname{rk}(s)<\operatorname{rk}(s)$.
\end{proof}

\begin{remark}[Use in epsilon interpretation, \coqterm{FHCanonicalChoice_unfold}]
To construct the value of $\varepsilon_\Phi^A$, we first interpret the
labelled constants occurring in $\Phi$. Their labels have smaller rank.
The resulting evaluation of $\Phi$ is then used to choose a witness.
In the $H_3$-valued model of Proposition~\ref{prop:h3-epsilon-model},
the finite carriers and the order $\bot<m<\top$ supply a maximizing
witness at each step. The rank argument justifies the recursion, but
does not by itself supply witnesses in an arbitrary model.
\end{remark}

\begin{remark}[Rocq implementation, \coqterm{CanonicalEpsilonFull.EpsilonLabel}]
Rocq first defines raw terms and formulae with an arbitrary label type,
then defines the nested inductive type \texttt{EpsilonLabel Func Pred}.
The sets $\mathsf{BTerm}_\Sigma^{\exists\varepsilon}$, $\mathsf{Form}_\Sigma^{\exists\varepsilon}$, and
$\mathsf{Lab}_\Sigma^{\exists\varepsilon}$ correspond to \texttt{b\_tm}, \texttt{h\_ty}, and
\texttt{EpsilonLabel}, respectively, with the label parameter
\texttt{Sign} of the first two instantiated by
\texttt{EpsilonLabel Func Pred}.
Its constructor is $(A,\Phi)\mapsto\langle A,\Phi\rangle$, also used as
\texttt{skolem\_index}. The simultaneous grammar above presents the same
finite data. The dependency and rank results apply to all raw codes,
independently of typing.
\end{remark}

\begin{definition}[Base terms and typeability, \coqterm{EpsilonFull.Base_Syntax.base_typing}]
\label{def:typed-base-terms}
In the base syntax of $\mathcal L_{\exists\varepsilon}$ over $\Sigma$, the judgement $I\vdash a:A$ is defined inductively by the following rules:
\[
\prfaxiom{I\vdash \mathrm{id}_I:I}
\qquad
\prftree
 {I\vdash f:J}
 {J\vdash g:A}
 {I\vdash g\circ f:A}
\qquad
\prfaxiom{I\vdash {!}_I:\mathbf 1},
\]
\[
\prftree
 {I\vdash f:A}
 {I\vdash g:B}
 {I\vdash\langle f,g\rangle:A\times B}
\qquad
\prfaxiom{A\times B\vdash\pi_1:A}
\qquad
\prfaxiom{A\times B\vdash\pi_2:B},
\]
\[
\prftree
 {u\in\mathcal F}
 {\operatorname{dom}(u)\vdash\widehat u:\operatorname{cod}(u)}
\qquad
\prftree
 {s\in\mathsf{Lab}_\Sigma^{\exists\varepsilon}}
 {I\vdash c_s^A:A}.
\]
This is the product fragment, without function types.
(Cf.~\cite[p.~139]{Jacobs1999}).
\end{definition}

\begin{remark}[Epsilon terms and labelled constants, \coqterm{EpsilonFull.Base_Syntax.BT_Const}]
For every raw $\Phi$ and every base context $I$, the distinguished constant
$\varepsilon_\Phi^A=c_{\langle A,\Phi\rangle}^A$ satisfies
$I\vdash\varepsilon_\Phi^A:A$. The grammar also contains $c_s^A$ for every
$s\in\mathsf{Lab}_\Sigma^{\exists\varepsilon}$ and every $A$. Thus, if
$s=\langle D,\Theta\rangle$, neither $D=A$ nor well-formedness of $\Theta$
is required for the raw constant $c_s^A$. This larger family is the base
syntax used in the formalization.
\end{remark}

\begin{definition}[Base-term equivalence, \coqterm{EpsilonFull.Base_Syntax.base_equiv}]
\label{def:base-term-equivalence}
For the base syntax of $\mathcal L_{\exists\varepsilon}$ over $\Sigma$, base-term equivalence is defined inductively as the least family of relations
$\equiv_{I,A}$ on well-typed terms $I\vdash f,g:A$ closed under reflexivity,
symmetry, transitivity, composition congruence, and pairing congruence, and
containing
\[
\begin{gathered}
h\circ(g\circ f)\equiv_{I,D}(h\circ g)\circ f,
\qquad
\mathrm{id}_A\circ f\equiv_{I,A}f
\equiv_{I,A}f\circ\mathrm{id}_I,\\
\pi_1\circ\langle f,g\rangle\equiv_{I,A}f,
\qquad
\pi_2\circ\langle f,g\rangle\equiv_{I,B}g,\\
\langle\pi_1\circ h,\pi_2\circ h\rangle
   \equiv_{I,A\times B}h,
\qquad
k\equiv_{I,\mathbf 1}{!}_I,\\
\langle f,g\rangle\circ h
   \equiv_{K,A\times B}\langle f\circ h,g\circ h\rangle.
\end{gathered}
\]
All equations are restricted to well-typed composites and pairs. For example, associativity uses $I\vdash f:J$, $J\vdash g:K$, and $K\vdash h:D$; the pairing equations use two arrows with a common domain, and the terminal equation uses $I\vdash k:\mathbf 1$. In cases where the domain and codomain are known, we omit the subscripts and instead write $\equiv$.
We also write $I\vdash a\equiv b:A$ for $a\equiv_{I,A}b$.
Its negation, $I\nvdash a\equiv b:A$, means that this base equation is not derivable.
(Cf.~\cite[p.~139]{Jacobs1999}).
\end{definition}

\begin{remark}[Named assumptions, \coqterm{EpsilonFull.Hyper_Syntax.tm_var}]
The Rocq development represents assumption positions by de Bruijn indices.
For readability, we write $x:\Phi\in\Gamma$ for the position labelled by
$x$.
\end{remark}

\begin{definition}[Well-formed formulae and contexts, \coqterm{EpsilonFull.Hyper_Syntax.WellFormed}, \coqterm{EpsilonFull.Hyper_Syntax.All_WellFormed}]
\label{def:formulae-contexts}
For $\mathcal L_{\exists\varepsilon}$ over $\Sigma$, the judgement
$I\vdash\Phi\;\mathsf{form}$ is defined inductively by
\[
\prfaxiom{I\vdash\top\;\mathsf{form}}
\qquad
\prfaxiom{I\vdash\bot\;\mathsf{form}}
\qquad
\prftree
 {R\in\mathcal R}
 {I\vdash t:\operatorname{ar}(R)}
 {I\vdash R(t)\;\mathsf{form}},
\]
\[
\prftree
 {I\vdash\Phi\;\mathsf{form}}
 {I\vdash\Psi\;\mathsf{form}}
 {I\vdash\Phi\mathbin{\star}\Psi\;\mathsf{form}}
\quad
(\star\in\{\land,\lor,\to\}),
\]
\[
\prftree
 {I\times A\vdash\Phi\;\mathsf{form}}
 {I\vdash\forall_A\Phi\;\mathsf{form}}
\qquad
\prftree
 {I\times A\vdash\Phi\;\mathsf{form}}
 {I\vdash\exists_A\Phi\;\mathsf{form}}.
\]
A logical context $\Gamma=(\Phi_1,\ldots,\Phi_n)$ is well formed over $I$
when every $I\vdash\Phi_k\;\mathsf{form}$. We then write
$I\vdash\Gamma\;\mathsf{ctx}$.
(Cf.~\cite[pp.~223--224]{Jacobs1999}).
\end{definition}

\begin{remark}[The two contexts, \coqterm{EpsilonFull.Hyper_Syntax.h_context}]
The base context $I$ contains individual variables, whereas $\Gamma$ contains
proof assumptions. Only the former is extended by a quantifier.
(Cf.~\cite[p.~224]{Jacobs1999}).
\end{remark}

\begin{remark}[Role of well-formedness, \coqterm{EpsilonFull.Hyper_Syntax_Lemmas.meta_ctx_wf}]
Well-formedness is required for base-term substitution and formula
conversion, the formal equality of formulae. It ensures coherent substitution
between fibres.
\end{remark}

\begin{example}[Formation under a binder, \coqterm{EpsilonFull.Examples.persons_under_binder_wellformed}]
In $\mathcal L_{\exists\varepsilon}$ over $\Sigma$, if $\operatorname{ar}(R)=A$ and $I\times B\vdash t:A$, then
\[
 I\times B\vdash R(t)\;\mathsf{form}
 \qquad\text{and}\qquad
 I\vdash\forall_B R(t)\;\mathsf{form}.
\]
Let \(S\) be the type of persons and let
\(\operatorname{ar}(R)=S\times(S\times S)\), where \(R(x,y,z)\)
means that \(z\) suspects that \(x\) loves \(y\). In the context
\(S\times(S\times S)\), define
\[
 y:=\pi_1,\qquad
 x:=\pi_1\circ\pi_2,\qquad
 z:=\pi_2\circ\pi_2,
 \qquad
 t:=\langle x,\langle y,z\rangle\rangle.
\]
Then
\[
 S\vdash\forall_{S\times S}R(t)\;\mathsf{form}.
\]
The formula says that, everybody suspects that everybody loves \(y\).
\end{example}

\subsection{Syntactic substitution and the intensional
\texorpdfstring{$\varepsilon$}{epsilon}-laws}

\begin{definition}[Syntactic substitution, \coqterm{EpsilonFull.Hyper_Syntax.form_subst}, \coqterm{EpsilonFull.Hyper_Syntax.ctx_subst}]
\label{def:syntactic-substitution}
In $\mathcal L_{\exists\varepsilon}$ over $\Sigma$, let $I,J,A\in\mathsf{BType}$ and $I\vdash f:J$. Put
\[
 I\times A\vdash
 f_A^+:=\langle f\circ\pi_1,\pi_2\rangle:J\times A.
\]
Syntactic substitution of raw formulae along $f$ is defined by structural
recursion. The well-formed roles of the terms and formulae are recorded on
the right:
\[
\begin{array}{rcll}
 \top[f]&=&\top,\\
 \bot[f]&=&\bot,\\
 R(t)[f]&=&R(t\circ f)
   &\bigl(J\vdash t:\operatorname{ar}(R)\bigr),\\
 (\Phi\mathbin{\star}\Psi)[f]
   &=&\Phi[f]\mathbin{\star}\Psi[f]
   &\bigl(J\vdash\Phi\;\mathsf{form},\
          \ J\vdash\Psi\;\mathsf{form},\
          \ \star\in\{\land,\lor,\to\}\bigr),\\
 (\forall_A\Phi)[f]&=&\forall_A(\Phi[f_A^+])
   &\bigl(J\times A\vdash\Phi\;\mathsf{form}\bigr),\\
 (\exists_A\Phi)[f]&=&\exists_A(\Phi[f_A^+])
   &\bigl(J\times A\vdash\Phi\;\mathsf{form}\bigr).
\end{array}
\]
If $J\vdash\Gamma\;\mathsf{ctx}$, then $\Gamma[f]$ is defined
componentwise and is a logical context over $I$.
(Cf.~\cite[pp.~223, 226]{Jacobs1999}).
\end{definition}

\begin{notation}[Instantiation, \coqterm{EpsilonFull.Hyper_Syntax.form_subst}]
\label{not:formula-instantiation}
In $\mathcal L_{\exists\varepsilon}$ over $\Sigma$, for $I,A\in\mathsf{BType}$, $I\vdash a:A$, and
$I\times A\vdash\Phi\;\mathsf{form}$, write
\[
 \Phi[a]:=\Phi[\langle\mathrm{id}_I,a\rangle].
\]
\end{notation}

\begin{remark}[Substitution by a base term, \coqterm{EpsilonFull.Hyper_Syntax.b_subst_ext}]
A base term $I\vdash f:J$ is a substitution for the free individual
variables of a formula over $J$. The lift $f_A^+$ leaves the bound $A$-variable
unchanged. Syntactic substitution is used below in the quantifier rules,
formula conversion, proof-term substitution, and later is interpreted by
the pullback functor $P(\sem{f})$.
(Cf.~\cite[pp.~48--49, 226]{Jacobs1999}).
\end{remark}

\begin{example}[Free and bound variables, \coqterm{EpsilonFull.Examples.free_bound_projection_equation}]
In $\mathcal L_{\exists\varepsilon}$ over $\Sigma$, let $I,J,A\in\mathsf{BType}$, $J\vdash f:I$, and
$I\times A\vdash\Theta\;\mathsf{form}$. Then
\[
\prftree
 {J\times A\vdash f_A^+:I\times A}
 {I\times A\vdash\Theta\;\mathsf{form}}
 {J\times A\vdash\Theta[f_A^+]\;\mathsf{form}},
\qquad
 (\forall_A\Theta)[f]=\forall_A(\Theta[f_A^+]).
\]
Writing $I\times A\vdash\bar\pi_I:A$ and
$J\times A\vdash\bar\pi_J:A$ for the second projections,
\[
 \bar\pi_I\circ f_A^+
 \equiv_{J\times A,A}\bar\pi_J.
\]
\end{example}

\begin{definition}[Formula conversion, \coqterm{EpsilonFull.Hyper_Syntax.form_equiv}]
\label{def:formula-conversion}
In $\mathrm{iFOL}_{\exists\varepsilon}$ over $\Sigma$, for each base context $I$, formula conversion $\Phi\simeq_I\Psi$ is defined
inductively as the least relation on raw formulae with reflexivity restricted
to well-formed formulae. It is closed under symmetry, transitivity, and
congruence for all connectives and quantifiers. Its atomic clause requires $R\in\mathcal R$,
$I\vdash t,u:\operatorname{ar}(R)$, and
$t\equiv_{I,\operatorname{ar}(R)}u$, and concludes
$R(t)\simeq_I R(u)$. Its substitution clauses are
\[
 \Phi[\mathrm{id}_I]\simeq_I\Phi
 \quad(I\vdash\Phi\;\mathsf{form}),
\]
\[
 (\Phi[f])[g]\simeq_I\Phi[f\circ g]
 \quad(K\vdash\Phi\;\mathsf{form},\ J\vdash f:K,\ I\vdash g:J),
\]
\[
 f\equiv_{I,J}g\Longrightarrow\Phi[f]\simeq_I\Phi[g]
 \quad(J\vdash\Phi\;\mathsf{form}).
\]
For $I\times A\vdash\Phi\;\mathsf{form}$ it contains
\[
 \exists_A\Phi\simeq_I\Phi[\varepsilon_\Phi],
\tag{$\varepsilon\Sigma$}
\]
and for $J\times A\vdash\Theta\;\mathsf{form}$ and $I\vdash f:J$ it contains
\[
 \Theta[f_A^+][\varepsilon_{\Theta[f_A^+]}]
 \simeq_I
 \Theta[\varepsilon_\Theta][f].
\tag{$\varepsilon\mathrm{BC}$}
\]
\end{definition}

\begin{definition}[Context conversion, \coqterm{EpsilonFull.Hyper_Syntax.ctx_equiv}]
\label{def:context-conversion}
In $\mathrm{iFOL}_{\exists\varepsilon}$ over $\Sigma$, context conversion for two lists of formulae over the same base context $I$ is defined componentwise:
\[
 (\Phi_1,\ldots,\Phi_n)\simeq_I(\Psi_1,\ldots,\Psi_n)
 \quad\Longleftrightarrow\quad
 \Phi_k\simeq_I\Psi_k\quad(1\leq k\leq n).
\]
\end{definition}

\begin{remark}[The role of formula conversion, \coqterm{EpsilonFull_Completeness.h_derivation_subst_general}]
Raw syntactic substitution is neither strictly unital nor strictly
associative, and equal base substitutions need not produce identical raw
formulae. Conversion makes these identifications available to the typing
rules. Its $\varepsilon$-laws
are used in Proposition~\ref{prop:derivation-substitution}, interpreted as
isomorphisms in soundness, and built into the syntactic model for
completeness.
\end{remark}

\begin{example}[Conversion along equal witnesses, \coqterm{EpsilonFull.Examples.conversion_along_equal_witnesses}]
In $\mathrm{iFOL}_{\exists\varepsilon}$ over $\Sigma$, let $I,A\in\mathsf{BType}$. If $I\vdash a,b:A$,
$a\equiv_{I,A}b$, and
$I\times A\vdash\Phi\;\mathsf{form}$, then
\[
\prftree[r]{$\scriptstyle\mathrm{pullback\mbox{-}cong}$}
 {
  \prftree[r]{$\scriptstyle\mathrm{pair\mbox{-}cong}$}
   {a\equiv_{I,A}b}
   {\langle\mathrm{id}_I,a\rangle
    \equiv_{I,I\times A}
    \langle\mathrm{id}_I,b\rangle}
 }
 {\Phi[a]\simeq_I\Phi[b]}.
\]
\end{example}

\begin{example}[Constant substitution and $\varepsilon$, \coqterm{EpsilonFull.Examples.projection_epsilon_beck_chevalley}]
In $\mathrm{iFOL}_{\exists\varepsilon}$ over $\Sigma$, let $I,A,B\in\mathsf{BType}$ and
$I\times A\vdash\Phi\;\mathsf{form}$. Put $K:=I\times B$ and write
\[
 K\vdash\rho:=\pi_1:I,
 \qquad
 K\times A\vdash
 \rho_A^+=\langle\rho\circ\pi_1,\pi_2\rangle:I\times A.
\]
The two $\varepsilon$-instances over $K$ are
\[
 \Phi[\rho_A^+][\varepsilon_{\Phi[\rho_A^+]}]
 \qquad\text{and}\qquad
 \Phi[\varepsilon_\Phi][\rho].
\]
The second formula is pulled back along the projection
$\rho:I\times B\to I$, hence is constant in the new $B$-coordinate. The
Beck--Chevalley clause gives
\[
 \Phi[\rho_A^+]
   [\varepsilon_{\Phi[\rho_A^+]}]
 \simeq_K
 \Phi[\varepsilon_\Phi][\rho].
\tag{$\varepsilon\mathrm{BC}_\rho$}
\]
\end{example}

\begin{remark}[Intensionality, \coqterm{EpsilonFullSubstitutivity.separating_model_refutes_substitutivity}]
In general,
\[
 \Phi\simeq_{I\times A}\Psi
 \quad\not\Longrightarrow\quad
 \varepsilon_\Phi\equiv_{I,A}\varepsilon_\Psi.
\]
Only the selected instances are identified. In particular,
$(\varepsilon\mathrm{BC}_\rho)$ imposes no equation
\[
 \varepsilon_{\Phi[\rho_A^+]}
 \equiv_{K,A}\varepsilon_\Phi\circ\rho.
\]
Adding the general implication would make $\varepsilon$ extensional with
respect to definitional formula equivalence. This is distinct from
$\varepsilon$-extensionality formulated with an internal equality predicate.
Concrete countermodels to extensionality and to general witness substitutivity are given in
Section~\ref{sec:witness-choice-substitution}.
(Cf.~\cite{Molnar2013}).
\end{remark}

\begin{remark}[The roles of the $\varepsilon$-laws, \coqterm{EpsilonFull.Hyper_Syntax.FE_Eps_Adjoint_Iso}]
The clause $(\varepsilon\Sigma)$ yields the comparison with existential
quantification. The clause $(\varepsilon\mathrm{BC})$ makes the selected
instance stable under syntactic substitution and is used in conversion
stability and in the $\varepsilon$-cases of
Proposition~\ref{prop:derivation-substitution}. Without it,
$\Phi[f_A^+][\varepsilon_{\Phi[f_A^+]}]$ and
$\Phi[\varepsilon_\Phi][f]$ may contain different labels.
\end{remark}

\begin{proposition}[Regularity and stability of conversion, \coqterm{EpsilonFull.Hyper_Syntax_Lemmas.form_equiv_wf}, \coqterm{EpsilonFull.Hyper_Syntax_Lemmas.form_equiv_subst}]
\label{prop:conversion-stability}
For $\mathrm{iFOL}_{\exists\varepsilon}$ over a fixed typed signature $\Sigma$, formula conversion is regular:
\[
 \Phi\simeq_I\Psi
 \quad\Longrightarrow\quad
 I\vdash\Phi\;\mathsf{form}
 \quad\text{and}\quad
 I\vdash\Psi\;\mathsf{form}.
\]
It is stable under syntactic substitution. For $I\vdash f:J$,
\[
 \Phi\simeq_J\Psi
 \quad\Longrightarrow\quad
 \Phi[f]\simeq_I\Psi[f].
\]
\end{proposition}

\begin{proof}
Both statements are proved by induction on the conversion derivation. The
$\varepsilon$-cases use $(\varepsilon\mathrm{BC})$.
\end{proof}

\subsection{Proof terms and derivability}

\begin{definition}[Proof terms, \coqterm{EpsilonFull.Hyper_Syntax.h_tm}]
\label{def:proof-terms}
Raw proof terms of $\mathrm{iFOL}_{\exists\varepsilon}$ over $\Sigma$ are defined inductively by
\[
\begin{aligned}
 p,q,r ::= {}&x\mid\mathsf{tt}\mid\lambda x.p\mid pq
 \mid\langle p,q\rangle\mid\mathsf{fst}(p)\mid\mathsf{snd}(p)\\
 &\mid\mathsf{inl}(p)\mid\mathsf{inr}(p)
 \mid\mathsf{case}(p,x.q,y.r)\mid\mathsf{abort}_{\Phi}(p)\\
 &\mid\mathsf{allI}(p)\mid\mathsf{allE}(p,a)
 \mid\mathsf{exI}(a,p)\mid\mathsf{exE}(p,u.q)\\
 &\mid\mathsf{epsI}(a,p)\mid\mathsf{epsE}(p,u.q).
\end{aligned}
\]
Here $a\in\mathsf{BTerm}_\Sigma^{\exists\varepsilon}$. The raw syntax for abstraction and
injections also carries the formula annotations determined by the typing
rules. These annotations are suppressed in the surface notation.
(Cf.~\cite{Prawitz1965}).
\end{definition}

\begin{definition}[Derivability, \coqterm{EpsilonFull.Hyper_Syntax.h_derivation}]
\label{def:proof-term-calculus}
For $\mathrm{iFOL}_{\exists\varepsilon}$ over $\Sigma$, the derivability judgement
\[
 I\mid\Gamma\vdash p:\Phi
\]
is defined inductively by the following rules. Every formula and logical
context displayed in a rule is required to be well formed in its displayed base context. In every quantifier or epsilon rule, $I\times A\vdash\Phi\;\mathsf{form}$ is required; $\Gamma$ and any elimination target $\Psi$ are well formed over $I$.

\medskip
\noindent\emph{Variables, truth, and falsity:}
\[
\prftree
 {x:\Phi\in\Gamma}
 {I\mid\Gamma\vdash x:\Phi}
\qquad
\prfaxiom{I\mid\Gamma\vdash\mathsf{tt}:\top}
\qquad
\prftree
 {I\mid\Gamma\vdash p:\bot}
 {I\mid\Gamma\vdash\mathsf{abort}_\Phi(p):\Phi}.
\]

\noindent\emph{Implication and conjunction:}
\[
\prftree
 {I\mid x:\Phi,\Gamma\vdash p:\Psi}
 {I\mid\Gamma\vdash\lambda x.p:\Phi\to\Psi}
\qquad
\prftree
 {I\mid\Gamma\vdash p:\Phi\to\Psi}
 {I\mid\Gamma\vdash q:\Phi}
 {I\mid\Gamma\vdash pq:\Psi},
\]
\[
\prftree
 {I\mid\Gamma\vdash p:\Phi}
 {I\mid\Gamma\vdash q:\Psi}
 {I\mid\Gamma\vdash\langle p,q\rangle:\Phi\land\Psi}
\qquad
\prftree
 {I\mid\Gamma\vdash p:\Phi\land\Psi}
 {I\mid\Gamma\vdash\mathsf{fst}(p):\Phi}
\qquad
\prftree
 {I\mid\Gamma\vdash p:\Phi\land\Psi}
 {I\mid\Gamma\vdash\mathsf{snd}(p):\Psi}.
\]

\noindent\emph{Disjunction:}
\[
\prftree
 {I\mid\Gamma\vdash p:\Phi}
 {I\mid\Gamma\vdash\mathsf{inl}(p):\Phi\lor\Psi}
\qquad
\prftree
 {I\mid\Gamma\vdash p:\Psi}
 {I\mid\Gamma\vdash\mathsf{inr}(p):\Phi\lor\Psi},
\]
\[
\prftree
 {I\mid\Gamma\vdash p:\Phi\lor\Psi}
 {I\mid x:\Phi,\Gamma\vdash q:\Xi}
 {I\mid y:\Psi,\Gamma\vdash r:\Xi}
 {I\mid\Gamma\vdash\mathsf{case}(p,x.q,y.r):\Xi}.
\]

\noindent\emph{Universal quantification:}
\[
\prftree
 {I\times A\mid\Gamma[\pi_1]\vdash p:\Phi}
 {I\mid\Gamma\vdash\mathsf{allI}(p):\forall_A\Phi}
\qquad
\prftree
 {I\mid\Gamma\vdash p:\forall_A\Phi}
 {I\vdash a:A}
 {I\mid\Gamma\vdash\mathsf{allE}(p,a):\Phi[a]}.
\]

\noindent\emph{Existential quantification:}
\[
\prftree
 {I\vdash a:A}
 {I\mid\Gamma\vdash p:\Phi[a]}
 {I\mid\Gamma\vdash\mathsf{exI}(a,p):\exists_A\Phi},
\]
\[
\prftree
 {I\mid\Gamma\vdash p:\exists_A\Phi}
 {I\times A\mid u:\Phi,\Gamma[\pi_1]\vdash q:\Psi[\pi_1]}
 {I\mid\Gamma\vdash\mathsf{exE}(p,u.q):\Psi}.
\]

\noindent\emph{Intensional $\varepsilon$:}
\[
\prftree
 {I\vdash a:A}
 {I\mid\Gamma\vdash p:\Phi[a]}
 {I\mid\Gamma\vdash\mathsf{epsI}(a,p):
  \Phi[\varepsilon_\Phi]},
\]
\[
\prftree
 {I\mid\Gamma\vdash p:\Phi[\varepsilon_\Phi]}
 {I\times A\mid u:\Phi,\Gamma[\pi_1]\vdash q:\Psi[\pi_1]}
 {I\mid\Gamma\vdash\mathsf{epsE}(p,u.q):\Psi}.
\]

\noindent\emph{Conversion:}
\[
\prftree
 {I\mid\Gamma\vdash p:\Phi}
 {\Phi\simeq_I\Psi}
 {I\mid\Gamma\vdash p:\Psi}
\qquad
\prftree
 {\Gamma\simeq_I\Delta}
 {I\mid\Gamma\vdash p:\Phi}
 {I\mid\Delta\vdash p:\Phi}.
\]
(Cf.~\cite{Prawitz1965}).
\end{definition}

\begin{remark}[Conversion and explicit $\varepsilon$-terms, \coqterm{EpsilonFull.Hyper_Syntax.FE_Eps_Adjoint_Iso}]
The clause $(\varepsilon\Sigma)$ places $\exists_A\Phi$ and
$\Phi[\varepsilon_\Phi]$ in the same conversion class. The rules
$\mathsf{epsI}$ and $\mathsf{epsE}$ retain explicit proof terms for its two
directions.
\end{remark}

\begin{definition}[Syntactic substitution of proof terms, \coqterm{EpsilonFull.Hyper_Syntax.tm_subst}]
\label{def:proof-term-substitution}
In $\mathrm{iFOL}_{\exists\varepsilon}$ over $\Sigma$, for $I\vdash f:J$, proof-term substitution $p\mapsto p[f]$ is defined by
structural recursion. It acts homomorphically on the propositional
term-formers and substitutes in every formula annotation. The quantifier and
$\varepsilon$ clauses are
\[
\begin{aligned}
 \mathsf{allI}(p)[f]
   &=\mathsf{allI}(p[f_A^+]),
&\mathsf{allE}(p,a)[f]
   &=\mathsf{allE}(p[f],a\circ f),\\
 \mathsf{exI}(a,p)[f]
   &=\mathsf{exI}(a\circ f,p[f]),
&\mathsf{exE}(p,u.q)[f]
   &=\mathsf{exE}(p[f],u.q[f_A^+]),\\
 \mathsf{epsI}(a,p)[f]
   &=\mathsf{epsI}(a\circ f,p[f]),
&\mathsf{epsE}(p,u.q)[f]
   &=\mathsf{epsE}(p[f],u.q[f_A^+]).
\end{aligned}
\]
\end{definition}

\begin{remark}[The role of proof-term substitution, \coqterm{EpsilonFull.Hyper_Syntax.tm_subst}]
The operation transports an entire derivation along a base substitution. It
is the term component of Proposition~\ref{prop:derivation-substitution} and
is interpreted by applying semantic substitution to the corresponding fibre
morphism.
\end{remark}

\begin{proposition}[Substitution, \coqterm{EpsilonFull_Completeness.h_derivation_subst_general}]
\label{prop:derivation-substitution}
In $\mathrm{iFOL}_{\exists\varepsilon}$ over any typed signature $\Sigma$, for every $J\vdash f:I$,
\[
 I\mid\Gamma\vdash p:\Phi
 \quad\Longrightarrow\quad
 J\mid\Gamma[f]\vdash p[f]:\Phi[f].
\]
\end{proposition}

\begin{proof}
Induction on the derivation.
\end{proof}

\begin{proposition}[Fiber trip, \coqterm{EpsilonFull.DualDrinker_Syntax.diagonal_base_equation}, \coqterm{EpsilonFull.DualDrinker_Syntax.FE_Pullback_Ext_Fst_Snd}]
\label{prop:diagonal-calculation}
Work in $\mathrm{iFOL}_{\exists\varepsilon}$ over $\Sigma$. Let $I,A\in\mathsf{BType}$, $J:=I\times A$, and $K:=(I\times A)\times A$.
Consider
$K\vdash\rho_A^+:=\langle\pi_1^J\circ\pi_1^K,\pi_2^K\rangle:J$
and
$J\vdash s_{\pi_2^J}:=\langle\mathrm{id}_J,\pi_2^J\rangle:K$.
Then, for every $J\vdash\Phi\;\mathsf{form}$,
\[
 \rho_A^+\circ s_{\pi_2^J}\equiv_{J,J}\mathrm{id}_J
 \quad\textup{(T)},
 \qquad
 \Phi[\rho_A^+][\pi_2^J]\simeq_J\Phi
 \quad\textup{(P)}.
\]
(Cf.~\cite[p.~139]{Jacobs1999}).
\end{proposition}
\begin{proof}
\leavevmode\par
\noindent
\begin{minipage}[t]{0.64\linewidth}
\centering
\(\displaystyle
\begin{aligned}
 \rho_A^+\circ s_{\pi_2^J}
 &=\langle\pi_1^J\circ\pi_1^K,\pi_2^K\rangle
   \circ s_{\pi_2^J}\\
 &\equiv
 \langle\pi_1^J\circ(\pi_1^K\circ s_{\pi_2^J}),
        \pi_2^K\circ s_{\pi_2^J}\rangle\\
 &=\langle
 \pi_1^J\circ(\pi_1^K\circ
   \langle\mathrm{id}_J,\pi_2^J\rangle),
 \pi_2^K\circ\langle\mathrm{id}_J,\pi_2^J\rangle
 \rangle\\
 &\equiv
 \langle\pi_1^J\circ\mathrm{id}_J,\pi_2^J\rangle\\
 &\equiv\langle\pi_1^J,\pi_2^J\rangle
 \equiv\mathrm{id}_J.
\end{aligned}
\)
\end{minipage}\hfill
\begin{minipage}[t]{0.33\linewidth}
\centering
\begin{tikzcd}[ampersand replacement=\&, column sep=6em, row sep=4em]
 \Phi
   \arrow[r, bend left=18, mapsto, "{[\rho_A^+]}"]
 \& \Phi[\rho_A^+]
   \arrow[l, bend left=18, mapsto,
     "{[s_{\pi_2^J}]=[\pi_2^J]}"]
 \\
 J
   \arrow[r, bend left=18, "s_{\pi_2^J}"]
 \& K
   \arrow[l, bend left=18, "\rho_A^+"]
\end{tikzcd}%
\end{minipage}\par\smallskip

\noindent Hence, by substitution composition, congruence, and the identity law,
\[
 \Phi[\rho_A^+][\pi_2^J]
 \simeq_J\Phi[\rho_A^+\circ s_{\pi_2^J}]
 \simeq_J\Phi[\mathrm{id}_J]\simeq_J\Phi.
\]

\end{proof}

\begin{proposition}[Existence and the selected instance, \coqterm{EpsilonFull.Examples.existence_to_selected_typing}, \coqterm{EpsilonFull.Examples.selected_to_existence_typing}]
\label{prop:epsilon-existence}
In $\mathrm{iFOL}_{\exists\varepsilon}$ over $\Sigma$, for all base types $I,A$ and every $I\times A\vdash\Phi\;\mathsf{form}$ there are proof terms
$e_\Phi$ and $i_\Phi$ such that
\[
 I\mid\varnothing\vdash
 e_\Phi:\exists_A\Phi\to\Phi[\varepsilon_\Phi],
 \qquad
 I\mid\varnothing\vdash
 i_\Phi:\Phi[\varepsilon_\Phi]\to\exists_A\Phi.
\]
Consequently,
\[
 \exists_A\Phi\dashv\vdash\Phi[\varepsilon_\Phi].
\]
(Cf.~\cite{HilbertBernays1939}, \cite[pp.~409, 414--415]{Pasquali2016}).
\end{proposition}

\begin{proof}
Write $J:=I\times A$, $K:=J\times A$, and
$K\vdash\rho_A^+:=\langle\pi_1^J\circ\pi_1^K,\pi_2^K\rangle:J$.
By \textup{(P)} of Proposition~\ref{prop:diagonal-calculation} and
$(\varepsilon\mathrm{BC})$,
\[
 \Phi[\rho_A^+][\pi_2^J]\simeq_J\Phi,
 \qquad
 \Phi[\rho_A^+][\varepsilon_{\Phi[\rho_A^+]}]
 \simeq_J\Phi[\varepsilon_\Phi][\pi_1^J].
\tag{*}
\]
Take
$e_\Phi:=\lambda z.\,\mathsf{exE}
 (z,u.\mathsf{epsI}(\pi_2^J,u))$ and
$i_\Phi:=\lambda z.\,\mathsf{exI}(\varepsilon_\Phi,z)$, and write
$\Gamma_z:=(z:\exists_A\Phi)$. In the first derivation, the two conversions
in $(*)$ type the premise and conclusion of $\mathsf{epsI}$, respectively:
\[
\prftree[r]{$\scriptstyle\to\mathrm I$}
 {
  \prftree[r]{$\scriptstyle\exists\mathrm E$}
   {I\mid\Gamma_z\vdash z:\exists_A\Phi}
   {J\mid u:\Phi,\Gamma_z[\pi_1^J]\vdash
    \mathsf{epsI}(\pi_2^J,u):
    \Phi[\varepsilon_\Phi][\pi_1^J]}
   {I\mid\Gamma_z\vdash
    \mathsf{exE}\bigl(z,u.\mathsf{epsI}(\pi_2^J,u)\bigr):
    \Phi[\varepsilon_\Phi]}
 }
 {I\mid\varnothing\vdash
  e_\Phi:\exists_A\Phi\to\Phi[\varepsilon_\Phi]},
\]
\[
\prftree[r]{$\scriptstyle\to\mathrm I$}
 {
  \prftree[r]{$\scriptstyle\exists\mathrm I$}
   {I\vdash\varepsilon_\Phi:A}
   {I\mid z:\Phi[\varepsilon_\Phi]\vdash z:\Phi[\varepsilon_\Phi]}
   {I\mid z:\Phi[\varepsilon_\Phi]\vdash
    \mathsf{exI}(\varepsilon_\Phi,z):\exists_A\Phi}
 }
 {I\mid\varnothing\vdash
  i_\Phi:\Phi[\varepsilon_\Phi]\to\exists_A\Phi}.
\]
\end{proof}

\begin{remark}[Inhabitational scope, \coqterm{EpsilonFull_Completeness.hom_fiber}]
No equality theory is imposed on proof terms. Completeness records only
whether such a term exists.
\end{remark}

\begin{example}[An existential prenex derivation, \coqterm{EpsilonFull.Examples.existential_prenex_typing}]
\label{ex:existential-prenex}
In $\mathrm{iFOL}_{\exists\varepsilon}$ over $\Sigma$, let $I,A\in\mathsf{BType}$,
$I\vdash\Psi\;\mathsf{form}$, and
$I\times A\vdash\Phi\;\mathsf{form}$. Then
\[
 I\mid((\exists_A\Phi)\to\Psi)
 \vdash\forall_A(\Phi\to\Psi^\uparrow),
\]
where $J:=I\times A$, $I\times A\vdash\pi_1^J:I$,
$I\times A\vdash\pi_2^J:A$, and
$\Xi^\uparrow:=\Xi[\pi_1^J]$ for every $I\vdash\Xi\;\mathsf{form}$.
Put $K:=(I\times A)\times A$ and
$K\vdash\pi_A^+:=\langle\pi_1^J\circ\pi_1^K,\pi_2^K\rangle:I\times A$.
The cartesian identity $\pi_A^+\circ\langle\mathrm{id}_J,\pi_2^J\rangle\equiv\mathrm{id}_J$ and substitution conversion give
\[
 \Phi[\pi_A^+][\pi_2^J]\simeq_{I\times A}\Phi,
\]
used below in the reverse direction. By the defining existential substitution clause,
\[
 (\exists_A\Phi)^\uparrow
 =(\exists_A\Phi)[\pi_1^J]
 =\exists_A(\Phi[\pi_A^+]).
\]
For
$\Gamma:=\bigl(u:\Phi,h:((\exists_A\Phi)\to\Psi)^\uparrow\bigr)$,
the derivation is
\[
\resizebox{\linewidth}{!}{%
\prftree[r]{$\scriptstyle\forall\mathrm I$}
 {
  \prftree[r]{$\scriptstyle\to\mathrm I$}
   {
    \prftree[r]{$\scriptstyle\to\mathrm E$}
     {
      \prftree[r]{$\scriptstyle\mathrm{var}$}
       {h:((\exists_A\Phi)\to\Psi)^\uparrow\in\Gamma}
       {I\times A\mid\Gamma\vdash
        h:(\exists_A\Phi)^\uparrow\to\Psi^\uparrow}
     }
     {
      \prftree[r]{$\scriptstyle\exists\mathrm I$}
       {I\times A\vdash\pi_2^J:A}
       {
        \prftree[r]{$\scriptstyle\mathrm{conv}$}
         {
          \prftree[r]{$\scriptstyle\mathrm{var}$}
           {u:\Phi\in\Gamma}
           {I\times A\mid\Gamma\vdash u:\Phi}
         }
         {\Phi\simeq_{I\times A}\Phi[\pi_A^+][\pi_2^J]}
         {I\times A\mid\Gamma\vdash
          u:\Phi[\pi_A^+][\pi_2^J]}
       }
       {I\times A\mid\Gamma\vdash
        \mathsf{exI}(\pi_2^J,u):
        \exists_A(\Phi[\pi_A^+])}
     }
     {I\times A\mid\Gamma\vdash
      h\,\mathsf{exI}(\pi_2^J,u):\Psi^\uparrow}
   }
   {I\times A\mid
    h:((\exists_A\Phi)\to\Psi)^\uparrow
    \vdash
    \lambda u.\,h\,\mathsf{exI}(\pi_2^J,u):
    \Phi\to\Psi^\uparrow}
 }
 {I\mid h:(\exists_A\Phi)\to\Psi
  \vdash
  \mathsf{allI}\bigl(\lambda u.\,
  h\,\mathsf{exI}(\pi_2^J,u)\bigr):
  \forall_A(\Phi\to\Psi^\uparrow)}
}%
\]
(Cf.~\cite{Prawitz1965}).
\end{example}

\begin{notation}[Intuitionistic intensional Hilbert $\varepsilon$-logic, \coqterm{EpsilonFull.Hyper_Syntax.h_derivation}]
For derivability in the calculus defined above, we henceforth write
$I\mid\Gamma\vdash_{\exists\varepsilon}p:\Phi$. We omit the proof term when
only inhabitation matters:
\[
 I\mid\Gamma\vdash_{\exists\varepsilon}\Phi
 \quad\Longleftrightarrow\quad
 \exists p\;I\mid\Gamma\vdash_{\exists\varepsilon}p:\Phi.
\]
\end{notation}

\subsection{The epsilon-free comparison calculus}

\begin{definition}[The $\mathrm{iFOL}$-language, \coqterm{IFOL.Base_Syntax.b_tm}, \coqterm{IFOL.Hyper_Syntax.h_ty}]
\label{def:ifol-language}
Fix a typed function--predicate signature $\Sigma$. The raw syntax of
$\mathcal L$ is given by
\[
\begin{aligned}
 t&::=\mathrm{id}\mid g\circ f\mid {!}\mid\langle f,g\rangle
       \mid\pi_1\mid\pi_2\mid\widehat u,\\
 \Phi&::=\top\mid\bot\mid R(t)\mid\Phi\land\Psi\mid\Phi\lor\Psi
       \mid\Phi\to\Psi\mid\forall_A\Phi\mid\exists_A\Phi.
\end{aligned}
\]
Here $A$ is a base type, $u\in\mathcal F$, and $R\in\mathcal R$.
There are no epsilon labels or labelled constants. Its typing and
formation rules are the cartesian base rules and the propositional and
quantifier formation rules. A logical context is a list of formulae
well formed over a common base context.
(Cf.~\cite[pp.~223--225]{Jacobs1999}, \cite{Prawitz1965}).
\end{definition}

\begin{definition}[The iFOL calculus, \coqterm{IFOL.Hyper_Syntax.h_derivation}]
\label{def:ifol-calculus}
Fix a typed signature $\Sigma$. The calculus $\mathrm{iFOL}$ over $\mathcal L$ has the assumption, truth, falsity, implication, conjunction, disjunction,
universal and existential introduction and elimination, and conversion
rules. Its formula conversion is the least equivalence on well-formed
formulae closed under atomic base equivalence, propositional and quantifier
congruence, and identity, composition, and congruence of substitution. It has no $\varepsilon$-terms or
$\varepsilon$-rules.
The resulting derivability judgement is written
$I\mid\Gamma\vdash_{\mathrm{iFOL}}p:\Phi$.

We omit the proof term by writing
\[
 I\mid\Gamma\vdash_{\mathrm{iFOL}}\Phi
 \quad\Longleftrightarrow\quad
 \exists p\;I\mid\Gamma\vdash_{\mathrm{iFOL}} p:\Phi.
\]
(Cf.~\cite{Prawitz1965}).
\end{definition}

\section{Hyperdoctrinal semantics}
\label{sec:hyperdoctrinal-semantics}

The semantics separates base terms from formulae by a thin, BiCCC-valued
hyperdoctrine in the sense of Lawvere~\cite{Lawvere1969,Lawvere1970}.

\subsection{Interpreting the base language}

\begin{definition}[Base model, \coqterm{BaseEpsilonFullModel}, \coqterm{EpsilonFull_Fibration.CoherentBaseModel}]
\label{def:base-model}
For the base language of $\mathcal L_{\exists\varepsilon}$ over a typed signature $\Sigma$, a base model is
$(\mathcal B,S_{\mathcal B},\sem{-})$, if it consists of a cartesian setoid category
$\mathcal B$, with a distinguished object
\[
 S_{\mathcal B}\in\operatorname{Ob}(\mathcal B),
 \qquad S_{\mathcal B}\ne\mathbf 1_{\mathcal B},
\]
where $\mathbf 1_{\mathcal B}$ is the chosen terminal object, and the following interptetation. The object interpretation is the map
\[
 \sem{-}_0:\mathsf{BType}\longrightarrow\operatorname{Ob}(\mathcal B)
\]
defined by structural recursion on base types:
\[
 \sem{\mathbf 1}=\mathbf 1_{\mathcal B},\qquad
 \sem{S}=S_{\mathcal B},\qquad
 \sem{A\times B}=\sem{A}\times\sem{B}.
\]
The morphism part of the interpretation is the following. For every $I,A\in\mathsf{BType}$ and typing derivation $d$ of the judgement $I\vdash t:A$, let
\[
 \sem{t}_d\in
 \operatorname{Hom}_{\mathcal B}(\sem{I},\sem{A})
\]
defined by recursion on $d$ as
\[
\begin{gathered}
 \sem{\mathrm{id}_I}=\mathrm{id}_{\sem{I}},
 \qquad
 \sem{g\circ f}=\sem{g}\circ\sem{f},
 \qquad
 \sem{{!}_I}={!}_{\sem{I}},\\
 \sem{\langle f,g\rangle}=\langle\sem{f},\sem{g}\rangle,
 \qquad
 \sem{\pi_1}=\pi_1,
 \qquad
 \sem{\pi_2}=\pi_2,\\
 \sem{c_s^A}=\iota^A_{I,s},
 \qquad
 \sem{\widehat u}=\iota_u.
\end{gathered}
\]
with modelwise fixed
\[
 \iota^A_{I,s}
 \in\operatorname{Hom}_{\mathcal B}(\sem{I},\sem{A}),
 \qquad
 \iota_u
 \in\operatorname{Hom}_{\mathcal B}
 \bigl(\sem{\operatorname{dom}(u)},\sem{\operatorname{cod}(u)}\bigr).
\]
where $s\in\mathsf{Lab}_\Sigma^{\exists\varepsilon}$, and $u\in\mathcal F$.
We require independence of the typing derivation: for any two derivations $d_1,d_2$ of
$I\vdash t:A$, one has $\sem{t}_{d_1}\sim_{\mathcal B}\sem{t}_{d_2}$.
We then write $\sem{t}$ for the resulting interpretation of a base term.
For $\mathcal L$ omit the labelled-constant interpretation.
The same independence condition applies.
We write $\mathsf{BaseMod}(\Sigma)$ for the class of all triples
$(\mathcal B,S_{\mathcal B},\sem{-})$ of this form.
(Cf.~\cite[pp.~147--148]{Jacobs1999}).
\end{definition}

\begin{remark}[Typing derivations, \coqterm{EpsilonFull_Fibration.CoherentBaseModel}] The interpretation is required to be independent of the typing derivation: if
\(
 d_1,d_2\) derivations of \(I\vdash t:A\), then
 \(
 \sem{t}_{d_1}\sim_{\mathcal B}\sem{t}_{d_2}.
\)
Rocq defines \texttt{interp\_tm} by recursion on derivation $d$ of $I\vdash t:A$. The class \texttt{CoherentBaseModel} records the above independence up to $\sim_{\mathcal B}$. In this paper this property is already part of being a base model.
\end{remark}

\subsection{The structure of a first-order hyperdoctrine}

\begin{definition}[First-order hyperdoctrine, \coqterm{GoodFibration}]
\label{def:goodfibration}
Let $\mathcal B$ be a cartesian setoid category. A first-order
hyperdoctrine over $\mathcal B$ consists of the following data and laws.
\begin{enumerate}
\item For every $X\in\operatorname{Ob}(\mathcal B)$, the fibre $P(X)$ is a
thin bicartesian closed setoid category.

\item Every $f:X\to Y$ induces a semantic substitution functor
\[
 P(f):P(Y)\longrightarrow P(X).
\]
For $X\xrightarrow{f}Y\xrightarrow{g}Z$,
$\Phi\in P(X)$, and $\Theta\in P(Z)$, there are isomorphisms
\[
 P(\mathrm{id}_X)(\Phi)\cong\Phi,
 \qquad
 P(g\circ f)(\Theta)\cong P(f)(P(g)(\Theta)).
\]

\item Semantic substitution preserves the bicartesian closed structure. For
$f:X\to Y$ and $\Phi,\Psi\in P(Y)$,
\[
\begin{gathered}
 P(f)(\top_Y)\cong\top_X,
 \qquad
 P(f)(\bot_Y)\cong\bot_X,\\
 P(f)(\Phi\times\Psi)\cong P(f)(\Phi)\times P(f)(\Psi),\\
 P(f)(\Phi+\Psi)\cong P(f)(\Phi)+P(f)(\Psi),\\
 P(f)(\Psi^\Phi)\cong P(f)(\Psi)^{P(f)(\Phi)}.
\end{gathered}
\]

\item For $X,A\in\operatorname{Ob}(\mathcal B)$ and the projection
$\pi_{X,A}:X\times A\to X$, semantic substitution has both adjoints:
\[
 \Sigma_{\pi_{X,A}}\dashv P(\pi_{X,A})
 \dashv\Pi_{\pi_{X,A}}.
\]

\item Let $f:K\to I$, let $A\in\operatorname{Ob}(\mathcal B)$, let
$\pi_I:I\times A\to I$ and $\pi_K:K\times A\to K$ be the projections, and define
\[
 f_A^+:=\langle f\circ\pi_1,\pi_2\rangle:K\times A\longrightarrow I\times A.
\]
For every $\Phi\in P(I\times A)$, the Beck--Chevalley isomorphisms are
\[
 P(f)(\Sigma_{\pi_I}\Phi)
 \cong
 \Sigma_{\pi_K}(P(f_A^+)\Phi),
 \qquad
 P(f)(\Pi_{\pi_I}\Phi)
 \cong
 \Pi_{\pi_K}(P(f_A^+)\Phi).
\tag{BC}
\]
\[
\begin{tikzcd}[row sep=3.5em, column sep=5em]
 K\times A \arrow[r,"f_A^+"] \arrow[d,"\pi_K"']
   & I\times A \arrow[d,"\pi_I"]\\
 K \arrow[r,"f"'] & I.
\end{tikzcd}
\]

\item Semantic substitution respects the base hom-setoids. If
$f\sim_{\mathcal B}g:X\to Y$, then for every $\Phi\in P(Y)$,
\[
 P(f)(\Phi)\cong P(g)(\Phi).
\tag{PB-cong}
\]
\end{enumerate}
(Cf.~\cite{Lawvere1970}, \cite[pp.~233--235]{Jacobs1999}, \cite{Seely1983}).
\end{definition}

\begin{remark}[Setoid-functoriality, \coqterm{BiCCCFunctor_equiv}]
For bicartesian closed functors $F,G:P(X)\to P(Y)$, write
\[
 F\sim G
 \quad\Longleftrightarrow\quad
 F(\Phi)\cong G(\Phi)
 \quad\text{for every }\Phi\in P(X).
\]
With this hom relation, the preceding data determine a setoid functor
\[
 P:\mathcal B^{\mathrm{op}}\longrightarrow\mathsf{BiCCC}.
\]
The preservation of bicartesian closed structure is up to the displayed
isomorphisms.
(Cf.~\cite[pp.~48--51]{Jacobs1999}).
\end{remark}

\begin{remark}[Quantifier adjunctions, \coqterm{pullback_has_left_adjoint}, \coqterm{pullback_has_right_adjoint}]
In a thin first-order hyperdoctrine $P$ over a cartesian setoid category $\mathcal B$, let $X,A\in\operatorname{Ob}(\mathcal B)$ and
$\pi:=\pi_{X,A}:X\times A\to X$. For $\Phi\in P(X\times A)$ and
$\Psi\in P(X)$, the adjunctions give
\[
 \operatorname{Hom}_{P(X)}(\Sigma_\pi\Phi,\Psi)
 \cong
 \operatorname{Hom}_{P(X\times A)}(\Phi,P(\pi)\Psi),
\]
\[
 \operatorname{Hom}_{P(X\times A)}(P(\pi)\Psi,\Phi)
 \cong
 \operatorname{Hom}_{P(X)}(\Psi,\Pi_\pi\Phi).
\]
These interpret existential and universal quantification, respectively.
(Cf.~\cite{Lawvere1969}, \cite[pp.~230, 233--235]{Jacobs1999}).
\end{remark}

\begin{example}[$H$-valued predicate hyperdoctrine]
\label{ex:h-valued-predicate-hyperdoctrine}
Fix a set $S$ and let $\mathcal B_S$ be the full cartesian subcategory of
$\mathsf{Set}$ whose objects are generated by the singleton $\mathbf 1$, the
set $S$, and binary products. No finiteness assumption is made on $S$. Let
$H$ be a finite, hence complete, Heyting algebra. For every
$X\in\operatorname{Ob}(\mathcal B_S)$, let the objects of the fibre $P(X)$
be the $H$-valued predicates on $X$:
\[
 \operatorname{Ob}(P(X)):=H^X=\{\Phi\mid\Phi:X\to H\}.
\]
For $\Phi,\Psi\in P(X)$, define
\[
 \Phi\leq\Psi
 \quad\Longleftrightarrow\quad
 \forall x\in X\;\Phi(x)\leq_H\Psi(x).
\]
Thus $P(X)$ is the thin category determined by the pointwise order. Its
terminal object, initial object, products, coproducts, and exponentials are
the pointwise operations $\top$, $\bot$, $\wedge$, $\vee$, and
$\Rightarrow$ of $H$. Hence thinness reduces the bicartesian closed
structure of the fibre to its Heyting-algebra structure.

For $X,Y\in\operatorname{Ob}(\mathcal B_S)$, a function $f:X\to Y$, and an object
$\Psi\in P(Y)=H^Y$, semantic substitution is precomposition:
\[
 P(f)(\Psi):=\Psi\circ f\in P(X).
\]
For $X,A\in\operatorname{Ob}(\mathcal B_S)$, let
$\pi:X\times A\to X$ be the projection. If
$\Phi\in P(X\times A)=H^{X\times A}$, the two adjoints are given by
\[
 (\Sigma_\pi\Phi)(x):=\bigvee_{a\in A}\Phi(x,a),
 \qquad
 (\Pi_\pi\Phi)(x):=\bigwedge_{a\in A}\Phi(x,a).
\]
Thus, for $\Psi\in P(X)$,
\[
 \Sigma_\pi\Phi\leq\Psi
 \Longleftrightarrow
 \Phi\leq P(\pi)\Psi,
 \qquad
 P(\pi)\Psi\leq\Phi
 \Longleftrightarrow
 \Psi\leq\Pi_\pi\Phi.
\]
The Beck--Chevalley equations hold pointwise.
(Cf.~\cite[p.~239]{Jacobs1999}, \cite[p.~159]{MaiettiPasqualiRosolini2017}).
\end{example}

\subsection{Interpreting the fixed language}

Fix a typed function--predicate signature $\Sigma$, its language
$\mathcal L_{\exists\varepsilon}$, a base model
$(\mathcal B,S_{\mathcal B},\sem{-})$, and a
first-order hyperdoctrine $P$ over $\mathcal B$.

\begin{notation}[Semantic substitution, \coqterm{Pullback}]
Fix a base model of $\mathcal L_{\exists\varepsilon}$ over $\Sigma$ and a thin first-order hyperdoctrine $P$ over its base. Write $\Phi[f]$ for syntactic substitution. For a well-typed base substitution $I\vdash f:J$, write
\[
 f^*:=P(\sem{f}):P(\sem{J})\longrightarrow P(\sem{I}).
\]
Thus $\sem{f}:\sem{I}\to\sem{J}$ interprets the base substitution, while
$f^*$ is its contravariant action on formula objects and fibre morphisms.
\end{notation}

\begin{remark}[Typing derivations, \coqterm{EpsilonFull_Fibration.CoherentBaseModel}]
In Rocq, $f^*$ is
\texttt{Pullback (interp\_tm $d_f$)} for a derivation
$d_f:I\vdash f:J$. Independence of $\sem{f}$ from $d_f$ up to
$\sim_{\mathcal B}$, together with \textup{(PB-cong)}, makes its action
independent, up to fibre isomorphism, of the choice of $d_f$.
\end{remark}

\begin{definition}[Atomic interpretation, \coqterm{EpsilonFull_Fibration.EpsilonHyperModel}]
\label{def:predicate-interpretation}
Fix a typed signature $\Sigma$, a base model of $\mathcal L_{\exists\varepsilon}$, and a thin first-order hyperdoctrine $P$ over its cartesian base $\mathcal B$. An atomic interpretation is a family
\[
 \mathcal A_I(R,t)\in\operatorname{Ob}(P(\sem{I}))
 \qquad
 (I\in\mathsf{BType},\ R\in\mathcal R,\ t\text{ a raw base term})
\]
satisfying the following conditions. For $I,J\in\mathsf{BType}$,
$R\in\mathcal R$, every raw base term $t$, and every $I\vdash f:J$,
\[
 f^*(\mathcal A_J(R,t))
 \cong
 \mathcal A_I(R,t\circ f).
\tag{At-subst}
\]
For $I,A\in\mathsf{BType}$, $R\in\mathcal R$, and $I\vdash t,u:A$,
\[
 t\equiv_{I,A}u
 \quad\Longrightarrow\quad
 \mathcal A_I(R,t)\cong\mathcal A_I(R,u).
\tag{At-cong}
\]
\end{definition}

\begin{remark}[Raw atomic data, \coqterm{EpsilonFull_Fibration.i_atom}]
The family $\mathcal A$ and \textup{(At-subst)} are defined for raw base
terms. In a well-formed atom $R(t)$, one has
$I\vdash t:\operatorname{ar}(R)$.
\end{remark}

\begin{definition}[Interpretation of formulae and logical contexts, \coqterm{EpsilonFull_Fibration.interp_h_ty}, \coqterm{EpsilonFull_Fibration.interp_h_ctx}]
\label{def:formula-context-interpretation}
For $\mathcal L_{\exists\varepsilon}$ over $\Sigma$, fix a base model, a thin first-order hyperdoctrine $P$, and raw atomic data $\mathcal A$ satisfying substitution and base congruence. Define
$\sem{\Phi}_I\in\operatorname{Ob}(P(\sem{I}))$ by structural recursion:
\[
\begin{gathered}
 \sem{\top}_I:=\top_{P(\sem{I})},
 \qquad
 \sem{\bot}_I:=\bot_{P(\sem{I})},
 \qquad
 \sem{R(t)}_I:=\mathcal A_I(R,t),\\
 \sem{\Phi\land\Psi}_I:=\sem{\Phi}_I\times\sem{\Psi}_I,
 \qquad
 \sem{\Phi\lor\Psi}_I:=\sem{\Phi}_I+\sem{\Psi}_I,\\
 \sem{\Phi\to\Psi}_I:=\sem{\Psi}_I^{\sem{\Phi}_I}.
\end{gathered}
\]
Let
\[
 \pi_{I,A}:\sem{I}\times\sem{A}\longrightarrow\sem{I}
\]
be the first projection. The quantifier clauses are
\[
 \sem{\forall_A\Phi}_I
 :=\Pi_{\pi_{I,A}}(\sem{\Phi}_{I\times A}),
 \qquad
 \sem{\exists_A\Phi}_I
 :=\Sigma_{\pi_{I,A}}(\sem{\Phi}_{I\times A}).
\]
Logical contexts are interpreted by
\[
 \sem{\varnothing}_I:=\top_{P(\sem{I})},
 \qquad
 \sem{\Phi,\Gamma}_I:=\sem{\Gamma}_I\times\sem{\Phi}_I.
\]
For $I\vdash\Phi\;\mathsf{form}$ and
$I\vdash\Gamma\;\mathsf{ctx}$, we suppress the subscript $I$.
(Cf.~\cite{Lawvere1970}, \cite[pp.~246--248]{Jacobs1999}).
\end{definition}

\begin{remark}[Order of assumptions, \coqterm{EpsilonFull_Fibration.interp_h_ctx}]
In $\Phi,\Gamma$, the new assumption $\Phi$ is the head of the Rocq list but
the second product factor. Thus $\pi_2$ interprets the newest assumption and
$\pi_1$ discards it.
\end{remark}

\begin{proposition}[Semantic substitution lemma, \coqterm{EpsilonFull_Soundness.form_subst_iso}]
\label{prop:semantic-substitution}
Fix a base model of $\mathcal L_{\exists\varepsilon}$ over $\Sigma$, a thin first-order hyperdoctrine $P$, and raw atomic data satisfying \textup{(At-subst)} and \textup{(At-cong)}. If $I\vdash f:J$ and $\Phi$ is a raw formula of this language, then
\[
 \sem{\Phi[f]}_I
 \cong
 f^*(\sem{\Phi}_J).
\tag{Subst}
\]
(Cf.~\cite[pp.~48--49, 226]{Jacobs1999}).
\end{proposition}

\begin{proof}
Induction on $\Phi$. \end{proof}

\subsection{The intensional Hilbert
\texorpdfstring{$\varepsilon$}{epsilon}-structure}

The interpretation fixed above already assigns a base morphism to every
syntactic term $\varepsilon_\Phi$. The additional structure relates its
selected instance to existential quantification.

\begin{notation}[The selected right inverse morphism, \coqterm{epsilon_section}]
\label{note:selected right inverse}
In a base model of $\mathcal L_{\exists\varepsilon}$ over $\Sigma$, let $I,A\in\mathsf{BType}$ and
$I\times A\vdash\Phi\;\mathsf{form}$. Then let
\[
 s_\Phi
 :=\langle\mathrm{id}_{\sem{I}},\sem{\varepsilon_\Phi}\rangle
 :\sem{I}\longrightarrow\sem{I}\times\sem{A}.
\]
\end{notation}

\begin{remark}[Categorical sections, \coqterm{epsilon_section}]
A section of a morphism $p:E\to B$ is a morphism $s:B\to E$ such that
$p\circ s=\mathrm{id}_B$. Thus $s$ is a right inverse of $p$. In a setoid
category, equality is replaced by the hom-setoid relation. In the present
case, the relevant arrows are
\[
\begin{tikzcd}[column sep=7em]
 \sem{I}
   \arrow[r, bend left=18, "s_\Phi"]
 & \sem{I}\times\sem{A}
   \arrow[l, bend left=18, "\pi_{I,A}"]
\end{tikzcd}
\]
(Cf.~\cite{LambekScott1986}).
\end{remark}

\begin{proposition}[Interpretation of the selected instance, \coqterm{EpsilonFull_SelectionModels.selected_section_projection}, \coqterm{EpsilonFull_SelectionModels.selected_instance_interpretation}]
\label{prop:selected-instance}
Fix a base model of $\mathcal L_{\exists\varepsilon}$ over $\Sigma$, a thin first-order hyperdoctrine $P$, and an atomic interpretation satisfying substitution and base congruence. For $I\times A\vdash\Phi\;\mathsf{form}$, put $s_\Phi=\langle\mathrm{id},\sem{\varepsilon_\Phi}\rangle$ and let $\pi_{I,A}:\sem{I}\times\sem{A}\to\sem{I}$ be the first projection. Then
\[
 \pi_{I,A}\circ s_\Phi\sim_{\mathcal B}\mathrm{id}_{\sem{I}},
 \qquad
 \sem{\Phi[\varepsilon_\Phi]}_I
 \cong
 P(s_\Phi)(\sem{\Phi}_{I\times A}).
\tag{E-subst}
\]
\end{proposition}

\begin{proof}
The first statement is the first product equation. For the second,
\[
\begin{aligned}
 \sem{\Phi[\varepsilon_\Phi]}_I
 &=\sem{\Phi[\langle\mathrm{id}_I,\varepsilon_\Phi\rangle]}_I\\
 &\cong
 P\bigl(\sem{\langle\mathrm{id}_I,\varepsilon_\Phi\rangle}\bigr)
   (\sem{\Phi}_{I\times A})\\
 &\cong P(s_\Phi)(\sem{\Phi}_{I\times A}).
\end{aligned}
\]
The two isomorphisms are \textup{(Subst)} and \textup{(PB-cong)}.
\end{proof}

\begin{remark}[Weakening and instantiation, \coqterm{EpsilonFull_Soundness.interp_subst_comp}]
The section equation and contravariant functoriality give
\[
 P(s_\Phi)\bigl(P(\pi_{I,A})(X)\bigr)\cong X
 \qquad(X\in\operatorname{Ob}(P(\sem{I}))).
\]
Thus $P(\pi_{I,A})$ weakens a formula and $P(s_\Phi)$ substitutes the
selected term.
(Cf.~\cite[pp.~48--49]{Jacobs1999}).
\end{remark}

\begin{definition}[Intensional Hilbert $\varepsilon$-structure, \coqterm{AdjointEpsilonFullFibration}]
\label{def:intensional-epsilon-structure}
Let $P$ be a first-order hyperdoctrine over a
base model $(\mathcal B,S_{\mathcal B},\sem{-})$, equipped with an atomic
interpretation of $\mathcal L_{\exists\varepsilon}$ over a typed function--predicate signature $\Sigma$. The base model
interprets the labelled constants, hence provides
$\sem{\varepsilon_\Phi}:\sem{I}\to\sem{A}$ for every
$I\times A\vdash\Phi\;\mathsf{form}$. An intensional Hilbert
$\varepsilon$-structure on $P$ is a family of isomorphisms
\[
 \Sigma_{\pi_{I,A}}(\sem{\Phi}_{I\times A})
 \cong
 \sem{\Phi[\varepsilon_\Phi]}_I
\tag{$\varepsilon\Sigma$}
\]
indexed by $I\times A\vdash\Phi\;\mathsf{form}$.
(Cf.~\cite[pp.~414--415]{Pasquali2016}, \cite[p.~452]{Pasquali2019}).
\end{definition}

\begin{definition}[Thin $\varepsilon$-hyperdoctrine, \coqterm{EpsilonFullModel}]
\label{def:intensional-epsilon-hyperdoctrine}
For a typed signature $\Sigma$, a model of $\mathrm{iFOL}_{\exists\varepsilon}$,
also called a thin $\varepsilon$-hyperdoctrine, consists of a cartesian base model of $\mathcal L_{\exists\varepsilon}$, a thin
first-order hyperdoctrine $P$ over that base, and raw atomic data satisfying
substitution and base congruence. For every well-formed
$I\times A\vdash\Phi\;\mathsf{form}$ it has an isomorphism
$\Sigma_{\pi_{I,A}}\sem{\Phi}\cong\sem{\Phi[\varepsilon_\Phi]}$.
The base interpretation supplies the selected term
$\sem{\varepsilon_\Phi}:\sem{I}\to\sem{A}$.
(Cf.~\cite[pp.~414--415]{Pasquali2016}).
\end{definition}

\begin{remark}[Rocq packaging, \coqterm{EpsilonAsSigmaFibration}]
The class \texttt{EpsilonFullModel} bundles the data of
Definition~\ref{def:intensional-epsilon-hyperdoctrine}. Its
$\varepsilon$-component is \texttt{AdjointEpsilonFullFibration}, and
\texttt{eps\_is\_left\_adj\_obj} supplies
\textup{$(\varepsilon\Sigma)$}.
The alternative class \texttt{EpsilonAsSigmaFibration} supplies a second
left adjoint $L_{I,A}\dashv P(\pi_{I,A})$ together with
\[
 L_{I,A}(\sem{\Phi}_{I\times A})
 \cong\sem{\Phi[\varepsilon_\Phi]}_I.
\]
Uniqueness of left adjoints gives
$\Sigma_{\pi_{I,A}}(\sem{\Phi}_{I\times A})
\cong L_{I,A}(\sem{\Phi}_{I\times A})$, hence
\textup{$(\varepsilon\Sigma)$}.
\end{remark}

\begin{remark}[Intensionality, \coqterm{EpsilonFullSubstitutivity.separating_model_refutes_substitutivity}]
The morphism $\sem{\varepsilon_\Phi}$ interprets a base constant labelled by
the raw formula $\Phi$. The model is not required to select a morphism for every object
of $P(\sem{I\times A})$. Moreover,
\[
 \sem{\Phi}_{I\times A}\cong\sem{\Psi}_{I\times A}
 \quad\not\Longrightarrow\quad
 \sem{\varepsilon_\Phi}\sim_{\mathcal B}\sem{\varepsilon_\Psi}.
\]
No equality of the selected morphisms is required. This is the semantic
form of the intensionality imposed in the syntax.
(Cf.~\cite{Molnar2013}).
\end{remark}

\begin{proposition}[Interpretation of formula conversion, \coqterm{EpsilonFull_Soundness.interp_form_equiv}]
\label{prop:semantic-formula-conversion}
In every thin $\varepsilon$-hyperdoctrine $\mathcal M$ interpreting $\mathcal L_{\exists\varepsilon}$ over $\Sigma$, if $\Phi\simeq_I\Psi$ is a conversion of $\mathrm{iFOL}_{\exists\varepsilon}$, then
\[
 \sem{\Phi}_I\cong\sem{\Psi}_I.
\tag{Form-cong}
\]
\end{proposition}

\begin{proof}
Induction on the conversion derivation. Atomic conversion uses
\textup{(At-cong)}. The connective and quantifier congruences use the
corresponding functors. The principal $\varepsilon$-case is
\textup{$(\varepsilon\Sigma)$}. The three substitution cases are
\textup{(Subst-id)}, \textup{(Subst-comp)}, and \textup{(Subst-cong)}.

For the $\varepsilon$-Beck--Chevalley case, let $I\vdash f:J$ and
$J\times A\vdash\Phi\;\mathsf{form}$. Then
\[
\begin{aligned}
 \sem{\Phi[f_A^+][\varepsilon_{\Phi[f_A^+]}]}_I
 &\cong\Sigma_{\pi_{I,A}}(\sem{\Phi[f_A^+]}_{I\times A})\\
 &\cong\Sigma_{\pi_{I,A}}
   \bigl((f_A^+)^*(\sem{\Phi}_{J\times A})\bigr)\\
 &\cong f^*\bigl(\Sigma_{\pi_{J,A}}(\sem{\Phi}_{J\times A})\bigr)\\
 &\cong f^*(\sem{\Phi[\varepsilon_\Phi]}_J)\\
 &\cong\sem{\Phi[\varepsilon_\Phi][f]}_I.
\end{aligned}
\]
The five comparisons are the inverse of
\textup{$(\varepsilon\Sigma)$}, \textup{(Subst)}, the inverse of
Beck--Chevalley, semantic substitution applied to
\textup{$(\varepsilon\Sigma)$}, and the
inverse of \textup{(Subst)}.
\end{proof}

\begin{remark}[Interpretation of context conversion, \coqterm{EpsilonFull_Soundness.interp_ctx_equiv}]
Proposition~\ref{prop:semantic-formula-conversion} extends componentwise to
logical contexts: if $\Gamma\simeq_I\Delta$, then
\[
 \sem{\Gamma}_I\cong\sem{\Delta}_I.
\]
This follows by induction on context conversion and product congruence.
\end{remark}

\begin{remark}[Use in soundness, \coqterm{EpsilonFull_Soundness.soundness}]
The forward direction of \textup{$(\varepsilon\Sigma)$} interprets
$\varepsilon$-introduction. Its inverse turns a selected instance into an
existential object and is used in $\varepsilon$-elimination.
\end{remark}

\subsection{The model classes and categorical choice}

\begin{proposition}[Three forms of the selection condition, \coqterm{epsilon_selection_unit_equivalence}, \coqterm{epsilon_selection_unit_iff_inhabited_adjunction}]
\label{prop:selection-characterization}
Let $P$ be a thin first-order hyperdoctrine over a cartesian setoid
category $\mathcal B$. Fix objects $X,A$, a predicate
$Q\in P(X\times A)$, and an arrow $e:X\to A$. Put
$\pi=\pi_{X,A}$, $s=\langle\mathrm{id}_X,e\rangle$, and $E=P(s)Q$.
The following conditions are equivalent:
\begin{enumerate}
\item there is a fibre arrow $Q\to P(\pi)E$;
\item there is a fibre isomorphism $\Sigma_\pi Q\cong E$;
\item for every object $R\in P(X)$,
\[
 \operatorname{Hom}_{P(X)}(E,R)\text{ is inhabited}
 \quad\Longleftrightarrow\quad
 \operatorname{Hom}_{P(X\times A)}(Q,P(\pi)R)\text{ is inhabited}.
\]
\end{enumerate}
This is the selection form of the existential adjunction.
(Cf.~\cite[pp.~414--415]{Pasquali2016}, \cite[p.~452]{Pasquali2019}).
\end{proposition}

\begin{proof}
An arrow in (1) transposes to $\Sigma_\pi Q\to E$.
The other direction always exists: reindex the adjunction unit
$Q\to P(\pi)\Sigma_\pi Q$ along $s$. Since $\pi s\sim\mathrm{id}_X$,
this gives $E\to\Sigma_\pi Q$. Thinness makes the two arrows inverse.
Conversely, the adjunction unit followed by the reindexing of (2) gives
(1). The adjunction and (2) give (3), and (3) applied to $R=E$ and its
identity arrow gives (1).
\end{proof}

\begin{definition}[iFOL-hyperdoctrine, \coqterm{IFOLHyperModel}]
\label{def:ifol-model}
Fix a typed signature $\Sigma$. An iFOL-hyperdoctrine consists of a
cartesian base model for $\mathcal L$, a thin first-order
hyperdoctrine $P$ over its base category, and raw atomic data satisfying
substitution and base congruence.
(Cf.~\cite[pp.~233--235]{Jacobs1999}).
\end{definition}

\begin{corollary}[The model class of $\mathrm{iFOL}_{\exists\varepsilon}$, \coqterm{EpsilonFull_SelectionModels.EpsilonFullModel_from_chosen_units}]
\label{cor:full-model-class}
Fix a typed signature $\Sigma$, a base model
$(\mathcal B,S_{\mathcal B},\sem{-})$ of $\mathcal L_{\exists\varepsilon}$,
a thin first-order hyperdoctrine $P$ over $\mathcal B$, and raw atomic
data satisfying substitution and base congruence. For every well-formed
$I\times A\vdash\Phi\;\mathsf{form}$, put
$Q_\Phi=\sem{\Phi}$,
$s_\Phi=\langle\mathrm{id},\sem{\varepsilon_\Phi}\rangle$, and
$E_\Phi=P(s_\Phi)Q_\Phi\cong\sem{\Phi[\varepsilon_\Phi]}$.
The additional structure of an $\mathrm{iFOL}_{\exists\varepsilon}$-model
can equivalently be supplied by a chosen arrow
$Q_\Phi\to P(\pi_{I,A})E_\Phi$ for each such $\Phi$, or by a chosen
isomorphism $\Sigma_{\pi_{I,A}}Q_\Phi\cong E_\Phi$.
No equality or naturality of the witnesses $\sem{\varepsilon_\Phi}$ is required.
\end{corollary}

\begin{proof}
The proof of Proposition~\ref{prop:selection-characterization} constructs
arrows in both directions from a chosen unit and hence an isomorphism,
using thinness. Conversely, compose the existential unit with the
reindexing of the chosen isomorphism.
\end{proof}

\begin{notation}[Validity for iFOL, \coqterm{ifol_completeness_inhabitational_gen}]
Fix a typed signature $\Sigma$ and an iFOL-hyperdoctrine $\mathcal M$
interpreting $\mathcal L$. For a well-formed context
$I\vdash\Gamma\;\mathsf{ctx}$ and formula
$I\vdash\Phi\;\mathsf{form}$ in this language, write
\[
 I\mid\Gamma\models_{\mathrm{iFOL}}^{\mathcal M}\Phi
 \quad\Longleftrightarrow\quad
 \operatorname{Hom}_{P_{\mathcal M}(\sem{I}_{\mathcal M})}
 \bigl(\sem{\Gamma}_{\mathcal M},\sem{\Phi}_{\mathcal M}\bigr)
 \text{ is inhabited}.
\]
\end{notation}

\label{sec:related-work-comparison}

\begin{definition}[Models with selection for all predicates, \coqterm{AllPredicateEpsilonFullModel}]
\label{def:all-predicate-model}
Fix a typed signature $\Sigma$ and a thin $\varepsilon$-hyperdoctrine
$\mathcal M$ interpreting $\mathcal L_{\exists\varepsilon}$. An
all-predicate selection on $\mathcal M$ is a chosen family of base arrows
\[
 e_Q:\sem{I}\longrightarrow\sem{A}
 \qquad
 (I,A\in\mathsf{BType},\ Q\in P(\sem{I}\times\sem{A}))
\]
together with isomorphisms
$\Sigma_{\pi_{I,A}}Q\cong P(\langle\mathrm{id},e_Q\rangle)Q$.
A model with these additional data is an all-predicate model.
The family need not agree with $\sem{\varepsilon_\Phi}$ when
$Q=\sem{\Phi}$ and has no imposed naturality or extensionality.
This is the ``enough epsilon terms'' condition on the interpreted base
objects, without an equality predicate.
(Cf.~\cite[p.~452]{Pasquali2019}).
\end{definition}

\begin{proposition}[The adjunction expressed by a selected instance, \coqterm{interpreted_formula_epsilon_selection_adjunction}]
\label{prop:selected-instance-adjunction}
Let $\mathcal M$ be a thin $\varepsilon$-hyperdoctrine
for $\mathrm{iFOL}_{\exists\varepsilon}$ over $\Sigma$. For
$I\times A\vdash\Phi\;\mathsf{form}$ and
$I\vdash\Psi\;\mathsf{form}$,
\[
\begin{aligned}
 &\operatorname{Hom}_{P_{\mathcal M}(\sem{I})}
   \bigl(\sem{\Phi[\varepsilon_\Phi]}_I,\sem{\Psi}_I\bigr)
   \text{ is inhabited}\\
 &\qquad\Longleftrightarrow\quad
 \operatorname{Hom}_{P_{\mathcal M}(\sem{I\times A})}
   \bigl(\sem{\Phi}_{I\times A},\sem{\Psi[\pi_1]}_{I\times A}\bigr)
   \text{ is inhabited}.
\end{aligned}
\]
(Cf.~\cite[pp.~414--415]{Pasquali2016}, \cite[p.~452]{Pasquali2019}).
\end{proposition}

\begin{proof}
For $\mathrm{iFOL}_{\exists\varepsilon}$, the model supplies a selection isomorphism
$\Sigma_\pi\sem{\Phi}\cong\sem{\Phi[\varepsilon_\Phi]}$.
Apply the adjunction $\Sigma_\pi\dashv P(\pi)$ to $\sem{\Psi}$,
and use $P(\pi)\sem{\Psi}\cong\sem{\Psi[\pi_1]}$.
\end{proof}

\begin{remark}[Comparison with categorical choice, \coqterm{PasqualiEnoughEpsilonTerms}]
Pasquali's condition of having enough $\varepsilon$-terms
 requires, for each admissible pair
of base objects $X,A$ and every $\varphi\in P(X\times A)$, an arrow
$e_\varphi:X\to A$ such that
\[
 P(\langle\mathrm{id}_X,e_\varphi\rangle)(\varphi)\leq\psi
 \quad\Longleftrightarrow\quad
 \varphi\leq P(\pi_{X,A})(\psi)
 \qquad(\psi\in P(X)).
\]
Here $\leq$ denotes the existence of a fibre morphism, and the admissible
objects exclude stable initial objects.
Proposition~\ref{prop:selected-instance-adjunction} gives the comparison
for interpreted base types and represented predicates, with witnesses
indexed by formulae. Selection for all semantic predicates is additional
data.
(Cf.~\cite[Def.~2.1, p.~452]{Pasquali2019}).
\end{remark}

\begin{corollary}[Selection from a total represented relation, \coqterm{EpsilonFull.Examples.represented_selection_witness_typing}, \coqterm{EpsilonFull.Examples.represented_selection_from_existence}]
\label{cor:represented-selection}
In $\mathrm{iFOL}_{\exists\varepsilon}$ over $\Sigma$, let $I,A$ be base types and $I\times A\vdash\Phi\;\mathsf{form}$. If
$I\mid\varnothing\vdash_{\exists\varepsilon}\exists_A\Phi$, then
\[
 I\vdash\varepsilon_\Phi:A,
 \qquad
 I\mid\varnothing
 \vdash_{\exists\varepsilon}\Phi[\varepsilon_\Phi].
\]
(Cf.~\cite{HilbertBernays1939}).
\end{corollary}

\begin{proof}
The first judgement is the typing rule for a labelled constant.
The second follows by conversion along $(\varepsilon\Sigma)$.
\end{proof}

\begin{remark}[Equality and presented surjections]
Let $X,Q\in\mathsf{BType}$. Suppose the signature contains a predicate
$\mathsf{Eq}_Q$ of arity $Q\times Q$, and let $X\vdash q:Q$.
Over $Q\times X$, define
$\Phi:=\mathsf{Eq}_Q(\langle q\circ\pi_2,\pi_1\rangle)$.
Corollary~\ref{cor:represented-selection} turns a derivation of
$Q\mid\varnothing\vdash_{\exists\varepsilon}\exists_X\Phi$ into one of
$\mathsf{Eq}_Q(\langle q\circ\varepsilon_\Phi,\mathrm{id}_Q\rangle)$.
To conclude that $\sem{q}\circ\sem{\varepsilon_\Phi}
\sim_{\mathcal B}\mathrm{id}_{\sem{Q}}$, the interpretation of
$\mathsf{Eq}_Q$ must additionally reflect truth into equality of base
morphisms.

Pasquali's $\varepsilon$-doctrines
 have equality predicates, and his
categorical characterization uses a proper factorization system and
injectivity of non-initial objects. These additional assumptions are not
part of our equality-free semantics. The preceding corollary alone
therefore neither identifies the base with a category of quotients nor
proves that every epimorphism has a section.
(Cf.~\cite[Def.~2.2, p.~452]{Pasquali2019}).
\end{remark}

\section{Soundness and completeness}
\label{sec:soundness-completeness}

\subsection{The syntactic base and its interpretation}

\begin{definition}[Syntactic base-category data, \coqterm{Base_EpsilonFull_Completeness.hom_base}]
\label{def:syntactic-base}
For the base language of $\mathcal L_{\exists\varepsilon}$ over a typed signature $\Sigma$, the
syntactic base-category data are
\[
 \operatorname{Ob}_{\mathrm{syn}}:=\mathsf{BType},
 \qquad
 \operatorname{Hom}_{\mathrm{syn}}(I,A)
 :=\{t\mid I\vdash t:A\}.
\]
furthermore the identity datum is $\mathrm{id}_I$ for type $I$, the composition datum is $g\circ f$ for
$I\vdash f:J$ and $J\vdash g:A$, and the datum for the setoid equivalence is
\[
 f\sim_{I,A}g
 \quad\overset{\mathrm{def}}{\Longleftrightarrow}\quad
 f\equiv_{I,A}g.
\]
(Cf.~\cite[p.~139]{Jacobs1999}).
\end{definition}

\begin{proposition}[The syntactic base category, \coqterm{Base_EpsilonFull_Completeness.Base_CartesianCategory}]
\label{prop:syntactic-base-category}
For the base syntax of $\mathcal L_{\exists\varepsilon}$ over $\Sigma$, take base types as objects, typed terms $I\vdash t:A$ as arrows, and base-term equivalence as the hom-setoid relation. With syntactic identity and composition these determine a setoid category
$\mathcal B_{\mathrm{syn}}$ with
$\operatorname{Ob}(\mathcal B_{\mathrm{syn}})=\operatorname{Ob}_{\mathrm{syn}}$
and
$\operatorname{Hom}_{\mathcal B_{\mathrm{syn}}}(I,A)
=\operatorname{Hom}_{\mathrm{syn}}(I,A)$.
It is cartesian with terminal object $\mathbf 1$, terminal morphisms ${!}_I$,
binary products $A\times B$, projections, and pairing. Its distinguished
primitive object is $S\ne\mathbf 1$.
(Cf.~\cite[pp.~139--140]{Jacobs1999}).
\end{proposition}

\begin{proof}
Reflexivity, symmetry, and transitivity of base-term equivalence give
the hom-setoids. Composition congruence makes composition proper;
associativity and the unit equations give the category laws.
The two projection equations and
$\langle\pi_1\circ q,\pi_2\circ q\rangle\equiv q$ give the product
universal property, and $k\equiv {!}_I$ gives terminality.
Finally, $S\ne\mathbf 1$ follows by discrimination on base types.
\end{proof}

\begin{definition}[Syntactic interpretation, \coqterm{Base_SyntacticEpsilonFullModel}]
\label{def:syntactic-interpretation}
For $\mathcal L_{\exists\varepsilon}$ over $\Sigma$, let $\mathcal B_{\mathrm{syn}}$ have base types as objects and typed base terms as arrows, modulo base-term equivalence as hom relation. Interpret the base syntax by
\[
 \sem{A}_{\mathrm{syn}}:=A,
 \qquad
 \iota^{\mathrm{syn},A}_{I,s}:=c_s^A
 \in\operatorname{Hom}_{\mathcal B_{\mathrm{syn}}}(I,A),
 \qquad
 \iota^{\mathrm{syn}}_u:=\widehat u
 \in\operatorname{Hom}_{\mathcal B_{\mathrm{syn}}}
 (\operatorname{dom}(u),\operatorname{cod}(u)).
\]
(Cf.~\cite[pp.~124--128]{Jacobs1999}).
\end{definition}

\begin{proposition}[The syntactic interpretation is a base model, \coqterm{Base_EpsilonFull_Completeness.interp_tm_term}]
\label{prop:syntactic-interpretation-identity}
For the base syntax of $\mathcal L_{\exists\varepsilon}$ over $\Sigma$, interpret each type by itself and each primitive function or labelled constant by its syntactic term. In the cartesian syntactic base category, this interpretation is a base model.
For every $I\vdash t:A$, the morphism $\sem{t}_{\mathrm{syn}}$ is represented
by $t$.
(Cf.~\cite[pp.~124--128]{Jacobs1999}).
\end{proposition}

\begin{proof}
Induction on the typing derivation shows that $\sem{t}_{\mathrm{syn}}$ is
represented by $t$. Hence two derivations of $I\vdash t:A$ have equivalent
interpretations by reflexivity of base-term equivalence.
\end{proof}

\begin{theorem}[Base soundness and completeness, \coqterm{Base_EpsilonFull_Soundness.base_soundness_equational}, \coqterm{Base_EpsilonFull_Completeness.base_completeness_equational}]
\label{thm:base-soundness-completeness}
Fix a typed function--predicate signature $\Sigma$ and the base language of $\mathcal L_{\exists\varepsilon}$. Let $\mathsf{BaseMod}(\Sigma)$ range over base models in cartesian setoid categories with a distinguished $S_{\mathcal B}\ne\mathbf 1$, arbitrary typed primitive-function and labelled-constant interpretations, and the cartesian interpretation of terms, independent of the typing derivation up to $\sim_{\mathcal B}$. For all base types $I,A$,
\[
 \bigl(\exists f\; I\vdash f:A\bigr)
 \quad\Longleftrightarrow\quad
 \left(
 \forall(\mathcal B,S_{\mathcal B},\sem{-})
 \in\mathsf{BaseMod}(\Sigma),
 \ \operatorname{Hom}_{\mathcal B}(\sem{I},\sem{A})
 \text{ is inhabited}
 \right).
\tag{inh}
\]
For all well-typed terms $I\vdash f,g:A$,
\[
 f\equiv_{I,A}g
 \quad\Longleftrightarrow\quad
 \left(
 \forall(\mathcal B,S_{\mathcal B},\sem{-})
 \in\mathsf{BaseMod}(\Sigma),
 \ \sem{f}\sim_{\mathcal B}\sem{g}
 \right).
\tag{eq}
\]
(Cf.~\cite[pp.~138--140, 147--148]{Jacobs1999}).
\end{theorem}

\begin{proof}
The forward implications are induction on typing and on base-term
equivalence, respectively. For both reverse implications, use the
syntactic base model: its arrows $I\to A$ are precisely the well-typed
terms, its interpretation sends a term to itself, and its hom-setoid
relation is precisely $\equiv_{I,A}$. Thus an inhabited hom-set gives a
typed term, and equality of the two interpretations gives the base equation.
\end{proof}

\subsection{The full epsilon term model}

\begin{definition}[Validity in thin $\varepsilon$-hyperdoctrines, \coqterm{EpsilonFull_Completeness.epsilon_adjoint_completeness_inhabitational_gen}]
\label{def:semantic-validity}
Let $\mathcal M$ be a thin $\varepsilon$-hyperdoctrine
interpreting $\mathcal L_{\exists\varepsilon}$ over a typed signature $\Sigma$. For
$I\vdash\Gamma\;\mathsf{ctx}$ and $I\vdash\Phi\;\mathsf{form}$, define
\[
 I\mid\Gamma\models_{\exists\varepsilon}^{\mathcal M}\Phi
 \quad\Longleftrightarrow\quad
 \operatorname{Hom}_{P_{\mathcal M}(\sem{I}_{\mathcal M})}
 \bigl(\sem{\Gamma}_{\mathcal M},\sem{\Phi}_{\mathcal M}\bigr)
 \text{ is inhabited},
\]
and write $I\mid\Gamma\models_{\exists\varepsilon}\Phi$ when
$I\mid\Gamma\models_{\exists\varepsilon}^{\mathcal M}\Phi$ for every such
$\mathcal M$.
(Cf.~\cite[p.~247]{Jacobs1999}).
\end{definition}

\begin{definition}[The term model of the
$\mathrm{iFOL}_{\exists\varepsilon}$-language, \coqterm{TermModel_EpsilonFullModel}]
\label{def:syntactic-term-model}
Fix a typed function--predicate signature $\Sigma$ and its finite-tree language $\mathcal L_{\exists\varepsilon}$. Fix the corresponding
intuitionistic intensional Hilbert $\varepsilon$-calculus
$\vdash_{\exists\varepsilon}$. Let
$\mathcal B_{\mathrm{syn}}$ have base types as objects and typed
base terms as arrows, with base-term equivalence and the identity
interpretation of syntax.
For
$I\in\mathsf{BType}$, define the fibre
$P_{\exists\varepsilon}(I)$ by
\[
 \operatorname{Ob}\bigl(
 P_{\exists\varepsilon}(I)\bigr)
 :=\{\Phi\mid I\vdash\Phi\;\mathsf{form}\},
 \qquad
 \operatorname{Hom}_{
 P_{\exists\varepsilon}(I)}
 (\Phi,\Psi)
 :=\{p\mid
 I\mid\varnothing\vdash_{\exists\varepsilon}
 p:\Phi\to\Psi\}.
\]
For $p:\Phi\to\Psi$ and $q:\Psi\to\Theta$, define
\[
 \mathrm{id}_\Phi:=\lambda x.x,
 \qquad
 q\circ p
 :=(\lambda f.\lambda g.\lambda x.f(gx))\,q\,p.
\]
All parallel fibre morphisms are equivalent. For
$I\vdash f:J$, syntactic substitution defines
\[
 P_{\exists\varepsilon}(f)(\Phi)
 :=\Phi[f],
 \qquad
 P_{\exists\varepsilon}(f)(p)
 :=p[f].
\]
The propositional operations and the two quantifiers are given by
$\top,\bot,\land,\lor,\to,\exists_A$, and $\forall_A$, respectively. The
atomic interpretation sends every well-formed atom $R(t)$ to itself and
every ill-typed raw atom to $\top$; it is therefore defined on all raw
atomic codes required by the semantic interface.
Write $\mathcal T_{\exists\varepsilon}$ for this construction, with
$\Sigma$ suppressed. Its base interpretation sends
every well-typed base term to itself. In particular,
\[
 \sem{\varepsilon_\Phi}_{\mathcal T_{\exists\varepsilon}}
 =c^A_{\langle A,\Phi\rangle}
 =\varepsilon_\Phi
 \in\operatorname{Hom}_{\mathcal B_{\mathrm{syn}}}(I,A)
 \qquad
 (I\times A\vdash\Phi\;\mathsf{form}).
\]
(Cf.~\cite[pp.~234--235]{Jacobs1999}).
\end{definition}

\begin{proposition}[The term model as a thin $\varepsilon$-hyperdoctrine, \coqterm{TermModel_AdjointEpsilonFullFibration}]
\label{prop:syntactic-term-model}
For every typed function--predicate signature $\Sigma$, the construction
$\mathcal T_{\exists\varepsilon}$ is a thin $\varepsilon$-hyperdoctrine
over $\mathcal B_{\mathrm{syn}}$. Its base,
atomic, and labelled-constant interpretations send well-typed syntax to
itself; ill-typed raw atoms are sent to $\top$. Its fibre arrows
$\Phi\to\Psi$ are closed proofs of the implication, with all parallel
arrows identified, and reindexing is syntactic substitution.
\end{proposition}

\begin{proof}
The proof-term rules give the bicartesian closed structure of the
fibres, while Proposition~\ref{prop:derivation-substitution} gives semantic
substitution. The quantifier rules and substitution conversions give the
adjunctions and Beck--Chevalley isomorphisms. Since all parallel fibre
morphisms are equivalent, the fibres are thin. Finally,
Proposition~\ref{prop:epsilon-existence} gives the
required $\varepsilon$-structure:
\[
 \Sigma_{\pi_{I,A}}
 \bigl(\sem{\Phi}_{\mathcal T_{\exists\varepsilon}}\bigr)
 =\exists_A\Phi
 \cong\Phi[\varepsilon_\Phi]
 =\sem{\Phi[\varepsilon_\Phi]}_{\mathcal T_{\exists\varepsilon}}.
\]
\end{proof}

\begin{definition}[Context formula, \coqterm{EpsilonFull_Completeness.h_make_ctx_ty}]
\label{def:context-formula}
For a well-formed logical context $I\vdash\Gamma\;\mathsf{ctx}$ in $\mathrm{iFOL}$ or $\mathrm{iFOL}_{\exists\varepsilon}$, its context formula in the same language is defined by
\[
 \varnothing^\times:=\top,
 \qquad
 (\Phi,\Gamma)^\times:=\Gamma^\times\land\Phi.
\]
(Cf.~\cite[pp.~224, 234]{Jacobs1999}).
\end{definition}

\begin{remark}[Interpretation in the syntactic term model, \coqterm{EpsilonFull_Completeness.term_model_interp_id}, \coqterm{EpsilonFull_Completeness.term_model_interp_ctx_id}]
\label{rem:term-model-interpretation}
The term model interprets every base type and every well-formed formula by
itself, and every well-formed context by its context formula:
\[
 \sem{I}_{\mathcal T_{\exists\varepsilon}}
 =I,
 \qquad
 \sem{\Phi}_{\mathcal T_{\exists\varepsilon}}
 =\Phi,
 \qquad
 \sem{\Gamma}_{\mathcal T_{\exists\varepsilon}}
 =\Gamma^\times.
\]
These equations follow by induction on $I$, $\Phi$, and $\Gamma$.
\end{remark}

\begin{proposition}[Derivation extraction, \coqterm{DerivationExtraction.extract_concrete}]
\label{prop:term-model-extraction}
Let $\vdash_\star$ be a typed proof-term calculus with raw proof-term set $\mathsf{proofterm}$, containing the assumption
and weakening rules, $\top$- and $\land$-introduction, and
$\to$-elimination. Suppose that, for every base type $I$, its syntactic
fibre $P_\star(I)$ has well-formed formulae over $I$ as objects and
\[
 \operatorname{Hom}_{P_\star(I)}(\Phi,\Psi)
 :=\{p\in \mathsf{proofterm}\mid \left ( I\mid\varnothing\right )\vdash_\star p:\Phi\to\Psi\}.
\]
If $I\vdash\Gamma\;\mathsf{ctx}$ and
$I\vdash\Phi\;\mathsf{form}$, then
\[
 \operatorname{Hom}_{P_\star(I)}(\Gamma^\times,\Phi)
 \text{ is inhabited}
 \quad\Longrightarrow\quad
 \exists p\;I\mid\Gamma\vdash_\star p:\Phi.
\]
(Cf.~\cite[pp.~234--235]{Jacobs1999}).
\end{proposition}

\begin{proof}
Choose $q$ with
$I\mid\varnothing\vdash_\star q:\Gamma^\times\to\Phi$.
Weakening gives the same judgement over $\Gamma$. Define
$r_\varnothing:=\mathsf{tt}$ and, for
$\Gamma=(\Theta,\Delta)$,
$r_\Gamma:=\langle r_\Delta,x_\Theta\rangle$, where $r_\Delta$ is weakened
to $\Gamma$ and $x_\Theta$ is the assumption of type $\Theta$. Then
$I\mid\Gamma\vdash_\star r_\Gamma:\Gamma^\times$, and therefore
$I\mid\Gamma\vdash_\star q\,r_\Gamma:\Phi$.
\end{proof}

\begin{theorem}[Soundness and completeness of
$\mathrm{iFOL}_{\exists\varepsilon}$, \coqterm{CanonicalEpsilonFullSemantics.soundness}, \coqterm{CanonicalEpsilonFullSemantics.completeness}]
\label{thm:soundness}
\label{thm:completeness}
Fix a typed signature $\Sigma$ and the calculus $\mathrm{iFOL}_{\exists\varepsilon}$ on $\mathcal L_{\exists\varepsilon}$. Then
\begin{enumerate}
\item
For every thin $\varepsilon$-hyperdoctrine $\mathcal M$
interpreting this language,
every derivation
$I\mid\Gamma\vdash_{\exists\varepsilon}p:\Phi$
has an interpretation
\[
 \sem{p}_{\mathcal M}\in
 \operatorname{Hom}_{P_{\mathcal M}(\sem{I}_{\mathcal M})}
 \bigl(\sem{\Gamma}_{\mathcal M},\sem{\Phi}_{\mathcal M}\bigr).
\]
\item
For every $I\vdash\Gamma\;\mathsf{ctx}$ and
$I\vdash\Phi\;\mathsf{form}$,
\[
 I\mid\Gamma\vdash_{\exists\varepsilon}\Phi
 \quad\Longleftrightarrow\quad
 I\mid\Gamma\models_{\exists\varepsilon}\Phi.
\]
\end{enumerate}
\end{theorem}

\begin{proof}
(1) Induction on the derivation. The two $\varepsilon$-cases use the two
directions of \textup{$(\varepsilon\Sigma)$}.

(2) The forward implication follows from (1). Conversely, assume
$I\mid\Gamma\models_{\exists\varepsilon}^{\mathcal M}\Phi$ for every thin
$\varepsilon$-hyperdoctrine $\mathcal M$ interpreting the fixed language.
This applies in particular to
$\mathcal T_{\exists\varepsilon}$ by
Proposition~\ref{prop:syntactic-term-model}. By
Remark~\ref{rem:term-model-interpretation}, it gives
\[
 \operatorname{Hom}_{P_{\exists\varepsilon}(I)}
 (\Gamma^\times,\Phi)
 \text{ is inhabited}.
\]
Proposition~\ref{prop:term-model-extraction} yields
$I\mid\Gamma\vdash_{\exists\varepsilon}\Phi$.
\end{proof}

\subsection{Completeness with selection for all predicates}

\begin{theorem}[All-predicate soundness and completeness, \coqterm{all_predicate_epsilonfull_soundness_completeness}]
\label{thm:all-predicate-completeness}
Fix a typed signature $\Sigma$, the calculus
$\mathrm{iFOL}_{\exists\varepsilon}$ on $\mathcal L_{\exists\varepsilon}$,
a well-formed context $I\vdash\Gamma\;\mathsf{ctx}$, and a well-formed
formula $I\vdash\Phi\;\mathsf{form}$. Then
\[
 I\mid\Gamma\vdash_{\exists\varepsilon}\Phi
 \quad\Longleftrightarrow\quad
 \forall\mathcal M,\quad
 \operatorname{Hom}_{P_{\mathcal M}(\sem{I})}
       (\sem{\Gamma},\sem{\Phi})\text{ is inhabited},
\]
where $\mathcal M$ ranges over thin $\varepsilon$-hyperdoctrines for this
language equipped with a chosen selector $e_Q:\sem{J}\to\sem{A}$ and
an isomorphism $\Sigma_\pi Q\cong
P(\langle\mathrm{id},e_Q\rangle)Q$ for every base type pair $J,A$ and
every $Q\in P_{\mathcal M}(\sem{J}\times\sem{A})$.
The syntactic term model itself admits these additional data.
\end{theorem}

\begin{proof}
Soundness applies after forgetting the additional selection data.
For completeness, an object $Q$ of the syntactic fibre over $J\times A$
is a well-formed formula together with its formation derivation.
Choose its syntactic epsilon term. The conversion
$\exists_A Q\simeq_J Q[\varepsilon_Q]$ gives the required selection
isomorphism. This makes the term model an all-predicate model, so the
semantic assumption applies to it and gives a closed proof of
$\Gamma^\times\to\Phi$. Context reification and application extract
$I\mid\Gamma\vdash_{\exists\varepsilon}\Phi$ as before.
\end{proof}

\subsection{The iFOL term model and completeness}

\begin{proposition}[The term model of $\mathrm{iFOL}$, \coqterm{TermModel_IFOLHyperModel}]
\label{prop:ifol-term-model}
Fix a typed function--predicate signature $\Sigma$. For $\mathrm{iFOL}$, take its base types and typed terms as the syntactic base, well-formed formulae as fibre objects, and closed implication proofs as fibre arrows; identify all parallel fibre arrows. Use syntactic substitution for reindexing and interpret well-typed syntax by itself, with ill-typed raw atoms sent to $\top$.
Write $\mathcal T_{\mathrm{iFOL}}$ for the resulting model and $P_{\mathrm{iFOL}}$ for its
fibration. Then $\mathcal T_{\mathrm{iFOL}}$ is an iFOL-hyperdoctrine.
For base types $I,A$ and a well-formed formula
$I\times A\vdash\Phi\;\mathsf{form}$ of $\mathcal L$, with
$\pi_{I,A}:I\times A\to I$, its quantifiers satisfy
\[
 \Sigma_{\pi_{I,A}}(\Phi):=\exists_A\Phi,
 \qquad \Pi_{\pi_{I,A}}(\Phi):=\forall_A\Phi,
 \qquad \Sigma_{\pi_{I,A}}\dashv P_{\mathrm{iFOL}}(\pi_{I,A})\dashv\Pi_{\pi_{I,A}}.
\]
Base types and well-formed formulae interpret as themselves, and
$\sem{\Gamma}_{\mathcal T_{\mathrm{iFOL}}}=\Gamma^\times$ for every well-formed context.
(Cf.~\cite[pp.~234--235]{Jacobs1999}).
\end{proposition}

\begin{proof}
The iFOL rules give the cartesian base, bicartesian closed fibres,
syntactic substitution, and both quantifier adjunctions. Substitution
conversion gives Beck--Chevalley. The interpretation identities follow
by induction on types, well-formed formulae, and contexts.
\end{proof}

\begin{theorem}[Soundness and completeness of $\mathrm{iFOL}$, \coqterm{ifol_soundness}, \coqterm{ifol_completeness_inhabitational_gen}]
\label{thm:ifol-soundness-completeness}
Fix a typed signature $\Sigma$, a well-formed context
$I\vdash\Gamma\;\mathsf{ctx}$ and a well-formed formula
$I\vdash\Phi\;\mathsf{form}$ in $\mathcal L$. Then
\[
 I\mid\Gamma\vdash_{\mathrm{iFOL}}\Phi
 \quad\Longleftrightarrow\quad
 \bigl(\forall\mathcal M,\quad
 I\mid\Gamma\models_{\mathrm{iFOL}}^{\mathcal M}\Phi\bigr),
\]
where $\mathcal M$ ranges over all iFOL-hyperdoctrines interpreting $\mathcal L$ over $\Sigma$.
(Cf.~\cite[pp.~234--235, 247--248]{Jacobs1999}).
\end{theorem}

\begin{proof}
Soundness is by induction on iFOL derivations; the quantifier rules are
interpreted by $\Sigma_\pi\dashv P(\pi)\dashv\Pi_\pi$.
For completeness, apply validity to the iFOL term model of
Proposition~\ref{prop:ifol-term-model}. Its interpretation identities
make $\operatorname{Hom}_{P_{\mathrm{iFOL}}(I)}(\Gamma^\times,\Phi)$
inhabited. The derivation-extraction argument yields
$I\mid\Gamma\vdash_{\mathrm{iFOL}}\Phi$.
\end{proof}

\subsection{Initiality for interpretation-preserving maps}

\begin{definition}[The model category, \coqterm{EpsilonFullModelHom}]
\label{def:epsilon-model-morphism}
\label{def:epsilon-model-category}
Fix a typed function--predicate signature $\Sigma$ and its language
$\mathcal L_{\exists\varepsilon}$. Let $\mathcal M,\mathcal N$ be thin
$\varepsilon$-hyperdoctrines interpreting this language.
A morphism $F:\mathcal M\to\mathcal N$ consists of proper maps
\[
\begin{aligned}
 F^{\mathrm b}_{I,A}&:\mathcal B_{\mathcal M}[I,A]
   \longrightarrow\mathcal B_{\mathcal N}[I,A],\\
 F^{\mathrm f}_{I,\Phi,\Psi}&:P_{\mathcal M}^{I}[\Phi,\Psi]
   \longrightarrow P_{\mathcal N}^{I}[\Phi,\Psi],\\
 F^{\mathrm s}_{I,\Gamma,\Phi}&:P_{\mathcal M}^{I}[\Gamma,\Phi]
   \longrightarrow P_{\mathcal N}^{I}[\Gamma,\Phi],
\end{aligned}
\]
where $I,A$ are base types, $\Phi,\Psi$ are well formed over $I$, $\Gamma$
is a well-formed context over $I$, and
\[
\begin{aligned}
 \mathcal B_{\mathcal M}[I,A]
 &:=\operatorname{Hom}_{\mathcal B_{\mathcal M}}
   (\sem{I}_{\mathcal M},\sem{A}_{\mathcal M}),\\
 P_{\mathcal M}^{I}[X,Y]
 &:=\operatorname{Hom}_{P_{\mathcal M}(\sem{I}_{\mathcal M})}
   (\sem{X}_{\mathcal M},\sem{Y}_{\mathcal M}).
\end{aligned}
\]
They preserve fibre identities and composition and satisfy
\[
 F^{\mathrm b}_{I,A}(\sem{t}_{\mathcal M})
 \sim\sem{t}_{\mathcal N},
 \qquad
 F^{\mathrm s}_{I,\Gamma,\Phi}(\sem{p}_{\mathcal M})
 \sim\sem{p}_{\mathcal N}
\]
for every $I\vdash t:A$ and
$I\mid\Gamma\vdash_{\exists\varepsilon}p:\Phi$.
Write
$\operatorname{Mod}_{\exists\varepsilon}
(\mathcal M,\mathcal N)$ for the resulting
setoid, with pointwise equivalence of the three components.
Let $\mathsf{Mod}_{\exists\varepsilon}(\Sigma)$ have these models as
objects and these morphisms as arrows, with componentwise identities
and composition.
\end{definition}

\begin{theorem}[Initiality of the syntactic term model, \coqterm{EpsilonFull_ModelCategory.TermModel_weak_initial}, \coqterm{EpsilonFull_ModelCategory.TermModel_initial_unique}]
\label{thm:term-model-initiality}
Fix a typed signature $\Sigma$. Thin $\varepsilon$-hyperdoctrines interpreting $\mathcal L_{\exists\varepsilon}$, with proper base, fibre, and sequent maps indexed by syntax that preserve term and derivation interpretations and fibre identity and composition, form a possibly
large setoid category $\mathsf{Mod}_{\exists\varepsilon}(\Sigma)$.
The term model $\mathcal T_{\exists\varepsilon}$ is initial: for every
object $\mathcal M$, there is a morphism
$\eta_{\mathcal M}:\mathcal T_{\exists\varepsilon}\to\mathcal M$,
unique up to the hom-setoid equivalence.
\end{theorem}

\begin{proof}
Pointwise equivalence gives the hom-setoids.
Properness makes the interpretation-preservation laws stable under
composition. The fibre functor laws, associativity, and the unit laws
hold componentwise, as does properness of composition.

A base morphism of $\mathcal T_{\exists\varepsilon}$ is a typed base term
and is sent to its interpretation in $\mathcal M$. A fibre morphism is a
closed implication derivation, which
Theorem~\ref{thm:soundness} and evaluation interpret in $\mathcal M$. For a
sequent morphism, precompose this interpretation with
$\sem{\Gamma}_{\mathcal M}\to\sem{\Gamma^\times}_{\mathcal M}$. These
assignments satisfy base compatibility by the equational soundness part of
Theorem~\ref{thm:base-soundness-completeness}, and all fibre and sequent
equations by thinness. Base-term preservation forces the base component of
any such morphism, while thinness forces its fibre and sequent components.
\end{proof}

\begin{remark}[Scope of initiality, \coqterm{EpsilonFull_ModelCategory.TermModel_initial_unique}]
The model morphisms are indexed by syntactic types, formulae, and contexts.
Their base component preserves interpretations of typed terms, but need
not preserve composition of arbitrary semantic base morphisms. Initiality
is asserted for these maps, not for arbitrary cartesian functors between
the underlying base categories. Its uniqueness does not identify raw
proof terms.
\end{remark}

\begin{corollary}[Compatibility of interpretations, \coqterm{EpsilonFull_ModelCategory.TermModel_interpretation_natural}, \coqterm{EpsilonFull_ModelCategory.TermModel_interpretation_identity}]
\label{cor:interpretation-naturality}
Fix a typed signature $\Sigma$. For models $\mathcal M,\mathcal N$ of $\mathrm{iFOL}_{\exists\varepsilon}$ and a morphism $F:\mathcal M\to\mathcal N$ in $\mathsf{Mod}_{\exists\varepsilon}(\Sigma)$, let $\eta$ denote the interpretation morphisms from the syntactic term model. Then
\[
 F\circ\eta_{\mathcal M}\sim\eta_{\mathcal N},
 \qquad
 \eta_{\mathcal T_{\exists\varepsilon}}
 \sim\mathrm{id}_{\mathcal T_{\exists\varepsilon}}.
\]
(Cf.~\cite{LambekScott1986}).
\end{corollary}

\begin{proof}
In each equation, both sides are morphisms from the initial object to
the same model. Apply Theorem~\ref{thm:term-model-initiality}.
\end{proof}

\section{Logical strength and separation}
\label{sec:logical-strength-separation}

\subsection{Dual Drinker and excluded middle}

\begin{definition}[Dual Drinker, \coqterm{DualDrinker_Syntax.Dual_Drinker}]
In $\mathcal L_{\exists\varepsilon}$, or in the epsilon-free language $\mathcal L$ when $\Phi$ belongs to it, let $I,A\in\mathsf{BType}$ and
$I\times A\vdash\Phi\;\mathsf{form}$. Write
$(\exists_A\Phi)^\uparrow:=(\exists_A\Phi)[\pi_1]$. The corresponding
Dual Drinker formula is
\[
 \exists_A\bigl((\exists_A\Phi)^\uparrow\to\Phi\bigr).
\]
For $\Phi=P(\pi_2)$ this is
$\exists x:A\,((\exists y:A\,P(y))\to P(x))$.
(Cf.~\cite[p.~5]{WarrenDienerMcKubreJordens2018}).
\end{definition}

\begin{proposition}[Dual Drinker is derivable, \coqterm{epsilonfull_dual_drinker_provable}]
\label{prop:dual-drinker-derivable}
In $\mathrm{iFOL}_{\exists\varepsilon}$ over a typed signature $\Sigma$, let $I,A$ be base types and $I\times A\vdash\Phi\;\mathsf{form}$. Write $(\exists_A\Phi)^\uparrow=(\exists_A\Phi)[\pi_1]$. Then there is a proof term $d_\Phi$ with
\[
 I\mid\varnothing
 \vdash_{\exists\varepsilon}d_\Phi:
 \exists_A\bigl((\exists_A\Phi)^\uparrow\to\Phi\bigr).
\]
(Cf.~\cite[p.~5]{WarrenDienerMcKubreJordens2018}).
\end{proposition}

\begin{proof}
Let $\Omega_\Phi:=(\exists_A\Phi)^\uparrow\to\Phi$.
Proposition~\ref{prop:epsilon-existence} gives
\[
 I\mid\varnothing\vdash_{\exists\varepsilon}
 e_\Phi:(\exists_A\Phi)\to\Phi[\varepsilon_\Phi].
\]
For the first projection $I\times A\vdash\pi_1:I$,
$\pi_1\circ\langle\mathrm{id}_I,\varepsilon_\Phi\rangle
\equiv_{I,I}\mathrm{id}_I$. Hence substitution composition and the
identity conversion give
\[
 \Omega_\Phi[\varepsilon_\Phi]
 =\bigl((\exists_A\Phi)[\pi_1][\varepsilon_\Phi]
   \to\Phi[\varepsilon_\Phi]\bigr)
 \simeq_I\bigl((\exists_A\Phi)\to\Phi[\varepsilon_\Phi]\bigr).
\tag{D}
\]
With $d_\Phi:=\mathsf{exI}(\varepsilon_\Phi,e_\Phi)$, the derivation is
\[
\prftree[r]{$\scriptstyle\exists\mathrm I$}
 {I\vdash\varepsilon_\Phi:A}
 {
  \prftree[r]{$\scriptstyle\mathrm{conv}$}
   {I\mid\varnothing\vdash
     e_\Phi:(\exists_A\Phi)\to\Phi[\varepsilon_\Phi]}
   {\bigl((\exists_A\Phi)\to\Phi[\varepsilon_\Phi]\bigr)
     \simeq_I\Omega_\Phi[\varepsilon_\Phi]
     \quad\textup{(D)}^{-1}}
   {I\mid\varnothing\vdash e_\Phi:\Omega_\Phi[\varepsilon_\Phi]}
 }
 {I\mid\varnothing\vdash_{\exists\varepsilon}
   d_\Phi:\exists_A\Omega_\Phi}.
\]
\end{proof}

\begin{remark}[The intensional witness, \coqterm{epsilonfull_dual_drinker_provable}]
The witness in Proposition~\ref{prop:dual-drinker-derivable} is the term
$\varepsilon_\Phi$ indexed by the raw presentation of $\Phi$. The derivation
uses $\mathrm{(\varepsilon BC)}$ but no extensionality principle for
$\varepsilon$.
(Cf.~\cite{Molnar2013}).
\end{remark}

\begin{definition}[The countermodel algebras, \coqterm{H3FiniteHeytingAlgebra}, \coqterm{H5FiniteHeytingAlgebra}]
Let $H_3=\{\bot,m,\top\}$ and
$H_5=\{\bot,h_1,h_2,h_{12},\top\}$ be the finite Heyting algebras
whose orders are displayed below. The order runs upwards.
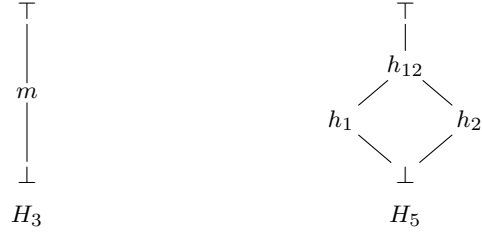
\begin{center}
\begin{tikzpicture}[
  x=1cm,y=0.72cm,
  every node/.style={inner sep=1.5pt}
]
 \node (h3b) at (0,0) {$\bot$};
 \node (h3m) at (0,1.5) {$m$};
 \node (h3t) at (0,3) {$\top$};
 \draw (h3b)--(h3m)--(h3t);
 \node at (0,-0.75) {$H_3$};

 \node (h5b) at (5,0) {$\bot$};
 \node (h51) at (4.15,1) {$h_1$};
 \node (h52) at (5.85,1) {$h_2$};
 \node (h512) at (5,2) {$h_{12}$};
 \node (h5t) at (5,3) {$\top$};
 \draw (h5b)--(h51)--(h512)--(h5t);
 \draw (h5b)--(h52)--(h512);
 \node at (5,-0.75) {$H_5$};
\end{tikzpicture}
\end{center}
Writing $a\Rightarrow b$ for Heyting implication, the values used below are
\[
 (m\Rightarrow\bot)=\bot,
 \qquad
 h_1\lor h_2=h_{12},
 \qquad
 (h_{12}\Rightarrow h_i)=h_i\quad(i=1,2).
\]
(Cf.~\cite[p.~239]{Jacobs1999}).
\end{definition}

\begin{proposition}[The $H_3$-valued epsilon model, \coqterm{H3CanonicalFullModel}]
\label{prop:h3-epsilon-model}
For $\mathrm{iFOL}_{\exists\varepsilon}$, let the signature have no primitive
function symbols and two predicate symbols $Q,R$, with
$\operatorname{ar}(Q)=\mathbf 1$ and $\operatorname{ar}(R)=S$.
Take the cartesian base $\mathcal B_2$ whose objects are base types, with
\[
 |\mathbf 1|=\{*\},\qquad |S|=\mathbf 2=\{\mathsf t,\mathsf f\},
 \qquad |I\times A|=|I|\times|A|,
\]
and whose arrows $I\to J$ are all functions $|I|\to|J|$, with pointwise
equality. The primitive interpretations
\[
 Q_3(*)=m,\qquad R_3(\mathsf t)=m,\qquad R_3(\mathsf f)=\bot
\]
extend to a thin $\varepsilon$-hyperdoctrine $\mathcal M_3$ for the full
language $\mathcal L_{\exists\varepsilon}$, with
\[
 P_3(I)=H_3^{|I|},\qquad P_3(f)(\Psi)=\Psi\circ f
 \quad(f:I\to J,\ \Psi\in H_3^{|J|}).
\]
The fibres are ordered pointwise; connectives are the pointwise Heyting
operations, and quantifiers are joins and meets over the quantified
coordinate.
(Cf.~\cite[p.~159]{MaiettiPasqualiRosolini2017}).
\end{proposition}

\begin{proof}
The predicate-hyperdoctrine operations are those of
Example~\ref{ex:h-valued-predicate-hyperdoctrine} with $H=H_3$ and
$S=\mathbf 2$. Each base carrier is finite and nonempty; fix its
recursively defined enumeration.

Interpret labelled constants by recursion on their label rank. For
$s=\langle D,\Theta\rangle$, an output type $B$, and an environment $v$,
interpret $c_s^B$ by a $b\in|B|$ maximizing the raw value
$V_\Theta(v,b)$ of $\Theta$, with ties settled by the fixed enumeration. This maximum exists because
$H_3=\{\bot<m<\top\}$. Evaluating $\Theta$ uses only constants whose
labels have smaller rank, so the recursion is well founded. The same
rule interprets constants whose output type $B$ differs from $D$.

For raw terms, use tagged unit, Boolean, and pair values, with sequential
evaluation of composition. A raw atom evaluates its term and decodes the
result at the predicate's declared arity; failed evaluation or decoding
has value $\top$. Formula evaluation uses the pointwise Heyting
operations and finite joins and meets. This makes atomic substitution
valid also for raw, ill-typed terms. Induction on typing shows that typed
term interpretation agrees with raw evaluation and is independent of
the typing derivation. The cartesian equations then give atomic base
congruence.

For a well-formed $I\times A\vdash\Phi\;\mathsf{form}$, induction on
$\Phi$ identifies raw formula evaluation with its interpretation. The
chosen distinguished constant therefore satisfies
\[
 \sem{\Phi}_{\mathcal M_3}
    \bigl(x,\sem{\varepsilon_\Phi}_{\mathcal M_3}(x)\bigr)
 =\max_{a\in|A|}\sem{\Phi}_{\mathcal M_3}(x,a)
 =\bigvee_{a\in|A|}\sem{\Phi}_{\mathcal M_3}(x,a).
\]
This is the required $(\varepsilon\Sigma)$ isomorphism in the pointwise
ordered fibres.
\end{proof}

\begin{theorem}[Excluded middle is underivable, \coqterm{H3_canonical_LEM_underivable}]
\label{thm:lem-underivable}
Consider $\mathrm{iFOL}_{\exists\varepsilon}$ with no primitive function symbols and predicate symbols $Q,R$ of arities $\mathbf 1,S$, respectively. Write $Q$ for the atom $Q(!_{\mathbf 1})$ over the base context $\mathbf 1$. Then
\[
 1\mid\varnothing
 \nvdash_{\exists\varepsilon}Q\lor(Q\to\bot).
\]
Thus the intensional $\varepsilon$-principles of
$\mathrm{iFOL}_{\exists\varepsilon}$ do not imply the law of excluded middle.
(Cf.~\cite[p.~409]{Pasquali2016}, \cite[p.~463]{Pasquali2019}).
\end{theorem}

\begin{proof}
Under the identification
$P_3(\mathbf 1)=H_3^{\{*\}}\cong H_3$, one has
\[
 \sem{Q}_{\mathcal M_3}=m,
 \qquad
 \sem{Q\to\bot}_{\mathcal M_3}=m\Rightarrow\bot=\bot,
 \qquad
 \sem{Q\lor(Q\to\bot)}_{\mathcal M_3}=m\lor\bot=m.
\]
The fibre $P_3(\mathbf 1)=H_3^{\{*\}}$ is ordered pointwise, so a morphism
from the constant $\top$ predicate to the constant $m$ predicate would
mean $\top\leq m$, which is false.
Theorem~\ref{thm:soundness} therefore excludes the displayed derivation.
\end{proof}

\begin{proposition}[The finite $H_5$ countermodel, \coqterm{ifol_dual_drinker_refuted}]
\label{prop:h5-dual-drinker-model}
For the constant-free signature with a primitive sort $S$ and a predicate
symbol $P$ with $\operatorname{ar}(P)=S$, there is an iFOL model
$\mathcal M_5$ in which
$\sem{S}_{\mathcal M_5}=\mathbf 2=\{\mathsf t,\mathsf f\}$,
$\sem{P}_{\mathcal M_5}(\mathsf t)=h_1$, and
$\sem{P}_{\mathcal M_5}(\mathsf f)=h_2$, while
\[
 \sem{\exists x:S\,((\exists y:S\,P(y))\to P(x))}_{\mathcal M_5}
 =h_{12}<\top.
\]
(Cf.~\cite[pp.~5--6]{WarrenDienerMcKubreJordens2018}).
\end{proposition}

\begin{proof}
Take the cartesian base $\mathcal B_2$ of finite products of
$\mathbf 2=\{\mathsf t,\mathsf f\}$ and the full predicate fibres
$P_5(I)=H_5^{|I|}$. Substitution is precomposition, connectives are
pointwise, and quantifiers are joins and meets over the quantified
coordinate. Interpret the unary predicate by
$\rho(\mathsf t)=h_1$ and $\rho(\mathsf f)=h_2$, so that
$\sem{P(u)}=\rho\circ\sem{u}$ for a well-typed $u$.
These are the predicate-hyperdoctrine operations already checked above.
For raw atoms the formal model uses a tagged evaluation of raw base
terms; substitution and congruence are verified before soundness is used.

\begin{figure}[htbp]
    \centering
    \definecolor{hFiveAccent}{HTML}{285F78}
    \begin{tikzpicture}[
        x=1cm, y=1cm,
        font=\small,
        value/.style={inner sep=3pt, outer sep=1pt},
        fibre/.style={draw=black!35, fill=black!1,
                      line width=0.6pt, rounded corners=3pt},
        order/.style={draw=black!65, line width=0.7pt},
        action/.style={draw=hFiveAccent, line width=0.85pt,
                       -{Stealth[length=2mm,width=1.4mm]}},
        lies over/.style={draw=black!45, densely dotted, line width=0.6pt}
    ]
        \draw[fibre] (-2.1,1.8) rectangle (2.1,6.0);
        \draw[fibre] (4.9,1.8) rectangle (9.1,6.0);
        \draw[black!20, line width=0.4pt] (-2.1,5.25) -- (2.1,5.25);
        \draw[black!20, line width=0.4pt] (4.9,5.25) -- (9.1,5.25);
        \node at (0,5.63) {$P_5(\mathbf 1)\cong H_5$};
        \node at (7,5.63) {$P_5(S)=H_5^{\mathbf 2}$};

        \node[value] (hTop) at (0,4.77) {$\top$};
        \node[value, text=hFiveAccent, fill=hFiveAccent!10,
              rounded corners=2pt] (hJoin) at (0,3.88) {$h_{12}$};
        \node[value] (hLeft) at (-0.85,2.98) {$h_1$};
        \node[value] (hRight) at (0.85,2.98) {$h_2$};
        \node[value] (hBot) at (0,2.22) {$\bot$};
        \draw[order] (hBot) -- (hLeft) -- (hJoin) -- (hTop);
        \draw[order] (hBot) -- (hRight) -- (hJoin);

        \node[text=hFiveAccent] at (7,4.60)
            {$\rho:\mathbf 2\longrightarrow H_5$};
        \node[value] (rhoT) at (5.65,3.65) {$\mathsf t$};
        \node[value] (rhoF) at (5.65,2.65) {$\mathsf f$};
        \node[value] (rhoHOne) at (8.35,3.65) {$h_1$};
        \node[value] (rhoHTwo) at (8.35,2.65) {$h_2$};
        \draw[action] (rhoT) -- (rhoHOne);
        \draw[action] (rhoF) -- (rhoHTwo);

        \draw[action] (4.72,3.88) -- node[above=5pt] {$\Sigma_{!_S}$}
            node[below=5pt, text=hFiveAccent] {$\rho\mapsto h_{12}$}
            (2.28,3.88);

        \draw[draw=black!25, fill=black!2, line width=0.6pt,
              rounded corners=3pt] (-2.1,-0.48) rectangle (9.1,1.0);
        \node at (-1.40,0.45) {$\mathcal B_2$};
        \node[value] (baseOne) at (0,0.45) {$\mathbf 1$};
        \node[value] (baseS) at (7,0.45) {$S$};
        \node[font=\footnotesize] at (7,-0.13)
            {$|S|=\mathbf 2=\{\mathsf t,\mathsf f\}$};
        \draw[lies over] (0,1.8) -- (baseOne.north);
        \draw[lies over] (7,1.8) -- (baseS.north);
        \draw[draw=black!70, line width=0.7pt,
              -{Stealth[length=2mm,width=1.4mm]}]
            (baseS.west) -- node[below=4pt] {$!_S$} (baseOne.east);
    \end{tikzpicture}
    \caption{Two fibres of the $H_5$-valued countermodel, with
    $|S|=\mathbf 2$. The right panel displays the predicate
    $\rho(\mathsf t)=h_1$, $\rho(\mathsf f)=h_2$.
    Existential quantification along $!_S:S\to\mathbf 1$ gives
    $\Sigma_{!_S}\rho=h_1\lor h_2=h_{12}$, using
    $\mathbf 1\times S\cong S$.}
    \label{fig:h5_hyperdoctrine_structure}
\end{figure}
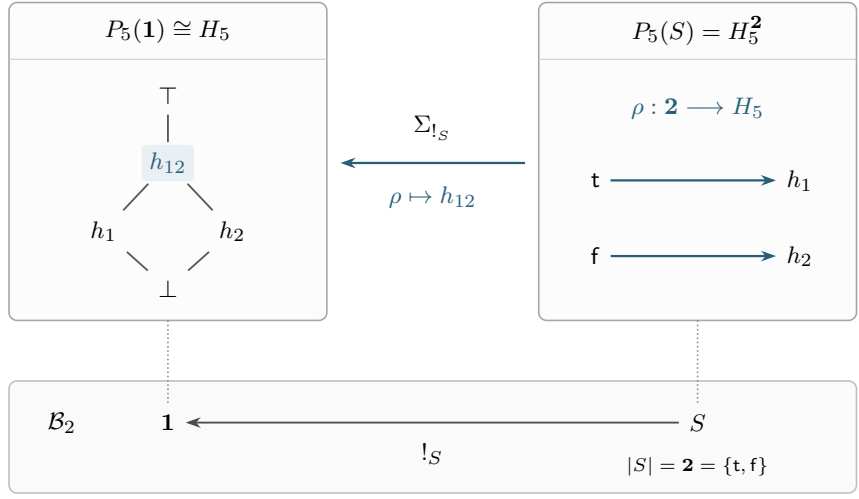

It remains to evaluate the displayed formula. Put
$\Theta:=\exists y:S\,P(y)$,
$\Omega:=\Theta[\pi_1]\to P(x)$, and $D:=\exists x:S\,\Omega$.
Here $\Theta,D$ are closed iFOL formulae and $\Omega$ is over
$\mathbf 1\times S$. The calculation in $\mathcal M_5$ is
\[
\begin{array}{c|c|c|c}
 x&\rho(x)&\sem{\Theta[\pi_1]}(*,x)&
        \sem{\Omega}(*,x)\\\hline
 \mathsf t&h_1&h_{12}&h_{12}\Rightarrow h_1=h_1\\
 \mathsf f&h_2&h_{12}&h_{12}\Rightarrow h_2=h_2
\end{array}
\]
Indeed, existential quantification first gives
$\sem{\Theta}(*)=h_1\lor h_2=h_{12}$; weakening makes this value
constant in $x$, and implication is computed pointwise.
A second existential quantification gives
\[
 \sem{D}(*)=\bigvee_{x\in\mathbf 2}\sem{\Omega}(*,x)
 =h_1\lor h_2=h_{12}<\top.
\]
Validity of this closed formula would require a fibre arrow
$\top\to h_{12}$, that is $\top\leq h_{12}$, which is false.
\end{proof}

\begin{definition}[The iFOL embedding, \coqterm{ifol_to_epsilon_h_ty}]
\label{def:ifol-embedding}
For a fixed typed function--predicate signature $\Sigma$, the structural map
\[
 \jmath:\mathcal L\longrightarrow
 \mathcal L_{\exists\varepsilon}
\]
preserves base types, primitive functions, predicates, logical
constructors, and their proof terms. Its image contains no
labelled base constant.
\end{definition}

\begin{proposition}[The iFOL embedding preserves derivations, \coqterm{ifol_to_epsilon_derivation}]
\label{prop:ifol-embedding}
Fix a typed function--predicate signature $\Sigma$ and the structural
embedding $\jmath:\mathcal L\to\mathcal L_{\exists\varepsilon}$ that
preserves base types, primitive symbols, logical constructors, and their
proof terms. For every iFOL derivation,
\[
 I\mid\Gamma\vdash_{\mathrm{iFOL}}p:\Phi
 \quad\Longrightarrow\quad
 \jmath(I)\mid\jmath(\Gamma)
 \vdash_{\exists\varepsilon}\jmath(p):\jmath(\Phi).
 \tag{iFOL-Emb}
\]
\end{proposition}

\begin{proof}
Induction on base typing and formula formation gives preservation of
well-formed sequents. Formula conversion is preserved because its
clauses are shared by the target calculus. Induction on the iFOL
derivation then applies the corresponding rule of
$\mathrm{iFOL}_{\exists\varepsilon}$ at each step.
\end{proof}

\begin{theorem}[Dual Drinker and non-conservativity over iFOL, \coqterm{epsilonfull_not_conservative_over_ifol}]
\label{thm:dual-drinker-separation}
For the constant-free signature with a primitive sort $S$ and a unary
predicate $P$,
\[
\begin{aligned}
 1\mid\varnothing
 &\nvdash_{\mathrm{iFOL}}
 \exists x:S\,((\exists y:S\,P(y))\to P(x)),\\
 1\mid\varnothing
 &\vdash_{\exists\varepsilon}
 \exists x:S\,((\exists y:S\,P(y))\to P(x)).
\end{aligned}
\]
Consequently, $\vdash_{\exists\varepsilon}$ is not conservative over
$\vdash_{\mathrm{iFOL}}$ along $\jmath$.
(Cf.~\cite[pp.~5--6]{WarrenDienerMcKubreJordens2018}).
\end{theorem}

\begin{proof}
The second statement is Proposition~\ref{prop:dual-drinker-derivable}. For
the first, the soundness direction of
Theorem~\ref{thm:ifol-soundness-completeness} sends an iFOL
derivation to a morphism
$\top\to h_{12}$ in $P_5(1)$. Such a morphism would mean
$\top\leq h_{12}$, contradicting
Proposition~\ref{prop:h5-dual-drinker-model}.

The first displayed formula belongs to $\mathcal L$, and $\jmath$ sends
it to the second. By \textup{(iFOL-Emb)}, $\jmath$ preserves derivations,
but the epsilon-free sequent displayed above is derivable in the
extension and underivable in iFOL. This is the asserted failure of
conservativity.
\end{proof}

\subsection{Base equations between selected witnesses}
\label{sec:witness-choice-substitution}

Substitution inside the indexing formula gives
$\varepsilon_{\Phi[g_A^+]}$, with the bound coordinate unchanged
\cite[Sec.~2.1]{GieseAhrendt1999}. Substitution of the base term itself is
composition, $\varepsilon_\Phi\circ g$. In the Skolem representation
\cite[pp.~481--482]{Monk1976}, compatibility of these two operations is an
additional requirement. We consider its definitional and categorical versions.

\begin{definition}[Definitional substitutivity, \coqterm{EpsilonFullSubstitutivity.DefinitionalSubstitutivity}]
\label{def:definitional-substitutivity}
\label{not:substitution-equation}
In the base syntax of $\mathrm{iFOL}_{\exists\varepsilon}$ over $\Sigma$, let $I,J,A\in\mathsf{BType}$,
$I\vdash g:J$, and $J\times A\vdash\Phi\;\mathsf{form}$. Write
$I\times A\vdash
g_A^+:=\langle g\circ\pi_1,\pi_2\rangle:J\times A$.
\emph{Definitional substitutivity} is the requirement that
\[
 I\vdash\varepsilon_{\Phi[g_A^+]}
 \equiv\varepsilon_\Phi\circ g:A
\tag{$\mathrm{Sub}_{g,\Phi}$}
\]
hold for all such $I,J,A,g,\Phi$. Here \emph{definitional} refers to
the cartesian base-term equivalence $\equiv$ generated by identity, associativity, product, terminal, and congruence equations.
(Cf.~\cite{Molnar2013}).
\end{definition}

\begin{remark}[Comparison with Moln\'ar's substitutivity, \coqterm{EpsilonFullSubstitutivity.counterexample_equation_valid_in}]
Let $\mathcal N$ be an intensional $\varepsilon$-structure with universe
$D$ and variable set $V$. Let $x,y\in V$ be distinct, $\varphi$ a
formula, and $t$ a term free for $y$ in $\varepsilon x\,\varphi$.
For an assignment $\alpha\in D^V$, define
\[
 e_\varphi(\alpha):=(\varepsilon x\,\varphi)^{\mathcal N}(\alpha),
 \qquad
 g_t(\alpha):=\alpha[y\mapsto t^{\mathcal N}(\alpha)].
\]
Thus $e_\varphi:D^V\to D$ and $g_t:D^V\to D^V$.
Moln\'ar's substitutivity condition, also formulated by Giese and Ahrendt, is

\[
 e_{\varphi[t/y]}(\alpha)
 =e_\varphi\bigl(\alpha[y\mapsto t^{\mathcal N}(\alpha)]\bigr)
 =(e_\varphi\circ g_t)(\alpha).
\]
For the data of Definition~\ref{def:definitional-substitutivity} and a thin
$\varepsilon$-hyperdoctrine $\mathcal M$ interpreting the fixed language,
the two requirements have the forms
\[
\begin{aligned}
 e_{\varphi[t/y]}
 &=e_\varphi\circ g_t
 &&\text{(Moln\'ar)},\\
 \sem{\varepsilon_{\Phi[g_A^+]}}_{\mathcal M}
 &\sim_{\mathcal B_{\mathcal M}}
 \sem{\varepsilon_\Phi}_{\mathcal M}\circ\sem{g}_{\mathcal M}
 &&\text{(categorical)}.
\end{aligned}
\]
Both sides of the second line are base morphisms
$\sem{I}_{\mathcal M}\to\sem{A}_{\mathcal M}$.
The map $g_t$ changes one free variable, while $\sem{g}$ changes the free
context. The lift $g_A^+$ leaves the bound variable unchanged.
(Cf.~\cite[Def.~2.11]{Molnar2013}, \cite[Def.~5]{GieseAhrendt1999}).
\end{remark}

\begin{remark}[Formula substitution and the witness equation, \coqterm{EpsilonFull.Hyper_Syntax.FE_Eps_BC}]
The $\varepsilon$-Beck--Chevalley conversion gives
$\Phi[g_A^+][\varepsilon_{\Phi[g_A^+]}]
\simeq_I\Phi[\varepsilon_\Phi][g]$.
It compares the two selected formula instances, whereas
$\mathrm{Sub}_{g,\Phi}$ compares their base terms. We study one instance of
the latter equation without adjoining the definitional substitutivity
schema to the calculus.
(Cf.~\cite{Molnar2013}).
\end{remark}

\begin{definition}[The separating instance, \coqterm{EpsilonFullSubstitutivity.substitution_counterexample_formula}, \coqterm{EpsilonFullSubstitutivity.witness_of_substituted_formula}, \coqterm{EpsilonFullSubstitutivity.substituted_witness}]
\label{def:separating-instance}
In $\mathrm{iFOL}_{\exists\varepsilon}$, take no primitive function symbols and one predicate symbol $R$ with
$\operatorname{ar}(R)=S$, and use its language. Define
$\mathbf 1\vdash g:=c^S_{\langle S,\top\rangle}:S$ and
$\Phi:=R(\pi_1)$ over $S\times S$. The lift of $g$ is
$\mathbf 1\times S\vdash
g_S^+:=\langle g\circ\pi_1,\pi_2\rangle:S\times S$.
The two terms under comparison are
$\varepsilon_{\Phi[g_S^+]}$ and $\varepsilon_\Phi\circ g$, both of
type $\mathbf 1\to S$.
\end{definition}

\begin{theorem}[Independence of a definitional substitution equation, \coqterm{EpsilonFullSubstitutivity.separating_model_refutes_substitutivity}, \coqterm{EpsilonFullSubstitutivity.uniform_model_validates_substitutivity}]
\label{thm:witness-substitution-separation}
Take the $\mathrm{iFOL}_{\exists\varepsilon}$-signature with no primitive functions and one unary predicate $R$, put $\mathbf 1\vdash g=c^S_{\langle S,\top\rangle}:S$, $\Phi=R(\pi_1)$ over $S\times S$, and $g_S^+=\langle g\circ\pi_1,\pi_2\rangle$. There are
thin $\varepsilon$-hyperdoctrines $\mathcal M_{\ne}$ and
$\mathcal M_{=}$ over the same cartesian base, with the same constant
$H_3$-fibres and atomic interpretation, such that
\[
\begin{aligned}
 \sem{\varepsilon_{\Phi[g_S^+]}}_{\mathcal M_{\ne}}
 &\not\sim
 \sem{\varepsilon_\Phi}_{\mathcal M_{\ne}}\circ
 \sem{g}_{\mathcal M_{\ne}},\\
 \sem{\varepsilon_{\Phi[g_S^+]}}_{\mathcal M_{=}}
 &\sim
 \sem{\varepsilon_\Phi}_{\mathcal M_{=}}\circ
 \sem{g}_{\mathcal M_{=}}.
\end{aligned}
\]
Moreover, the corresponding base equation is not derivable:
\[
 \mathbf 1\nvdash\varepsilon_{\Phi[g_S^+]}
 \equiv\varepsilon_\Phi\circ g:S.
\]
(Cf.~\cite{Molnar2013}).
\end{theorem}

\begin{proof}
Use the cartesian base $\mathcal B_2$: its objects are base types,
$|\mathbf 1|=\{*\}$, $|S|=\mathbf 2$, $|A\times B|=|A|\times|B|$,
and its arrows are all functions between these carriers, with pointwise
equality. Define default elements by
$d_{\mathbf 1}:=*$, $d_S:=\mathsf f$, and
$d_{A\times B}:=(d_A,d_B)$.

Let $k_*:=\langle S,\Phi[g_S^+]\rangle$. In $\mathcal M_{\ne}$, interpret
each constant $I\vdash c_k^A:A$ by the constant function with value
\[
 \sem{c_k^A}_{\mathcal M_{\ne}}(x):=
 \begin{cases}
  \mathsf t,& A=S\text{ and }k=k_*,\\
  d_A,&\text{otherwise},
 \end{cases}
 \qquad x\in|I|.
\]
In $\mathcal M_{=}$, interpret it by
$\sem{c_k^A}_{\mathcal M_{=}}(x):=d_A$ for every $k$.
These interpretations extend to base terms by the clauses of
Definition~\ref{def:base-model} and respect base-term equivalence.

For either interpretation, take $P(I):=H_3=\{\bot<m<\top\}$,
$P(f):=\mathrm{id}_{H_3}$, and
$\Sigma_\pi=\Pi_\pi:=\mathrm{id}_{H_3}$. Interpret every raw atom by $m$.
The relevant fibres are
\[
\begin{tikzcd}[row sep=2.8em,column sep=6em]
 P(S\times S)=H_3
   \arrow[r,"P(\sem{g_S^+})=\mathrm{id}"]
   \arrow[d,phantom,"\scriptstyle\textup{over}" description]
 & P(\mathbf 1\times S)=H_3
   \arrow[d,phantom,"\scriptstyle\textup{over}" description]\\
 S\times S
 & \mathbf 1\times S \arrow[l,"\sem{g_S^+}"']
\end{tikzcd}
\]
Since every fibre substitution in these two models is the identity,
induction on a raw formula $\Theta$ gives
$\sem{\Theta[f]}=\sem{\Theta}$ for every base term $f$.
Consequently,
\[
 \Sigma_{\pi_{I,A}}(\sem{\Theta})
 =\sem{\Theta}
 =\sem{\Theta[\varepsilon_\Theta]}.
\]
Thus both choices of constant interpretation satisfy
$(\varepsilon\Sigma)$ and all formula-conversion laws.

The raw formulae $\top$, $\Phi=R(\pi_1)$, and
$\Phi[g_S^+]=R(\pi_1\circ g_S^+)$ have distinct labels.
Therefore
\[
\begin{aligned}
 \sem{g}_{\mathcal M_{\ne}}(*)&=\mathsf f,
 &\sem{\varepsilon_\Phi}_{\mathcal M_{\ne}}(b)&=\mathsf f
       &&(b\in\mathbf 2),\\
 \sem{\varepsilon_{\Phi[g_S^+]}}_{\mathcal M_{\ne}}(*)&=\mathsf t,
 &(\sem{\varepsilon_\Phi}_{\mathcal M_{\ne}}\circ
   \sem{g}_{\mathcal M_{\ne}})(*)&=\mathsf f.
\end{aligned}
\]
Thus the two interpretations in $\mathcal M_{\ne}$ differ because
$\mathsf t\ne\mathsf f$.
In $\mathcal M_{=}$, both functions $\mathbf 1\to S$ have value
$\mathsf f$, proving the second equation. Finally, a derivation of
$\mathbf 1\vdash\varepsilon_{\Phi[g_S^+]}\equiv
\varepsilon_\Phi\circ g:S$ would give
$\mathsf t=\mathsf f$ in $\mathcal M_{\ne}$ by base soundness,
Theorem~\ref{thm:base-soundness-completeness}\textup{(eq)}.
\end{proof}

\begin{corollary}[Definitional $\varepsilon$-extensionality is not forced, \coqterm{EpsilonFullSubstitutivity.convertible_witnesses_not_definitionally_equal}]
\label{cor:epsilon-extensionality-independent}
In $\mathrm{iFOL}_{\exists\varepsilon}$ with no primitive functions and one unary predicate $R$, set $g=c^S_{\langle S,\top\rangle}:\mathbf 1\to S$, $\Phi=R(\pi_1)$ over $S\times S$, and $g_S^+=\langle g\circ\pi_1,\pi_2\rangle$. Let
\[
 \Theta:=\Phi[g_S^+],
 \qquad
 \Psi:=\Theta[\mathrm{id}_{\mathbf 1\times S}].
\]
Then $\Theta,\Psi$ are well formed over $\mathbf 1\times S$ and
\[
 \Theta\simeq_{\mathbf 1\times S}\Psi,
 \qquad
 \mathbf 1\nvdash\varepsilon_\Theta\equiv\varepsilon_\Psi:S.
\]
In $\mathcal M_{\ne}$ their interpretations are distinct, whereas in
$\mathcal M_{=}$ they are equal.
(Cf.~\cite{Molnar2013}).
\end{corollary}

\begin{proof}
The identity-substitution conversion gives
$\Psi\simeq_{\mathbf 1\times S}\Theta$. Symmetry gives the displayed
conversion. As raw formulae,
\[
 \Theta=R(\pi_1\circ g_S^+),
 \qquad
 \Psi=R\bigl((\pi_1\circ g_S^+)\circ
              \mathrm{id}_{\mathbf 1\times S}\bigr),
\]
so their labels are distinct. The exceptional label of
$\mathcal M_{\ne}$ is $\langle S,\Theta\rangle$. Hence
\[
 \sem{\varepsilon_\Theta}_{\mathcal M_{\ne}}(*)=\mathsf t,
 \qquad
 \sem{\varepsilon_\Psi}_{\mathcal M_{\ne}}(*)=\mathsf f.
\]
Base soundness gives the asserted syntactic inequivalence. In
$\mathcal M_{=}$ both terms have the constant value $\mathsf f$.
\end{proof}

\begin{remark}[Scope of the separation, \coqterm{EpsilonFullSubstitutivity.uniform_model_validates_substitutivity}, \coqterm{EpsilonFullSubstitutivity.separating_model_refutes_substitutivity}]
The two models show that the present semantic axioms neither force nor
exclude the concrete equation $\mathrm{Sub}_{g,\Phi}$. This is the sense of
independence used here. The positive part concerns that equation alone. We
do not construct a calculus satisfying the whole definitional
substitutivity schema. For $I\vdash a,b:A$, the assertion
$I\nvdash a\equiv b:A$ is metatheoretic. An expression
$I\vdash a\not\equiv b:A$ is not a judgement of this calculus.
The negative statement about the witnesses themselves is semantic:
their interpretations are unequal in $\mathcal M_{\ne}$.
\end{remark}

\begin{remark}[What the fibres identify, \coqterm{EpsilonFullSubstitutivity.fin_H3_interp_h_ty_subst_const}]
In both models, the two selected instances have value $m$:
\[
 \sem{\Phi[g_S^+][\varepsilon_{\Phi[g_S^+]}]}
 =m
 =\sem{\Phi[\varepsilon_\Phi][g]}.
\]
Thus the formula-level comparison holds even in $\mathcal M_{\ne}$,
where the selected witnesses differ. Corollary~\ref{cor:epsilon-extensionality-independent}
shows moreover that converted formulae can have unequal selected base terms.
For its pair $\Theta\simeq_{\mathbf 1\times S}\Psi$, formula conversion
still gives
$\Theta[\varepsilon_\Theta]
\simeq_{\mathbf 1}\Psi[\varepsilon_\Psi]$.
\end{remark}

\begin{remark}[Internal equality]
Neither separation uses an equality predicate. Suppose one later adds a
predicate $\mathsf{Eq}_A$ of arity $A\times A$ and the rule
\[
 I\vdash a\equiv b:A
 \quad\Longrightarrow\quad
 I\mid\varnothing\vdash_{\exists\varepsilon}
 \mathsf{Eq}_A(\langle a,b\rangle).
\]
Here $I\vdash a,b:A$. In such an extension, definitional substitutivity
would imply its propositional version, for the data of
Definition~\ref{def:definitional-substitutivity},
\[
 I\mid\varnothing\vdash_{\exists\varepsilon}
 \mathsf{Eq}_A\bigl(\langle\varepsilon_{\Phi[g_A^+]},
                          \varepsilon_\Phi\circ g\rangle\bigr).
\]
The converse need not hold. In particular, a countermodel to a base equation
does not by itself refute the corresponding internal equality formula. That
conclusion would require an interpretation reflecting internal equality
into equality of base morphisms.
(Cf.~\cite[p.~226]{Jacobs1999}, \cite{Molnar2013}).
\end{remark}
\section{Formal scope}
\label{sec:initiality-formal-scope}
\label{sec:formalization-coq}

The constructors of \texttt{h\_derivation} take values in \texttt{Type};
logical derivability only asserts inhabitation. Likewise, a chosen
selector or isomorphism is data, while its existence is a proposition.
The model-class construction starts with chosen unit arrows and does not
extract data from propositional existence.

\begin{remark}[Formal assumptions, \coqterm{EpsilonFull_Completeness.epsilon_adjoint_completeness_inhabitational_gen}, \coqterm{EpsilonFull_ModelCategory.TermModel_initial_unique}]
The semantic computations underlying
Theorem~\ref{thm:witness-substitution-separation} and
Corollary~\ref{cor:epsilon-extensionality-independent} use no additional
Rocq axioms. The current proofs of coherence for the two models, and hence
the passage from these semantic refutations to syntactic base-term
inequivalence, depend on
\texttt{Eqdep.Eq\_rect\_eq.eq\_rect\_eq}.
For $\mathrm{iFOL}_{\exists\varepsilon}$, the term-model interpretation identity, ordinary and all-predicate completeness, and initiality for interpretation-preserving maps have the same dependency. The $\mathrm{iFOL}_\varepsilon$ and iFOL completeness theorems do
not. The $H_3$ and $H_5$ countermodel constructions are closed under
the global Rocq context and do not use this dependency.
This is an assumption about dependent equality in the metatheory,
not an object-language witness equation.
\end{remark}

\begin{remark}[Proof terms and further comparisons, \coqterm{EpsilonFull_Completeness.fiber_equiv}]
\label{sec:conclusion-future-work}
The logical fibres identify parallel morphisms, although the syntax
retains proof terms. A proof-relevant version would require an explicit
equality or reduction theory on those terms. Likewise, a comparison with
a calculus with definitional substitutivity requires adjoining every equation
$\mathrm{Sub}_{g,\Phi}$ and studying the corresponding model class. The
present separation concerns a concrete instance and does not make that
extension.
\end{remark}

\backmatter

\section*{Statements and Declarations}

\bmhead{Author contributions}
Zolt\'an G\'abor Moln\'ar contributed the logical development;
Csaba T\'oth contributed the algebraic development.

\bmhead{Data availability}
This is a theoretical study and uses no empirical datasets.
The Rocq formalization is described in the code availability statement.

\bmhead{Code availability}
The HyperLogic repository is at
\url{https://github.com/mozow01/HyperLogic}.
The formalization is checked with Coq 8.18.0. The supported source files
are listed in \texttt{\_CoqProject}; the repository's README gives the
build instructions.

\begin{appendices}
\section{The epsilon-only calculus}
\label{app:epsilon-only}

This appendix gives $\mathrm{iFOL}_\varepsilon$, whose existential
quantification is expressed by epsilon substitution. Its syntax uses its
own recursively nested labels. We give its derivation rules, models,
term model, and soundness and completeness theorem.

\subsection{Syntax and derivations}

\begin{definition}[The $\mathrm{iFOL}_\varepsilon$-language, \coqterm{Epsilon.Hyper_Syntax.h_ty}, \coqterm{CanonicalEpsilon.EpsilonLabel}]
\label{def:ifol-epsilon-language}
Fix a typed function--predicate signature $\Sigma$. The raw terms,
formulae, and labels of $\mathcal L_\varepsilon$ are finite trees given
simultaneously by
\[
\begin{aligned}
 t&::=\mathrm{id}\mid g\circ f\mid {!}\mid\langle f,g\rangle
       \mid\pi_1\mid\pi_2\mid\widehat u\mid c_s^A,\\
 \Phi&::=\top\mid\bot\mid R(t)\mid\Phi\land\Psi\mid\Phi\lor\Psi
       \mid\Phi\to\Psi\mid\forall_A\Phi,\\
 s&::=\langle A,\Phi\rangle.
\end{aligned}
\]
Here $A$ is a base type, $u\in\mathcal F$, $R\in\mathcal R$, and
$\varepsilon_\Phi^A:=c^A_{\langle A,\Phi\rangle}$.
The formulae stored in labels, including nested labels, are also
$\exists$-free. Its typing and formation rules are the cartesian base
rules and the displayed connective and universal formation rules.
\end{definition}

\begin{definition}[The epsilon-only calculus, \coqterm{Epsilon.Hyper_Syntax.h_derivation}]
\label{def:epsilon-only-calculus}
Fix a typed signature $\Sigma$. The calculus $\mathrm{iFOL}_\varepsilon$ over $\mathcal L_\varepsilon$
has assumption, truth, falsity, implication, conjunction, disjunction,
universal introduction and elimination, and formula and context conversion
rules, together with
$\mathsf{epsI}$ and $\mathsf{epsE}$. Its formula conversion is the least equivalence on well-formed formulae
closed under atomic base equivalence, propositional and universal
congruence, identity, composition, and congruence of substitution, the
selected-instance substitution law $(\varepsilon\mathrm{BC})$, and
\[
 \Phi\simeq_{I\times A}\Psi
 \quad\Longrightarrow\quad
 \Phi[\varepsilon_\Phi]\simeq_I\Psi[\varepsilon_\Psi].
\tag{$\varepsilon$-funct}
\]
It has no existential formula constructor, existential proof rule, or
$\textup{(}\varepsilon\Sigma\textup{)}$ conversion.
The resulting derivability judgement is written
$I\mid\Gamma\vdash_\varepsilon p:\Phi$.

We omit the proof term by writing
\[
 I\mid\Gamma\vdash_\varepsilon\Phi
 \quad\Longleftrightarrow\quad
 \exists p\;I\mid\Gamma\vdash_\varepsilon p:\Phi.
\]
(Cf.~\cite{Prawitz1965}).
\end{definition}

\begin{remark}[$\varepsilon$-functoriality without $\exists$, \coqterm{Epsilon.Hyper_Syntax.FE_Eps_Functorial}]
In $\vdash_{\exists\varepsilon}$, $\varepsilon$-functoriality follows from
existential congruence and $\textup{(}\varepsilon\Sigma\textup{)}$. This
argument is unavailable in $\vdash_\varepsilon$, where
$\textup{(}\varepsilon\text{-funct}\textup{)}$ is therefore a conversion
clause. It is the constructor \texttt{FE\_Eps\_Functorial} in Rocq.
\end{remark}

\subsection{Models and the selection condition}

\begin{definition}[Models of the epsilon-only calculus, \coqterm{EpsilonGoodFibration}]
\label{def:epsilon-only-model}
Fix a typed signature $\Sigma$. A thin $\varepsilon$-hyperdoctrine for $\mathrm{iFOL}_\varepsilon$ consists
of a cartesian base model for $\mathcal L_\varepsilon$, a thin
first-order hyperdoctrine $P$ over its base, and raw atomic data satisfying
substitution and base congruence. The base model also interprets its
labelled constants $c_s^A$ with $s$ an epsilon-only label, and its
interpretation is independent of typing derivations up to
$\sim_{\mathcal B}$. Interpret truth, falsity, conjunction, disjunction,
and implication by fibre terminal objects, initial objects, products,
coproducts, and exponentials, and interpret $\forall_A$ by
$\Pi_{\pi_{I,A}}$. Interpret logical contexts by fibre products.
For
$I\times A\vdash\Phi\;\mathsf{form}$ and $I\vdash f:A$, it is equipped with
arrows
\[
 u_\Phi:\sem{\Phi}_{I\times A}
 \longrightarrow
 \sem{\Phi[\varepsilon_\Phi][\pi_1]}_{I\times A},
 \qquad
 i_{\Phi,f}:\sem{\Phi[f]}_I
 \longrightarrow\sem{\Phi[\varepsilon_\Phi]}_I,
\tag{$\varepsilon$-I}
\]
and, for every $I\vdash\Gamma\;\mathsf{ctx}$ and
$I\vdash\Theta\;\mathsf{form}$, an operation
\[
\begin{aligned}
 &\operatorname{Hom}_{P(\sem{I\times A})}
 \bigl(\sem{\Gamma[\pi_1]}_{I\times A}\times\sem{\Phi}_{I\times A},
       \sem{\Theta[\pi_1]}_{I\times A}\bigr)\\
 &\hspace{5em}\longrightarrow
 \operatorname{Hom}_{P(\sem{I})}
 \bigl(\sem{\Gamma}_I\times\sem{\Phi[\varepsilon_\Phi]}_I,
       \sem{\Theta}_I\bigr).
\end{aligned}
\tag{$\varepsilon$-E}
\]
For $I\times A\vdash\Phi\;\mathsf{form}$ and
$I\times A\vdash\Psi\;\mathsf{form}$, it satisfies
$\varepsilon$-functoriality:
\[
 \sem{\Phi}_{I\times A}\cong\sem{\Psi}_{I\times A}
 \Longrightarrow
 \sem{\Phi[\varepsilon_\Phi]}_I
 \cong\sem{\Psi[\varepsilon_\Psi]}_I.
\tag{$\varepsilon$-funct$_\mathrm{s}$}
\]
For $I\vdash g:J$ and $J\times A\vdash\Phi\;\mathsf{form}$, it satisfies
$\varepsilon$-Beck--Chevalley:
\[
 \sem{\Phi[\varepsilon_\Phi][g]}_I
 \cong
 \sem{\Phi[g_A^+][\varepsilon_{\Phi[g_A^+]}]}_I.
\tag{$\varepsilon\mathrm{BC}_\mathrm{s}$}
\]
(Cf.~\cite[pp.~414--415]{Pasquali2016}).
\end{definition}

\begin{proposition}[Semantic substitution in the epsilon-only language, \coqterm{Epsilon_Soundness.form_subst_iso}]
\label{prop:epsilon-only-semantic-substitution}
Fix a typed signature $\Sigma$, a base model of $\mathcal L_\varepsilon$
whose interpretation is independent of typing derivations, a thin
first-order hyperdoctrine $P$ over its base, and raw atomic data satisfying
\textup{(At-subst)} and \textup{(At-cong)}. If $I\vdash f:J$ and $\Phi$
is a raw formula of $\mathcal L_\varepsilon$, then
\[
 \sem{\Phi[f]}_I\cong P(\sem{f})(\sem{\Phi}_J).
\]
\end{proposition}

\begin{proof}
Induct on $\Phi$. Atomic substitution gives the atomic case; reindexing
preserves the bicartesian closed operations. The universal case uses
Beck--Chevalley and the lifted base substitution.
\end{proof}

\begin{corollary}[The model class of $\mathrm{iFOL}_\varepsilon$, \coqterm{Epsilon_SelectionModels.EpsilonGoodFibration_from_chosen_units}]
\label{cor:epsilon-only-model-class}
Fix a typed signature $\Sigma$, a base model
$(\mathcal B,S_{\mathcal B},\sem{-})$ of $\mathcal L_\varepsilon$ whose
interpretation is independent of typing derivations, a thin first-order
hyperdoctrine $P$ over $\mathcal B$, and raw atomic data satisfying
substitution and base congruence. For every well-formed
$I\times A\vdash\Phi\;\mathsf{form}$, put
$Q_\Phi=\sem{\Phi}$,
$s_\Phi=\langle\mathrm{id},\sem{\varepsilon_\Phi}\rangle$, and
$E_\Phi=P(s_\Phi)Q_\Phi\cong\sem{\Phi[\varepsilon_\Phi]}$.
A chosen family of unit arrows
$Q_\Phi\to P(\pi_{I,A})E_\Phi$ supplies all five laws of an
$\mathrm{iFOL}_\varepsilon$-model: unit, introduction, contextual
elimination, functoriality, and Beck--Chevalley. Conversely, every such
model supplies the unit arrows. These conditions range over represented
well-formed formulae of $\mathcal L_\varepsilon$; no equality or naturality
of the witnesses $\sem{\varepsilon_\Phi}$ is required.
\end{corollary}

\begin{proof}
Reindex the unit along
$\langle\mathrm{id},\sem{f}\rangle$ to obtain introduction for any
$I\vdash f:A$. Reindex an elimination premise along $s_\Phi$;
preservation of products and $\pi s_\Phi\sim\mathrm{id}$ give the
contextual elimination arrow. The chosen unit gives $\Sigma_\pi Q_\Phi\cong E_\Phi$ by
the selection characterization. This isomorphism gives functoriality through $\Sigma_\pi$.
The $\Sigma$ Beck--Chevalley isomorphism and semantic substitution give
\[
 P(\sem{g})E_\Phi
 \cong P(\sem{g})\Sigma_\pi Q_\Phi
 \cong\Sigma_{\pi'}Q_{\Phi[g_A^+]}
 \cong E_{\Phi[g_A^+]}
\]
for every $I\vdash g:J$ and $J\times A\vdash\Phi\;\mathsf{form}$,
where $\pi=\pi_{J,A}$ and $\pi'=\pi_{I,A}$ are the projections over $\sem{J}$ and $\sem{I}$, respectively.
Well-formed formulae are closed under these typed substitutions.
The converse simply takes the unit component of the model.
\end{proof}

\begin{remark}[The scope of the model characterization, \coqterm{CanonicalEpsilonComparison.EpsilonSign}, \coqterm{CanonicalEpsilonComparison.EpsilonFullSign}]
The two epsilon calculi have different raw formula and label types.
The preceding characterization describes the required structure on an
interpretation of each language. It supplies no equivalence between their
model categories.
\end{remark}

\begin{notation}[Validity for the epsilon-only calculus, \coqterm{epsilon_completeness_inhabitational_gen}]
Fix a typed signature $\Sigma$ and a thin $\varepsilon$-hyperdoctrine for $\mathrm{iFOL}_\varepsilon$ $\mathcal M$
interpreting $\mathcal L_\varepsilon$. For a well-formed context
$I\vdash\Gamma\;\mathsf{ctx}$ and formula
$I\vdash\Phi\;\mathsf{form}$ in this language, write
\[
 I\mid\Gamma\models_{\varepsilon}^{\mathcal M}\Phi
 \quad\Longleftrightarrow\quad
 \operatorname{Hom}_{P_{\mathcal M}(\sem{I}_{\mathcal M})}
 \bigl(\sem{\Gamma}_{\mathcal M},\sem{\Phi}_{\mathcal M}\bigr)
 \text{ is inhabited}.
\]
\end{notation}

\begin{proposition}[Selected-instance adjunction for the epsilon-only calculus, \coqterm{pasquali_has_enough_epsilon_terms_semantic}]
\label{prop:epsilon-only-selected-instance-adjunction}
Let $\mathcal M$ be a thin $\varepsilon$-hyperdoctrine
for $\mathrm{iFOL}_\varepsilon$ over $\Sigma$. For
$I\times A\vdash\Phi\;\mathsf{form}$ and
$I\vdash\Psi\;\mathsf{form}$,
\[
\begin{aligned}
 &\operatorname{Hom}_{P_{\mathcal M}(\sem{I})}
   \bigl(\sem{\Phi[\varepsilon_\Phi]}_I,\sem{\Psi}_I\bigr)
   \text{ is inhabited}\\
 &\qquad\Longleftrightarrow\quad
 \operatorname{Hom}_{P_{\mathcal M}(\sem{I\times A})}
   \bigl(\sem{\Phi}_{I\times A},\sem{\Psi[\pi_1]}_{I\times A}\bigr)
   \text{ is inhabited}.
\end{aligned}
\]
(Cf.~\cite[pp.~414--415]{Pasquali2016}, \cite[p.~452]{Pasquali2019}).
\end{proposition}

\begin{proof}
For $\mathrm{iFOL}_\varepsilon$, the model supplies a selection isomorphism
$\Sigma_\pi\sem{\Phi}\cong\sem{\Phi[\varepsilon_\Phi]}$.
Apply the adjunction $\Sigma_\pi\dashv P(\pi)$ to $\sem{\Psi}$,
and use $P(\pi)\sem{\Psi}\cong\sem{\Psi[\pi_1]}$.
\end{proof}

\subsection{The term model, soundness, and completeness}

\begin{proposition}[The term model of $\mathrm{iFOL}_\varepsilon$, \coqterm{Epsilon_Completeness.TermModel_EpsilonGoodFibration}]
\label{prop:ifol-epsilon-term-model}
Fix a typed function--predicate signature $\Sigma$. For $\mathrm{iFOL}_\varepsilon$, take its base types and typed terms as the syntactic base, well-formed formulae as fibre objects, and closed implication proofs as fibre arrows; identify all parallel fibre arrows. Use syntactic substitution for reindexing and interpret well-typed syntax by itself, with ill-typed raw atoms sent to $\top$.
Write $\mathcal T_\varepsilon$ for the resulting model and $P_\varepsilon$ for its
fibration. Then $\mathcal T_\varepsilon$ is a thin $\varepsilon$-hyperdoctrine for $\mathrm{iFOL}_\varepsilon$.
For base types $I,A$ and a well-formed formula
$I\times A\vdash\Phi\;\mathsf{form}$ of $\mathcal L_\varepsilon$, with
$\pi_{I,A}:I\times A\to I$, its quantifiers satisfy
\[
 \Sigma_{\pi_{I,A}}(\Phi):=\Phi[\varepsilon_\Phi],
 \qquad \Pi_{\pi_{I,A}}(\Phi):=\forall_A\Phi,
 \qquad \Sigma_{\pi_{I,A}}\dashv P_\varepsilon(\pi_{I,A})\dashv\Pi_{\pi_{I,A}}.
\]
Base types and well-formed formulae interpret as themselves, and
$\sem{\Gamma}_{\mathcal T_\varepsilon}=\Gamma^\times$ for every well-formed context.
Here the context formula is defined by
$\varnothing^\times=\top$ and $(\Phi,\Gamma)^\times=\Gamma^\times\land\Phi$.
(Cf.~\cite[pp.~234--235]{Jacobs1999}).
\end{proposition}

\begin{proof}
The shared rules give the cartesian base, bicartesian closed fibres,
syntactic substitution, and universal quantification.
For $\mathrm{iFOL}_\varepsilon$, its unit is represented by
\[
 I\times A\mid\varnothing\vdash_\varepsilon
 \lambda u.\,\mathsf{epsI}(\pi_2,u):
 \Phi\to\Phi[\varepsilon_\Phi][\pi_1].
\]
For $I\vdash\Psi\;\mathsf{form}$ and
$I\times A\mid\varnothing\vdash_\varepsilon
p:\Phi\to\Psi[\pi_1]$, its adjunct transpose is represented by
\[
 I\mid\varnothing\vdash_\varepsilon
 \lambda z.\,\mathsf{epsE}(z,u.\,p\,u):
 \Phi[\varepsilon_\Phi]\to\Psi.
\]
The unit and transposition exhibit the left adjoint, and thinness gives
its adjunction equations. General $\varepsilon$-introduction and
context-parametrized $\varepsilon$-elimination supply the operations of
Definition~\ref{def:epsilon-only-model}.
The clauses $(\varepsilon\text{-funct})$ and $(\varepsilon\mathrm{BC})$
give the comparison isomorphisms. The interpretation identities follow by induction on types, well-formed formulae, and contexts.
\end{proof}

\begin{theorem}[Soundness and completeness of $\mathrm{iFOL}_\varepsilon$, \coqterm{CanonicalEpsilonSemantics.soundness}, \coqterm{CanonicalEpsilonSemantics.completeness}]
\label{thm:epsilon-only-soundness-completeness}
Fix a typed signature $\Sigma$, a well-formed context
$I\vdash\Gamma\;\mathsf{ctx}$ and a well-formed formula
$I\vdash\Phi\;\mathsf{form}$ in $\mathcal L_\varepsilon$. Then
\[
 I\mid\Gamma\vdash_{\varepsilon}\Phi
 \quad\Longleftrightarrow\quad
 \bigl(\forall\mathcal M,\quad
 I\mid\Gamma\models_{\varepsilon}^{\mathcal M}\Phi\bigr),
\]
where $\mathcal M$ ranges over all thin $\varepsilon$-hyperdoctrines for $\mathrm{iFOL}_\varepsilon$ interpreting $\mathcal L_\varepsilon$ over $\Sigma$.
(Cf.~\cite[pp.~234--235, 247--248]{Jacobs1999}).
\end{theorem}

\begin{proof}
Soundness is by induction on derivations. For $\vdash_\varepsilon$, the two
$\varepsilon$-rules are interpreted by $\textup{(}\varepsilon\text{-I)}$
and $\textup{(}\varepsilon\text{-E)}$, while the two additional conversion
clauses use $\textup{(}\varepsilon\text{-funct}_{\mathrm{s}}\textup{)}$ and
$\textup{(}\varepsilon\mathrm{BC}_{\mathrm{s}}\textup{)}$. The universal rules are interpreted by $P(\pi)\dashv\Pi_\pi$.

For completeness, by
Proposition~\ref{prop:ifol-epsilon-term-model}, the semantic assumption may be applied
to $\mathcal T_\varepsilon$. Writing $P_{\mathcal T_\varepsilon}$ for its fibration, it
gives
$\operatorname{Hom}_{P_{\mathcal T_\varepsilon}(I)}
(\Gamma^\times,\Phi)$ inhabited.
Proposition~\ref{prop:term-model-extraction} yields
$I\mid\Gamma\vdash_\varepsilon\Phi$.
\end{proof}

\end{appendices}

\bibliography{references}

\end{document}